%% file: MinkCA.tex
\documentclass[draft,table]{asl}

\usepackage[table]{xcolor}

\newcommand{\semph}[1]{\emph{\textbf{#1}}} 
\newcommand{\lemph}[1]{\emph{#1}}

\usepackage{arydshln} 
\usepackage{multirow} 

\usepackage{amssymb}

\usepackage{tikz}
\usetikzlibrary{shapes}

\usepackage{url}

\usepackage{wasysym} 
\usepackage{xspace} 

\theoremstyle{definition}
\newtheorem{theorem}{Theorem}
\newtheorem{lemma}{Lemma}
\newtheorem{prop}{Proposition}
\newtheorem{remark}{Remark}
\newtheorem{corollary}{Corollary}

\newcommand{\comp}{\mathop{;}}
\newcommand{\Ccone}{\mathsf{C}}

\newcommand{\Q}{\mathit{Q}}

\newcommand{\vektoro}[1]{\mathsf{\bar#1}}
\newcommand{\nombro}[1]{\mathsf{#1}}
\newcommand{\na}{\nombro{a}}
\newcommand{\nb}{\nombro{b}}
\newcommand{\np}{\nombro{p}}
\newcommand{\nq}{\nombro{q}}
\newcommand{\nr}{\nombro{r}}
\newcommand{\nt}{\nombro{t}}
\newcommand{\nx}{\nombro{x}}
\newcommand{\ny}{\nombro{y}}
\newcommand{\nv}{\nombro{v}}
\newcommand{\vo}{\vektoro{o}}
\newcommand{\vp}{\vektoro{p}}
\newcommand{\vq}{\vektoro{q}}
\newcommand{\vr}{\vektoro{r}}
\newcommand{\vs}{\vektoro{s}}
\newcommand{\vx}{\vektoro{x}}
\newcommand{\vy}{\vektoro{y}}
\newcommand{\vz}{\vektoro{z}}
\newcommand{\vv}{\vektoro{v}}
\newcommand{\vu}{\vektoro{u}}

\newcommand{\etalambda}{{\scriptstyle \lambda}}
\newcommand{\etetalambda}{{\scriptscriptstyle \lambda}}
\newcommand{\tinyneq}{{\scriptscriptstyle\neq}}
\newcommand{\etacup}{{\scriptstyle\cup}}
\newcommand{\aleq}{\mathrel{\rho}}
\newcommand{\alneq}{\mathrel{\bar \rho}}
\newcommand{\bleq}{\mathrel{\delta}}
\newcommand{\tleq}{\mathrel{\tau}}
\newcommand{\lleq}{\mathrel{\etalambda}}
\newcommand{\sleq}{\mathrel{\sigma}}
\newcommand{\tlneq}{\mathrel{\bar \tau}}
\newcommand{\llneq}{\mathrel{\bar \etalambda}}
\newcommand{\slneq}{\mathrel{\bar \sigma}}
\newcommand{\defeq}{\stackrel{\textsf{\tiny def}}{=}}
\newcommand{\defiff}{\stackrel{\textsf{\tiny def}}{\iff}}
\newcommand{\ie}{i.e., }
\newcommand{\eg}{e.g., }
\newcommand{\cf}{cf.\ }
\newcommand{\etc}{etc.\ }
\newcommand{\converse}{{}^{-1}} 
\newcommand{\Id}{=}
\newcommand{\komplementu}[1]{\bar{#1}}
\newcommand{\komplemento}{\komplementu{\phantom{n}}}
\newcommand{\ttos}{\Psi_{\tau\to\sigma}}
\newcommand{\ttol}{\Psi_{\tau\to \etetalambda}}
\newcommand{\ltot}{\Psi_{\etetalambda\to \tau}}
\newcommand{\ltos}{\Psi_{\etetalambda\to \sigma}}
\newcommand{\stot}{\Psi_{\sigma\to \tau}}
\newcommand{\stol}{\Psi_{\sigma\to \etetalambda}}
\newcommand{\newettos}{\hat\mathcal{E}_{\tau\to \sigma}}
\newcommand{\newestot}{\hat\mathcal{E}_{\sigma\to \tau}}
\newcommand{\ettos}{\mathcal{E}_{\tau\to \sigma}}
\newcommand{\estot}{\mathcal{E}_{\sigma\to \tau}}

\newcommand{\wstot}{\mathcal{W}_{\sigma\to \tau}}
\newcommand{\wstol}{\mathcal{W}_{\sigma\to \etetalambda}}
\newcommand{\wstots}{\mathcal{W}_{\sigma\to \tau}^{\scriptscriptstyle [\tleq/\sleq]}}
\newcommand{\wstols}{\mathcal{W}_{\sigma\to \etetalambda}^{\scriptscriptstyle[\tleq/\sleq]}}

\newcommand{\uttol}{\mathcal{U}_{\tau\to \etetalambda}}
\newcommand{\ustol}{\mathcal{U}_{\sigma\to \etetalambda}}
\newcommand{\newuttol}{\hat\mathcal{U}_{\tau\to \etetalambda}}
\newcommand{\newustol}{\hat\mathcal{U}_{\sigma\to \etetalambda}}
\newcommand{\kaj}{\mathop{,\,}}
\newcommand{\aux}{\lor}
\colorlet{verda}{green!42}
\newcommand{\open}{\cellcolor{yellow!42}\textbf{?}}
\newcommand{\done}{\cellcolor{verda}\checkmark}

\newcommand{\txplane}{$\nt\nx$-plane\xspace}

\title[Interdefinability of some relations of spacetime]{Complexity in
  the interdefinability of timelike, lightlike and spacelike
  relatedness of Minkowski spacetime}

\thanks{This research is supported by the Hungarian National Research,
  Development and Innovation Office (NKFIH), grants no.\ FK-134732 and
  TKP2021-NVA-16.}

\author[{H.\ Andr\'eka, J.\ X.\ Madar\'asz, I.\ N\'emeti,
    G.\ Sz\'ekely}]{Hajnal Andr\'eka, Judit X.\ Madar\'asz, Istv\'an
  N\'emeti, Gergely Sz\'ekely}

\address{Alfr\'ed R\'enyi Institute of Mathematics,\\
  H-1053 Budapest, Re\'altanoda st. 13-15, Budapest, Hungary}

\begin{document}

\begin{abstract}
  Interdefinability of timelike, lightlike and spacelike relatedness
  of Minkowski spacetime is investigated in detail in the paper, with
  the aim of finding the simplest definitions. Based on ideas scattered
  in the literature, definitions are given between any two of these
  relations that use only 4 variables. All these definitions work over
  arbitrary Euclidean fields in place of the field of reals, if the
  dimension n of spacetime is greater than two. If n=2, the
  definitions work over arbitrary ordered fields except the ones based
  on lightlike relatedness (where no definition can work by
  symmetry). None of these relations can be defined from another one
  using only 3 variables. Our four-variable definitions use only one
  universal and one existential quantifiers in a specific order. In
  some of the cases, we show that the order of these quantifiers can
  be reversed for the price of using twice as many quantifiers.
  Except two cases, we provide existential/universal definitions using
  5 variables or show that no existential/universal definition
  exists. There are no existential/universal definitions between any
  two of these relations using only 4 variables.  It remains open
  whether there is an existential (universal) definition of timelike
  (lightlike) relatedness from spacelike relatedness if
  $n>2$. Finally, several other open problems related to the quantifier
  complexity of the simplest possible definitions are given.
\end{abstract}

\maketitle

\section{Introduction}

There is an extensive literature on the axiomatization of Minkowski
spacetime in terms of various basic concepts. Robb \cite{Robb14},
revised as \cite{Robb36}, gives an axiomatization of Minkowski
spacetime using only one binary ``after'' relation as primitive
notion. Goldblatt~\cite[Appendix B]{Goldblatt} shows that ``between''
and ``orthogonality'', the primitives used therein, can be defined in
terms of Robb's ``after'' relation. Latzer \cite{Latzer72} shows that
Minkowski spacetime can be axiomatized using the relation of lightlike
relatedness alone.  Both in terms of the binary asymmetric lightlike
after relation and in terms of the binary symmetric lightlike
relatedness relation, Mundy~\cite{Mundy} axiomatizes Minkowski
spacetime significantly simplifying Robb's axiom
system. Pambuccian~\cite{Pambuccian} explicitly defines collinearity
and equidistance from lightlike relatedness in metric-affine Fano
spaces in order to present Alexandrov--Zeeman type theorems as
definability results. For a contemporary axiomatization of Minkowski
spacetime pursued in the style of Tarski, see \cite{CB20}.

The three most fundamental symmetric binary relations in Minkowski
spacetime are timelike, lightlike and spacelike relatedness.
Scattered in the extensive literature mentioned above, there can be
found some definitions and some claims about certain interdefinability
among these relations. For example, Malament \cite[p.294]{Malament}
claims that the three basic relations are first-order definable in
terms of one another, but he gives no proof or reference except for
one specific formula in footnote 3. Interdefinability results
concerning Minkowski spacetime are used in general relativity, see
e.g., \cite{KP67}, \cite[especially p.180]{HKM76} and \cite{M77}. The
results in the present paper are used in the follow-up paper
\cite{MinkCA2}.

Herein we carefully investigate the interdefinability of timelike,
lightlike and spacelike relatedness relations aiming to find the
simplest definitions.  For each pair of these relations, we present
definitions that are simplest in terms of number of variables. We also
investigate how definability depends on the dimension of spacetime and
on the underlying ordered field. For example, lightlike relatedness is
definable form timelike relatedness even if the dimension $n$ of
spacetime is 2 (\ie not just time but space is also one dimensional)
but $n$ has to be at least 3 to be able to define timelike relatedness
from lightlike relatedness. In most of the cases, we assume that
either $n=2$ and the underlying field is an arbitrary ordered field or
$n>2$ and the underlying field is a Euclidean field (\ie an ordered
field in which every positive number has a square root).

In Section~\ref{sec-interdef}, we show that the relations of timelike,
lightlike and spacelike relatedness can be defined from any of them
using only four variables, see Theorems~\ref{thm-ttosl},
\ref{thm-stol} and \ref{thm-ltost}. Then we show that the number of
variables used in these definitions are minimal, \ie none of these
relations can be defined from another one using only three variables,
see Theorem~\ref{thm-nondef}. In Section~\ref{sec-complexity}, using a
well-visualizable construction, we show that spacelike relatedness can
be defined existentially from timelike relatedness using six
variables, see Theorem~\ref{thm-edef2}, but no such definitions exists
using only four variables \cf Theorem~\ref{thm-noedef2}. In
preparation to the proof of Theorem~\ref{thm-noedef2}, we develop a
picturesque graph-embedding-based method to understand relations
definable by existential formulas. Then using this method, we give a
five variable existential formula defining spacelike relatedness from
timelike relatedness, see Theorem~\ref{thm-newedef2}.  Over arbitrary
ordered fields, we show that lightlike relatedness cannot be defined
existentially neither from timelike nor from spacelike relatedness,
see Theorem~\ref{thm-noedef}.  Over arbitrary ordered fields, we show
that neither timelike nor spacelike relatedness can be defined
existentially or universally from lightlike relatedness, see
Theorem~\ref{thm-noedefll}. Finally in Section~\ref{sec-open}, we give
several open problems related to the quantifier complexity of the
simplest possible definitions.

\section{Some notations and definitions}

Here, we collect the most fundamental notations used in this paper. We
work over arbitrary ordered fields. So herein, we assume that
$(\Q,+,\cdot,\le)$ is an ordered field.\footnote{That
  $(\Q,+,\cdot,\le)$ is an ordered field means that $(\Q,+,\cdot)$ is
  a field which is totally ordered by $\le$, and we have the following
  two properties for all $x,y,z\in Q$: (1) $x+z\le y+z$ if $x\le y$,
  and (2) $0\le xy$ if $0\le x$ and $0\le y$.} Let $\Q^n$ denote the
set of $n$-dimensional vectors over $\Q$. We also use the vector space
structure on $\Q^n$. Following the convention common in relativity
theory, we start counting axes from $0$, draw the $0$th coordinate
axis vertically and consider it as the time-axis.

In our formulas, we use the following symbols for logical connectives:
both ``$\land$'' and ``$\kaj$'' interchangeably for conjunction,
``$\aux$'' for disjunction, ``$\lnot$'' for negation, ``$\exists$''
for existential quantifier, and ``$\forall$'' for universal
quantifier.

Let $\vp=(\np_0,\np_1,\ldots,\np_{n-1})\in\Q^n$ and
$\vq=(\nq_0,\nq_1,\ldots,\nq_{n-1})\in\Q^n$.  We are going to work
with the following relations illustrated by Figure~\ref{fig:tls}:
\begin{align*}
  \vp\tleq\vq  &\defiff (\np_0-\nq_0)^2> (\np_1-\nq_1)^2+\ldots +(\np_{n-1}-\nq_{n-1})^2\\
  \vp\lleq\vq  &\defiff (\np_0-\nq_0)^2 = (\np_1-\nq_1)^2+\ldots +(\np_{n-1}-\nq_{n-1})^2\kaj \np_0\neq \nq_0\\
  \vp\sleq\vq &\defiff (\np_0-\nq_0)^2< (\np_1-\nq_1)^2+\ldots +(\np_{n-1}-\nq_{n-1})^2.
\end{align*}
We say that \lemph{$\vp$ and
  $\vq$ are}, \semph{timelike}, \semph{lightlike},\footnote{Usually,
  it is allowed for lightlike related points to be equal, see \eg
  \cite{Malament}. Here we forbid equality by assuming $\np_0\neq
  \nq_0$. From the point of view of definability, this is a negligible
  difference since the two versions are clearly definable from each
  other. We prefer to use the strict one because that is minimal among
  the definable nonempty binary concepts.} \semph{spacelike related}
in the respective cases. A line is called timelike (lightlike,
spacelike) if{}f all of its distinct points are timelike (lightlike,
spacelike) related. By that $\vp$ is in the \semph{timelike
  future}\emph{ of} $\vq$, we simply mean that $\vp\tleq\vq$ and
$\np_0>\nq_0$. Analogously, $\vp$ is in the \semph{timelike
  past}\emph{ of} $\vq$ if{}f $\vp\tleq\vq$ and $\np_0<\nq_0$.

\begin{figure}
  \begin{center}
    \input{tlseq.tikz}
    \caption{Illustration for relations $\vp\tleq \vq$,
      $\vp\lleq\vq'$, and $\vp\sleq\vq''$, as well as for light cone
      $\Lambda_{\vp}$ \label{fig:tls}}
  \end{center}
\end{figure}
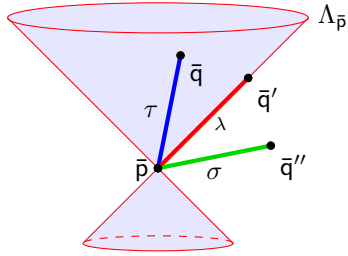

We denote the complement of these relations by $\tlneq$, $\llneq$,
$\slneq$. By $\aleq_{\tinyneq}$, we are going to abbreviate the
  intersection of binary relations $\aleq $ and $\neq$.  In our
formulas, we use the same symbols $\tleq$, $\lleq$ and $\sleq$ in
infix notation also to denote the corresponding relation symbols. We
abbreviate negated atomic formulas $\lnot (x = y)$ to $x\neq y$,
$\lnot (x \tleq y)$ to $x\tlneq y$, \etc to make them easier to
read. In case $\varphi(x,y)$ is a formula having all its free
variables among $x$ and $y$, we will write $\varphi(\vx,\vy)$ to
denote that the binary relation defined by $\varphi$ holds for
$\vx,\vy\in\Q^n$.

By the \semph{light cone} $\Lambda_{\vp}$ \lemph{through point $\vp$},
we understand the set of points which are lightlike related or equal to
$\vp$, \ie
\begin{equation*}
  \Lambda_{\vp}\defeq\big\{\vq\in\Q^n:\vq\lleq\vp \text{ or }
  \vp=\vq\big\}.
\end{equation*}

By the \semph{causal cone} $\Ccone_{\vp}$ \lemph{through point $\vp$},
we understand the set of points which are equal, timelike or lightlike
(aka.\ causally) related to $\vp$, \ie
\begin{equation*}
  \Ccone_{\vp}\defeq\big\{\vq\in\Q^n: \vq\slneq\vp\big\},
\end{equation*}
and $\vp$ is in the \semph{causal future
    (past)}\emph{ of} $\vq$ if{}f $\vp\slneq\vq$ and $\np_0\ge\nq_0$
  ($\np_0\le\nq_0$).

\begin{remark}[$n\ge2$]\label{rem-aut}
  The following is known and also straightforward to check. For all
  natural numbers $n\ge2$, the followings are all automorphisms of
  model $(\Q^n,\tleq,\lleq,\sleq)$: any uniform scaling, reversing
  time,\footnote{That is, the transformation $(\np_0,\np_1\ldots,
    \np_{n-1})\mapsto (-\np_0,\np_1\ldots, \np_{n-1})$.} any
  translation, any spatial rotation,\footnote{That is, any
    distance-preserving linear transformation of determinant 1 that
    fixes the time-axis, \ie maps vector $(1,0,\ldots0)$ to itself.}
  any Lorentz boost\footnote{A Lorentz boost in the $i$-th spatial
    coordinate direction is a transformation of the form
    \begin{equation*}
      B_i:(\np_0,\np_1,\ldots, \np_{n-1})\mapsto
      \left(\frac{\np_0-\nv\np_i}{\sqrt{1-\nv^2}},\np_1,\ldots,\np_{i-1},\frac{\np_i-\nv\np_0}{\sqrt{1-\nv^2}},\np_{i+1},\ldots,\np_{n-1}\right)
    \end{equation*}
    for some velocity $\nv\in \Q$ for which $-1<\nv<1$ and
    $\sqrt{1-\nv^2}\in\Q$.}  in any spatial direction.
\end{remark}

An ordered field $(\Q,+,\cdot,\le)$ is called a \semph{Euclidean
  field} if its every positive element has a square root, \ie if{}f
the following holds in it
\begin{equation}\label{eq-Eucl}\tag{Eucl.}
  \text{for all } x\in\Q, \text{ if } 0\le x,\text{ there is a }y\in\Q
  \text{ such that } x=y^2.
\end{equation}
In figures and tables, we are going to refer to this property as
\eqref{eq-Eucl}.
In some proofs, we will refer to plane $\{(\nt,\nx,0,\ldots,0):
\nt,\nx\in\Q\}$ as \semph{\txplane}.

\begin{prop}[$n\ge2$; \ref{eq-Eucl} or $n=2$]\label{prop-aut}
  If $n=2$ or $(\Q,+,\cdot,\le)$ is a Euclidean field, then the
  automorphism group of $(\Q^n,\tleq,\lleq,\sleq)$ acts transitively
  on the timelike, lightlike and spacelike related pairs of points,
  \ie for all $\aleq\;\in\{\tleq, \lleq, \sleq\}$ and for all
  $\vp,\vq,\vp',\vq'\in \Q^n$ for which both $\vp\aleq\vq$ and
  $\vp'\aleq\vq'$ holds, there is an automorphism $\alpha$ of
  $(\Q^n,\tleq,\lleq,\sleq)$ for which $\alpha(\vp)=\vp'$ and
  $\alpha(\vq)=\vq'$.
\end{prop}

\begin{proof}
  Because the inverses and the compositions of automorphisms are
  automorphisms, it is enough to show that for every two distinct
  points $\vp,\vq\in\Q^n$ there is an automorphism $\alpha$ of
  $(\Q^n,\tleq,\lleq,\sleq)$ such that $\alpha$ maps $\vp$ to
  $(0,0,\ldots,0)$ and $\alpha(\vq)=(1,0,\ldots,0)$ or
  $\alpha(\vq)=(1,1,0,\ldots,0)$ or $\alpha(\vq)=(0,1,0,\ldots,0)$.

  Since translations are automorphisms of $(\Q^n,\tleq,\lleq,\sleq)$,
  \cf Remark~\ref{rem-aut}, there is an automorphism which maps $\vp$
  to $(0,0,\ldots,0)$. So because the composition of two automorphisms
  is also an automorphism, we can assume that $\vp=(0,0,\ldots,0)$
  without loosing generality. Similarly, because reversing time is
  also an automorphism, \cf Remark~\ref{rem-aut}, we can also assume
  that $\nq_0\ge0$.

  Assume now that $(\Q,+,\cdot,\le)$ is a Euclidean field. Then there
  is a spatial rotation $R$ which does not change $\vp=(0,0,\ldots,0)$
  and maps $\vq$ to $(\nt,\nx,0,\ldots, 0)$ for some non-negative
  $\nt,\nx\in\Q$. This rotation $R$ is an automorphism, \cf
  Remark~\ref{rem-aut}. Hence we can assume that $\vp=(0,0,\ldots,0)$
  and $\vq=(\nt,\nx,0,\ldots, 0)$ for some $\nt\ge0$, $\nx\ge0$ for
  which $(\nt,\nx)\neq(0,0)$.

  If $\nt>\nx\ge0$, then $\vp\tleq \vq$. Let us consider the following
  map:
  \begin{equation*}
    B_{\tau}:(\nr_0,\nr_1,\nr_2,\ldots, \nr_{n-1}) \mapsto
    \frac{1}{\sqrt{\nt^2-\nx^2}}\left(\frac{\nt\nr_0-\nx\nr_1}{\sqrt{\nt^2-\nx^2}},\frac{\nt\nr_1-\nx\nr_0}{\sqrt{\nt^2-\nx^2}},\nr_2,\ldots,\nr_{n-1}\right).
  \end{equation*}
  This map $B_{\tau}$ is an automorphism because it is the composition
  of a Lorentz boost corresponding to speed $\nv=\nx/\nt$ and a
  uniform scaling by the factor $1/\sqrt{\nt^2-\nx^2}$. It is
  straightforward to verify that $B_{\tau}$ maps $\vp$ to
  $(0,0,\ldots,0)$ and $\vq$ to $(1,0,\ldots,0)$.  If $\nt=\nx>0$,
  then $\vp\lleq \vq$ and there is a uniform scaling that maps $\vp$
  to $(0,0,\ldots,0)$ and $\vq$ to $(1,1,0,\ldots,0)$. If $0<\nt<\nx$,
  then $\vp\sleq\vq$ and let
  \begin{equation*}
    B_{\sigma}:(\nr_0,\nr_1,\nr_2,\ldots, \nr_{n-1}) \mapsto
    \frac{1}{\sqrt{\nx^2-\nt^2}}\left(\frac{\nx\nr_0-\nt\nr_1}{\sqrt{\nx^2-\nt^2}},\frac{\nx\nr_1-\nt\nr_0}{\sqrt{\nx^2-\nt^2}},\nr_2,\ldots,\nr_{n-1}\right).
  \end{equation*}
  Analogously, this $B_{\sigma}$ is an automorphism and maps $\vp$ to
  $(0,0,\ldots,0)$ and $\vq$ to $(0,1,0,\ldots,0)$. If $\nt=0<\nx$, then
  $\vp\sleq \vq$ and there is a uniform scaling which maps $\vp$ and
  $\vq$ to $(0,0,\ldots,0)$ and $(0,1,0,\ldots,0)$, respectively. This
  completes the proof in the case when $(\Q,+,\cdot,\le)$ is a
  Euclidean field.

  There were only two steps where we used the Euclidean
    property: the one where we used the existence of an appropriate
    rotation $R$ to rotate points $\vp$ and $\vq$ to the \txplane, and
    the one where we used the existence of maps $B_{\tleq}$ and
    $B_{\sleq}$. Since when $n=2$, the rotation step is not needed and
    $B_{\tleq}$ and $B_{\sleq}$ simplify into rational functions in
    both coordinates, we also have the statement over arbitrary
  ordered fields if $n=2$.
\end{proof}

Because any relation definable in a model has to be closed under the
automorphisms of the model, the following is a corollary of
Proposition~\ref{prop-aut}.

\begin{corollary}[$n\ge2$; \ref{eq-Eucl} or $n=2$]\label{cor-aut}
  Let $\aleq$ and $\bleq$ be any two different relations from the set
  $\{\tleq, \lleq, \sleq\}$.  For any relation $R$ definable in
  $(\Q^n,\aleq)$, if $n=2$ or $(\Q,+,\cdot,\le)$ is a Euclidean
  field, we have that
  \begin{equation*}
    \bleq\;\subseteq R \iff(p,q)\in R \text{ for some }p,q\in\Q^n
    \text{ such that }p\bleq q \iff R\cap \bleq\neq\emptyset.
  \end{equation*}
  \hfill \qed
 \end{corollary}

Because the union of relations $\tleq$, $\sleq$, $\lleq$ and $=$ is
the universal relation $\Q^n\times\Q^n$, Corollary~\ref{cor-aut}
implies the following.

\begin{corollary}[$n\ge2$; \ref{eq-Eucl} or $n=2$]\label{cor-atoms}
  Let $\aleq$ be any of relations $\tleq$, $\lleq$ and $\sleq$. If
  $n=2$ or $(\Q,+,\cdot,\le)$ is a Euclidean field, then any nonempty
  binary relation which is definable in $(\Q^n,\aleq)$ is the union of
  some of relations $\tleq$, $\sleq$, $\lleq$ and $=$.\hfill \qed
\end{corollary}

In other words, Corollary~\ref{cor-atoms} says that, assuming $n=2$ or
$(\Q,+,\cdot,\le)$ is a Euclidean field, if relations $\tleq$,
$\sleq$ and $\lleq$ are definable in any of structures $(\Q^n,\tleq)$,
$(\Q^n,\lleq)$ and $(\Q^n,\sleq)$, then they and $=$ are atoms in the
corresponding Boolean algebra of definable binary concepts. We will
see that they are all definable except in $(\Q^2,\lleq)$, see
Remark~\ref{rem-2d} and Theorems~\ref{thm-ttosl}, \ref{thm-stol} and
\ref{thm-ltost} in Section~\ref{sec-interdef}. So except in case
$(\Q^2,\lleq)$, we have that the Boolean algebra of definable binary
concepts is the 16 element one illustrated by Figure~\ref{fig-BA}. In
case of $(\Q^2,\lleq)$, the algebra of definable binary concepts is
the 8 element one generated by relations $\lleq$ and $=$.

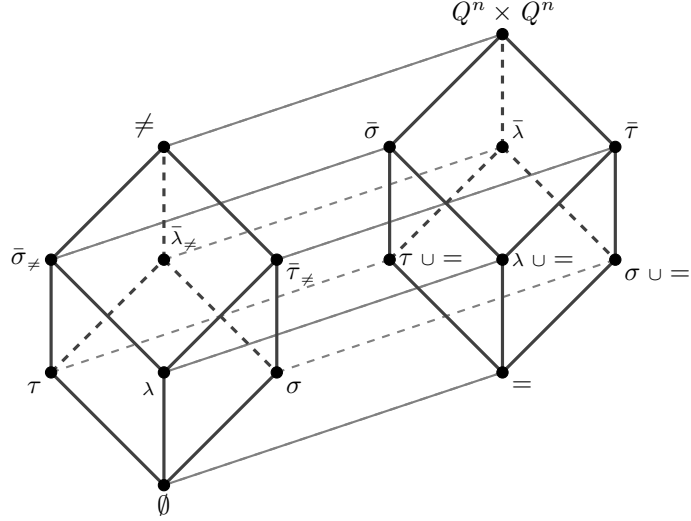
\begin{figure}[!htb]
  \input{BA2relat.tikz}
  \caption{The figure illustrates the Boolean algebra generated by
    relations $\tleq$, $\lleq$, $\sleq$, and $=$. \label{fig-BA}}
\end{figure}

\section{Interdefinability using minimal number of variables}
\label{sec-interdef}

In this section, we are going to show that timelike, spacelike and
lightlike relatedness are definable from one another using 4
variables, but not definable using only 3 variables. To do so, let us
introduce first formula $\ttos$ defining spacelike relatedness $\sleq$
from timelike relatedness $\tleq$:
\begin{equation*}
  \ttos(x,y)\defeq x\neq y\kaj \forall z \big(z=x\aux z=y \aux \exists
  u (u\tleq z\kaj u\tlneq x \kaj u\tlneq y)\big).
\end{equation*}

Formula $\ttos(x,y)$ intuitively says that $x\neq y$ and inside the
light cone of every point $z$ (different from $x$ and $y$) there is a
point $u$ which is inside neither the light cone of $x$ nor that of
$y$, see Figure~\ref{fig-ttos}. Using this definition, we can easily
define $\lleq$ from $\tleq$ by the following formula:
\begin{equation*}
  \ttol(x,y)\defeq \lnot \ttos(x,y)\kaj x\tlneq y\kaj x\neq y.
\end{equation*}

\begin{figure}[h!tb]
  \input{ttos.tikz}
  \caption{Formula $\ttos(x,y)$, defining spacelike relatedness
    $\sleq$ from timelike relatedness $\tleq$, intuitively says that
    $x$ and $y$ are distinct points and inside the light cone of every
    point $z$ (different from $x$ and $y$) there is a point $u$ which
    is inside neither the light cone of $x$ nor that of
    $y$. \label{fig-ttos}}
\end{figure}

\begin{theorem}[$n\ge 2$; \ref{eq-Eucl} or $n=2$]\label{thm-ttosl}
  Assume that $n=2$ or that $(\Q,+,\cdot,\le)$ is a Euclidean field.
  Then in model $(\Q^n,\tleq)$, spacelike relatedness $\sleq$ can be
  defined from timelike relatedness $\tleq$ using 4 variables by
  formula $\ttos$. Hence lightlike relatedness $\lleq$ can be defined
  from timelike relatedness $\tleq$ using 4 variables by formula
  $\ttol$.
\end{theorem}

\begin{proof}
  Let us first show, for all $\vx,\vy\in\Q^n$, that $\vx\sleq\vy$
  holds exactly if $\ttos(\vx, \vy)$ holds.  To do so, assume that
  $\vx\slneq \vy$. Then we should show that $\ttos(\vx,\vy)$ does not
  hold. If $\vx=\vy$, then clearly $\ttos(\vx,\vy)$ does not hold as
  it contains $\vx\neq \vy$.  So assume that $\vx\neq \vy$. Let $\vz$
  be the midpoint of line segment $\vx\vy$. Then, since $\vx\neq \vz$
  and $\vz\neq \vy$ (as $\vx\neq \vy$), we should show that $\vu\tleq
  \vx$ or $\vu\tleq \vy$ holds for all $\vu$ for which $\vu\tleq \vz$
  holds. By symmetry, without loss of generality, we can assume that
  $\vy$ is in the causal future of $\vx$. There are two possibilities:
  if $\vu$ is in the timelike future of $\vz$, then $\vu$ is also in
  the timelike future of $\vx$ and hence $\vu\tleq \vx$ holds; if
  $\vu$ is in the timelike past of $\vz$, then $\vu$ is in the
  timelike past of $\vy$ and hence $\vu\tleq \vy$ holds, see
  Figure~\ref{fig-ttos1}.

  \begin{figure}[h!tb]
    \input{ttos1.tikz}
    \caption{\label{fig-ttos1}}
  \end{figure}

  To show the other direction, assume that $\vx\sleq \vy$. Without
  loss of generality, we can assume that $\vx$ and $\vy$ are in the
  same horizontal hyperplane $H$ because there is an automorphism of
  $(\Q^n,\tleq,\sleq)$ that maps any two spacelike related points to
  horizontally related ones, \cf Proposition~\ref{prop-aut}. We should
  show that for all $\vz$ which is different from $\vx$ and $\vy$,
  there is a $\vu$ timelike related to $\vz$, which is timelike
  related neither to $\vx$ nor to $\vy$. There are two cases, see
  Figure~\ref{fig-ttos2}:
  \begin{figure}[h!tb]
    \input{ttos2.tikz}
    \caption{\label{fig-ttos2}}
  \end{figure}
   \begin{itemize}
  \item either $\vz$ is not in the hyperplane $H$, and then there is a
    $\vu\in H$ vertically related and hence also timelike
    related to $\vz$; since $\vu\in H$, this $\vu$ is spacelike
    related to both $\vx$ and $\vy$;
  \item or $\vz$ is in the hyperplane $H$ and thus spacelike related
    to both $\vx$ and $\vy$, and hence there is a close enough $\vu$
    which is timelike related to $\vz$ but still spacelike related to
    $\vx$ and $\vy$.
  \end{itemize}
  
  The proof of the second part is straightforward since for all $\vx$
  and $\vy$ exactly one of relations $\vx\tleq \vy$, $\vx\lleq \vy$,
  $\vx\sleq \vy$ and $\vx=\vy$ holds.
\end{proof}

Timelike relatedness and lightlike relatedness can be defined from
spacelike relatedness by the following analogous
formulas:
\begin{equation*}
  \stot(x,y)\defeq x\neq y\kaj \forall z \big(z=x\aux z=y \aux
  \exists u (u\sleq z\kaj u\slneq x \kaj u\slneq y)\big).
\end{equation*}
\begin{equation*}
  \stol(x,y)\defeq \lnot \stot(x,y)\kaj x\slneq y\kaj x\neq
  y.\footnotemark
\end{equation*}\footnotetext{Let us note that this formula is basically the one mentioned by Malament~\cite[footnote 8]{Malament} defining lightlike relatedness  from causality relation $\bar\sigma$.}
Let us note here that, even though there is no symmetry between
relations $\sleq$ and $\tleq$ unless $n=2$, there is a nice symmetry
between defining formulas $\stot$ and $\ttos$ (as well as $\stol$ and
$\ttol$) as they can be achieved from each other by interchanging
relations $\sleq$ and $\tleq$.

\begin{figure}[h!tb]
  \input{stot.tikz}
  \caption{Formula $\stot(x,y)$, defining timelike relatedness $\tleq$
    from spacelike relatedness $\sleq$, intuitively says that $x$ and
    $y$ are distinct and outside the light cone of every point $z$
    (different from $x$ and $y$) there is a point $u$ which is outside
    neither the light cone of $x$ nor that of $y$. \label{fig-stot}}
\end{figure}

\begin{theorem}[$n\ge2$; \ref{eq-Eucl} or $n=2$]\label{thm-stol}
  Assume that $n=2$ or that $(\Q,+,\cdot,\le)$ is a Euclidean field.
  Then in model $(\Q^n,\sleq)$, timelike relatedness $\tleq$ can be
  defined from spacelike relatedness $\sleq$ using 4 variables by
  formula $\stot$. Hence lightlike relatedness $\lleq$ can be defined
  from spacelike relatedness $\sleq$ using 4 variables by formula
  $\stol$.
\end{theorem}

\begin{proof}
  First we show, for all $\vx,\vy\in\Q^n$, that $\vx\tleq\vy$ holds
  exactly if $\stot(\vx, \vy)$ holds. So let us assume that $\vx\tleq
  \vy$. Without loss of generality, we can assume that $\vx$ and $\vy$
  are in the same vertical line $\ell$ because there is an
  automorphism of $(\Q^n,\sleq,\tleq)$ that maps any two timelike
  related points to vertically related ones, \cf
  Proposition~\ref{prop-aut}. We should show that for all $\vz$ which
  is different from $\vx$ and $\vy$, there is a $\vu$ spacelike
  related to $\vz$, which is  spacelike related neither to $\vx$ nor to
  $\vy$. To find such a $\vu$, let us consider the hyperplane $H$
  through $\vz$ orthogonal to $\ell$, see Figure~\ref{fig-stot2}. If
  $\vz\not\in\ell$, let $\vu$ be the intersection point of $\ell$ and
  $H$; and if $\vz\in\ell$, let $\vu\in H$ be such a point distinct
  from $\vz$ that the distance between $\vu$ and $\vz$ is less than
  both the distance of $\vz$ and $\vx$ and the distance of $\vz$ and
  $\vy$. As desired, by its choice, $\vu$ is spacelike related to
  $\vz$ but not to $\vx$ or $\vy$.

  \begin{figure}[h!tb]
    \input{stot2.tikz}
    \caption{\label{fig-stot2}}
  \end{figure}
  
  To show the other direction, assume that $\vx\tlneq \vy$. Then we
  should show that $\stot(\vx,\vy)$ does not hold. If $\vx=\vy$, then
  clearly $\stot(\vx,\vy)$ does not hold as it contains $\vx\neq
  \vy$. So assume that $\vx\neq \vy$ and let us choose $\vz$ to be the
  midpoint of line segment $\vx\vy$, see Figure~\ref{fig-stot1}. Then,
  since $\vz\neq \vx$ and $\vz\neq \vy$ (as $\vx\neq \vy$), we should
  show that for all $\vu$, if $\vu\slneq \vx$ and $\vu\slneq \vy$,
  then $\vu\slneq \vz$. There are two cases to consider: either
  $\vx\lleq \vy$ or $\vx\sleq \vy$ (in this second case, we can assume
  that $\vx$ and $\vy$ are horizontally related by the usual
  automorphism argument, \cf Proposition~\ref{prop-aut}).  In both
  cases, $\vu\slneq \vz$ holds for all the points for which $\vu\slneq
  \vx$ and $\vu\slneq \vy$ because the intersection of causal cones
  through $\vx$ and $\vy$ are contained in the causal cone through
  $\vz$.

  \begin{figure}[h!tb]
    \input{stot1.tikz}
    \caption{\label{fig-stot1}}
  \end{figure}
  
  As before, the proof of the second part is straightforward since for
  all $\vx$ and $\vy$ exactly one of relations $\vx\tleq \vy$,
  $\vx\lleq \vy$, $\vx\sleq \vy$ and $\vx=\vy$ holds.
\end{proof}

\begin{figure}[!htb]
    \input{ltos.tikz}
    \caption[Caption]{Formula $\ltos(x,y)$, defining spacelike
    relatedness $\sleq$ from lightlike relatedness $\lleq$,
    intuitively says that $x$ and $y$ are not lightlike related, and
    there is a point $z\neq x$ on the light cone of $x$ such that $z$
    is not lightlike related to $y$ and there is no point $u$ on the
    light cone of $y$ to which both $x$ and $z$ are lightlike
    related. Since there are no non-degenerate lightlike
    triangles, this basically means that there
    is a lightlike line $\ell$ containing $x$ but not $y$ such that
    there is no lightlike line through $y$ which intersects
    $\ell$.}\label{fig-ltos}
\end{figure}
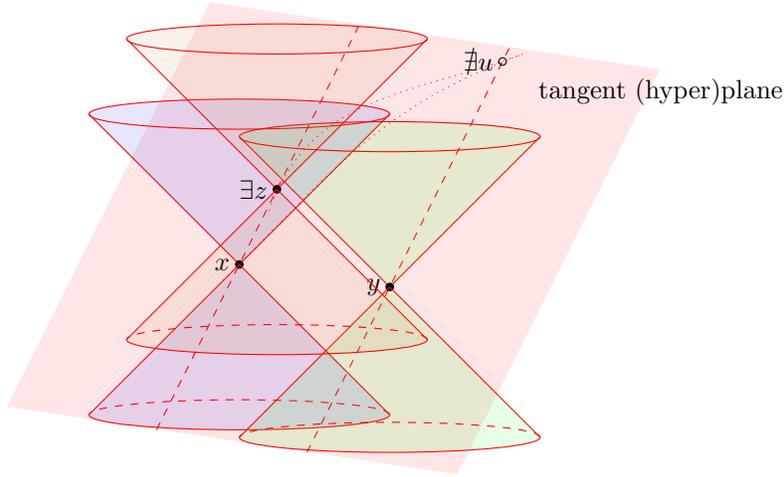

The following formulas are based on ideas used in \cite[Appendix
  B]{Goldblatt}:
\begin{equation*}
  \ltos(x,y)\defeq x\llneq y \kaj \exists z \big( x\lleq z\kaj y\llneq
  z\kaj\neg\exists u (u\lleq x\kaj u\lleq y\kaj u\lleq z)\big).
\end{equation*}
\begin{equation*}
  \ltot(x,y)\defeq \lnot \ltos(x,y)\kaj x\llneq y\kaj x\neq y.
\end{equation*}

\begin{theorem}[\ref{eq-Eucl} and $n\ge3$]\label{thm-ltost}
  Let $n\ge3$ and assume that ordered field $(\Q,+,\cdot, \le)$ is
  Euclidean. In model $(\Q^n,\lleq)$, spacelike relatedness $\sleq$
  can be defined from lightlike relatedness $\lleq$ using 4
  variables by formula $\ltos$. Hence timelike relatedness $\tleq$ can
  be defined from lightlike relatedness $\lleq$ using 4 variables
  by formula $\ltot$.
\end{theorem}

\begin{proof}
  First we show, for all $\vx,\vy\in\Q^n$, that $\vx\sleq\vy$ holds
  exactly if $\ltos(\vx, \vy)$ holds. Thus let us assume that
  $\vx\sleq \vy$. Then we should show that $\ltos(\vx,\vy)$ also
  holds. To do so, consider the light cone $\Lambda_{\vx}$ through
  $\vx$ and take a tangent hyperplane $H$ to this light cone
  containing $\vy$, such $H$ exists because $\vx$ is spacelike related
  to $\vy$ and $(\Q,+,\cdot,\le)$ is a Euclidean field, see
  Figure~\ref{fig-ltos}. Let $\vz$ be any point but $\vx$ from line
  $\ell:=\Lambda_{\vx}\cap H$. Then $\vx\llneq \vy$ holds
    because $\vx$ and $\vy$ are assumed to be spacelike related. Since
    $\vz\in\Lambda_{\vx}$ and $\vz\neq\vx$, we have that
    $\vx\lleq\vz$.  Since it is a tangent hyperplane of a light
      cone, all lightlike lines are parallel in $H$. Hence
      $\vy\llneq\vz$ also holds since $\vy\not\in\ell$ (as
      $\vy\not\in\Lambda_{\vx}$). So it remains to show that there is
    no $\vu$ for which $\vu\lleq \vx$, $\vu\lleq \vy$ and $\vu\lleq
    \vz$. To show this, let $\vu$ be lightlike related to both $\vx$
  and $\vz$. Then $\vu\in\ell$ because all lightlike triangles are
  degenerate and $\vu\vx\vz$ is a lightlike triangle. Therefore, $\vu$
  cannot be lightlike related to $\vy$ because all lightlike lines are
  parallel in $H$. Hence, there is no $\vu$ for which all of
    $\vu\lleq \vx$, $\vu\lleq \vy$ and $\vu\lleq \vz$
    hold. Consequently, $\ltos(\vx,\vy)$ holds and this is what we
  wanted to show.

  Let us now assume that $\vx\slneq \vy$. Then we should show that
  $\ltos(\vx,\vy)$ does not hold, \ie $\vx\lleq\vy$ or,
  for all $\vz$ which is lightlike related to $\vx$ but not
    lightlike related to $\vy$, we should be able to find a
  $\vu$ such that $\vu$ is lightlike related to $\vx$, $\vy$ and
  $\vz$. Assumption $\vx\slneq \vy$ means that $\vx=\vy$,
    $\vx\lleq\vy$ or $\vx\tleq \vy$. Since $\vx\lleq \vz$ and
      $\vy\llneq \vz$ imply that $\vx\neq\vy$, the only nontrivial
    case to be checked is $\vx\tleq \vy$.

  So let us assume that $\vx\tleq \vy$ and take any $\vz$ for
    which $\vz\lleq\vx$ and $\vz\llneq\vy$ hold. Then let us consider
  the lightlike line containing $\vx$ and $\vz$ and the
  nonparallel lightlike line through $\vy$ in the plane determined by
  $\vx$, $\vy$ and $\vz$ (this lightlike line through
    $\vy$ exists because the plane determined by $\vx$, $\vy$ and
    $\vz$ contains a timelike line). These two lines intersect
  because they are coplanar nonparallel lightlike lines. Let
    $\vu$ be their intersection. By its choice, this $\vu$ is
    lightlike related to $\vx$, $\vu$ and $\vz$ as desired. 
  
  Again the proof of the second part is straightforward since for all
  $\vx$ and $\vy$ exactly one of relations $\vx\tleq \vy$, $\vx\lleq
  \vy$, $\vx\sleq \vy$, and $\vx=\vy$ holds.
\end{proof}

\begin{remark}[$n=2$]\label{rem-2d}
  Over any ordered field $(\Q,+,\cdot,\le)$, if $n=2$, neither
  spacelike relatedness $\sleq$ nor timelike relatedness $\tleq$ can
  be defined from lightlike relatedness $\lleq$ because map $\alpha:
  (\nt,\nx)\mapsto (\nx,\nt)$ is an automorphism of model
  $(\Q^2,\lleq)$ taking relation $\sleq$ to $\tleq$.
\end{remark}

\begin{remark}[Non-\ref{eq-Eucl}]\label{rem-nonEucl}
  The assumption that $(\Q,+,\cdot, \le)$ is Euclidean cannot be
  omitted from Theorem~\ref{thm-ltost}. To see this, let us consider
  the set of real numbers which can be written as $\na+\nb\sqrt{2}$
  for some rational numbers $\na$ and $\nb$. It is known and also easy
  to check that these numbers form an ordered subfield of real
  numbers. Map $\alpha: \nx+\ny\sqrt{2}\mapsto \nx-\ny\sqrt{2}$
  preserves the addition and the multiplication in this field, and
  hence bijection $\hat\alpha: (\np_0,\np_1,\ldots,\np_{n-1})\mapsto
  (\alpha(\np_0),\alpha(\np_1),\ldots,\alpha(\np_{n-1}))$ takes lines
  to lines, and preserves lightlike relatedness. Since $\hat\alpha$
  interchanges timelike vector $\left(1,1-\sqrt{2}/2,0\ldots,0\right)$
  and spacelike vector $\left(1,1+\sqrt{2}/2,0,\ldots,0\right)$,
  neither timelike relatedness nor spacelike relatedness can be
  defined from lightlike relatedness over this ordered field. This
  counterexample can be generalized over any ordered field
  $(\Q,+,\cdot, \le)$ which has an automorphism $\alpha$ of
  $(\Q,+,\cdot)$ which does not preserve the ordering $\le$. This is
  so because in those fields there is an element $\nb>0$ such that
  $\alpha(\nb)<0$. If $\nb<1$, then $(1,1-\nb,0,\ldots 0)$ is a
  timelike vector whose image is spacelike, otherwise
  $(1,1-1/\nb,0,\ldots, 0)$ is such a vector.
\end{remark}

\begin{remark}[Non-\ref{eq-Eucl}]
  Even though their picturesque proof used in this paper do not work
  over arbitrary ordered fields since they use
  Proposition~\ref{prop-aut} which does not hold, for example, in the
  field of rational numbers, Theorems~\ref{thm-ttosl} and
  \ref{thm-stol} hold over non-Euclidean fields.
\end{remark}

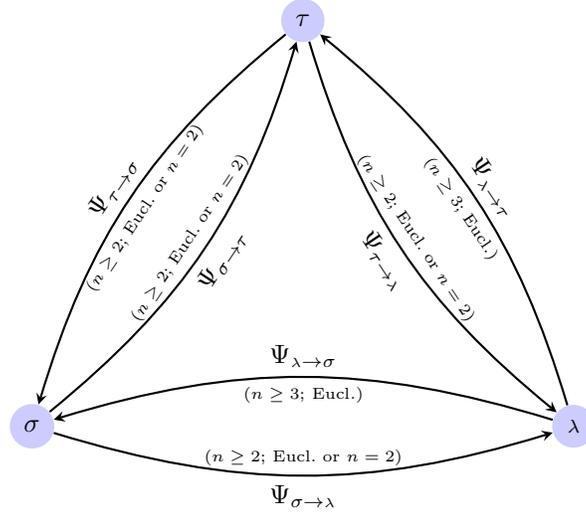
\begin{figure}[h!tb]
  \input{deftriangle.tikz}
  \caption{Here we illustrate formulas defining timelike $\tleq$,
    spacelike $\sleq$, lightlike relatedness $\lleq$ from each other
    using only 4 variables in the corresponding space-time dimensions
    $n$. }
\end{figure}

\begin{theorem}[$n\ge 2$]\label{thm-nondef}
  Over any ordered field $(\Q,+,\cdot,\le)$, from none of the
  relations $\tleq$, $\lleq$ and $\sleq$, another one of them can be
  defined using only 3 variables.
\end{theorem}

\begin{proof}
  By a theorem of Tarski and Givant \cite{TG87}, see also
  \cite[Prop. 3.18, p.111]{BD07}, a binary relation $B$ on a set $U$
  is first-order definable from a set $\mathcal{R}$ of binary
  relations on $U$ using only 3 variables if{}f $B$ can be built up
  from members of $\mathcal{R}$ and the identity relation $\Id$ by
  using the following operations: intersection $\cap$, union $\cup$,
  complement $\komplemento$, relation composition $\comp$ and converse
  $\converse$. Let $\aleq$ be any of the relations $\tleq$, $\lleq$
  and $\sleq$. To show that none of the other relations is definable
  from $\aleq$, we are going to show that set
  $$\mathcal{S}=\big\{\emptyset,\; \Id,\; \neq,\; \aleq,\; \alneq,\;
  \alneq\cap\neq,\; \aleq\cup\Id,\;\Q^n\times\Q^n\big\}$$is closed
  under the above relation operations. $\mathcal{S}$ is clearly closed
  under complement by De Morgan's law and it is easy to check that
  $\mathcal{S}$ is closed under intersection and union, \cf
  Figure~\ref{fig-BA}. Hence $\mathcal{S}$ is a \emph{field of
    sets}\footnote{That is a subalgebra of the power set Boolean
    algebra of $\Q^n\times\Q^n$.} with atoms $\Id$, $\aleq$ and
  $\alneq_{\tinyneq}$ (which is $\alneq\cap \neq$). Since identity $(P
  \cup Q) \comp R \equiv (P \comp R) \cup (Q \comp R)$ holds for all
  binary relations $P$, $Q$ and $R$, it is enough to show that the
  various compositions of atoms $\Id$, $\aleq$ and $\alneq_{\tinyneq}$
  are in $\mathcal{S}$, which holds because $\aleq\comp\aleq$ and
  $\alneq_{\tinyneq}\comp\alneq_{\tinyneq}$ are the universal relation
  $\Q^n\times\Q^n$, compositions $\alneq_{\tinyneq}\comp\aleq$ and
  $\aleq\comp\alneq_{\tinyneq}$ are equal to the relation $\neq$, and
  $\Id$ is an identity element with respect to the relation
  composition.  Since all the relations in $\mathcal{S}$ are
  symmetric, they are equal to their converses. Therefore, all the
  binary relations first-order definable from $\aleq$ using only 3
  variables are in $\mathcal{S}$. This completes the proof of the
  theorem.
\end{proof}

\section{Quantifier complexity}
\label{sec-complexity}

Let us now investigate the quantifier complexity of the possible
definitions. To do so, let us recall that a formula is called
\semph{universal} if{}f it is of form $\forall x_1\ldots \forall x_k
\psi$ for some quantifier free formula $\psi$ and variables
$x_1,\ldots, x_k$. A formula is \semph{existential} in the completely
analogous case when $\forall$ is replaced by $\exists$. A formula is
called \semph{universal-existential} if{}f it is of form $\forall
y_1\ldots \forall y_m \exists x_1\ldots \exists x_k\psi$ for some
quantifier free formula $\psi$ and variables $x_1,\ldots,
x_k,y_1\ldots, y_m$. The definition of an
\semph{existential-universal} formula can be obtained from this by
interchanging the universal and existential quantifiers. In the
superscript, we indicate the number of quantifies, and we use * in
case this number is not specified, \eg by
$\forall^1\exists^1$-formulas we mean universal-existential formulas
containing exactly one universal and one existential quantifiers, and
by $\forall^*\exists^*$-formulas we mean all universal-existential
formulas.

It is straightforward to show that defining $\ttos$, $\stot$ and
$\ltot$ are logically equivalent to $\forall^1\exists^1$-formulas and
$\ttol$, $\stol$ and $\ltos$ are logically equivalent to
$\exists^1\forall^1$-formulas, see Table~\ref{q-table}. So these
formulas are also quite simple in terms of quantifiers.

\begin{table}[!hbt]
  \caption{\label{q-table}}
  \begin{tabular}{|c|c;{1pt/2pt}c|c;{1pt/2pt}c|c;{1pt/2pt}c|}
    \hline defining formula & $\ttos$ & $\ttol$ & $\stot$ & $\stol$
    &$\ltot$ & $\ltos$ \\
    
    \hline complexity & $\forall^1\exists^1$& $\exists^1\forall^1$ &
    $\forall^1\exists^1$ &$\exists^1\forall^1$ & $\forall^1\exists^1$
    & $\exists^1\forall^1$ \\ \hline
  \end{tabular}
\end{table}

It is a natural question whether it is possible to find even simpler
defining formulas in terms of quantifier complexity. Now we are going
to investigate this question looking for existential and universal
definitions. To make our formulas more concise, we are going to use
the following abbreviation for certain binary relations $\aleq$ and
$\bleq$:
\begin{equation*}
  x \aleq\bleq yz \defiff x\aleq y \kaj x\bleq z.
\end{equation*}
We will also use the natural generalization of this abbreviation for
more than two binary relations.

In Theorem~\ref{thm-edef2} below, we are going to show that the
following six variable existential formula $\ettos$ defines spacelike
relatedness $\sleq$ from timelike relatedness $\tleq$ in
$(\Q^n,\tleq)$:
\begin{equation*}
  \ettos(p,q)  \defeq
  \exists r\exists x\exists s \exists
  z(r\tleq\tleq pq \kaj x \tleq \tlneq pq \kaj s \tlneq \tlneq
  pq\kaj z \tlneq \tleq pq\kaj r \tlneq\tleq\tlneq xsz),
\end{equation*}
and hence 
\begin{equation*}
  \uttol(p,q) \defeq \lnot\ettos(p,q) \kaj p\neq q \kaj p\tlneq q  
\end{equation*}
gives a universal definition of lightlike relatedness from $\tleq$.

\begin{theorem}[$n\ge2$; \ref{eq-Eucl} or $n=2$]\label{thm-edef2}
  Assume that $n=2$ or that $(\Q,+,\cdot,\le)$ is a Euclidean field,
  then:
  \begin{enumerate}
  \item\label{e2ttos} Spacelike relatedness $\sleq$ can be defined
    existentially from timelike relatedness $\tleq$ using only 6
    variables by $\exists^4$-formula $\ettos$ in $(\Q^n,\tleq)$.
  \item\label{u2ttol} Lightlike relatedness $\lleq$ can be defined
    universally from timelike relatedness $\tleq$ using only 6
    variables by $\forall^4$-formula $\uttol$ in $(\Q^n,\tleq)$.
  \end{enumerate}
\end{theorem}

\begin{proof}
  First we are going to show that $\vp\sleq\vq$ holds exactly if
  $\ettos(\vp, \vq)$ holds. If $\vp\sleq\vq$, then we can assume,
  without loosing generality, that $\vp$ and $\vq$ are horizontally
  related, \cf Proposition~\ref{prop-aut}. In this case, it is easy to
  verify that there are (even coplanar) points
  $\vr,\vx,\vs,\vz\in\Q^n$ for which relations $\vr\tleq\tleq \vp\vq$,
  $\vx \tleq \tlneq \vp\vq$, $\vs \tlneq \tlneq \vp\vq$, $\vz \tlneq
  \tleq \vp\vq$, $\vr\tleq \vs$ and $\vr \tlneq\tlneq \vx\vz$ hold;
  and hence $\ettos(\vp, \vq)$ also holds, see Figure~\ref{fig-ettos},
  where the regions are colored and labeled based on how the points
  from there are respectively related to $\vp$ and $\vq$.  So,
  according to Figure~\ref{fig-ettos}, $\vr\tleq\tleq \vp\vq$ holds
  if{}f $\vr$ is in a brown $\tleq\tleq$ region, $\vx \tleq \tlneq
  \vp\vq$ holds if{}f $\vx$ is in a blue $\tleq \tlneq$ region, $\vs
  \tlneq \tlneq \vp\vq$ holds if{}f $\vs$ is in a red $\tlneq\tlneq$
  region, and $\vz \tlneq \tleq \vp\vq$ holds if{}f $\vz$ in a green
  $\tlneq\tleq$ region.

  \begin{figure}[!htb]
    \centering
    \input{eTtoS.tikz}
    \caption{Here, we illustrate the proof of that $\ettos(\vp,\vq)$
      holds exactly if $\vp\sleq \vq$ holds.
      \label{fig-ettos}}
  \end{figure}
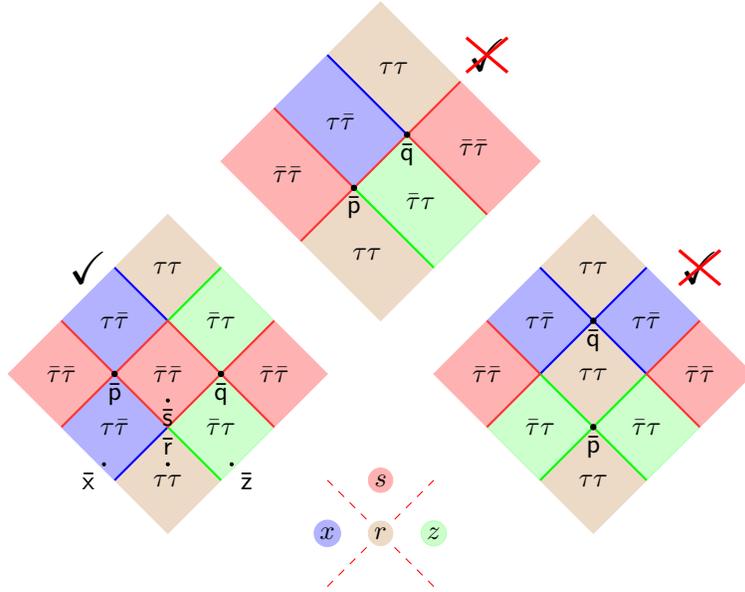

  To prove the other direction, let us assume that $\vp\sleq\vq$ does
  not hold and prove that $\ettos(\vp, \vq)$ does not hold. Then there
  are three cases to consider: either $\vp=\vq$ or $\vp\lleq \vq$ or
  $\vp\tleq \vq$.

  If $\vp\tleq \vq$, then, without loosing generality, we can assume
  that $\vq$ is in the timelike future of $\vp$ and that $\vp$ and
  $\vq$ are vertically related by Proposition~\ref{prop-aut}. In this
  case, we are going to show that there are no appropriate points
  $\vr,\vx,\vs,\vz\in\Q^n$. Even though Figure~\ref{fig-ettos} is
  two-dimensional, it also illustrates the higher-dimensional cases
  well because of rotational symmetry with respect to the
  time-axis. So we are going to use Figure~\ref{fig-ettos} to refer
  the regions having appropriate relations with respect to $\vp$ and
  $\vq$.

  Relation $\ettos(\vp,\vq)$ requires $\vr$ to be in one of the three
  brown $\tleq\tleq$ regions. Now we are going to show that choosing
  $\vr$ from any of these three regions leads to contradiction.
  \begin{itemize}
    \item If $\vr$ is in the bottom brown $\tleq\tleq$ region, then
      there is no $\vx$ from a blue $\tleq \tlneq$ region, which is
      $\tlneq$-related to $\vr$. This is so because of the followings.
      Such an $\vx$ should be in the timelike future of $\vp$ because,
      if it was in the timelike past of $\vp$, then (by transitivity)
      it would also be in the timelike past of $\vq$ contradicting
      that $\vx\tlneq\vq$. Thus, since $\vp$ is in the timelike future
      of $\vr$, by the transitivity, $\vx$ should also be in the
      timelike future of $\vr$; and hence relation $\vr\tlneq\vx$
      cannot hold.
    \item If $\vr$ is in the middle brown $\tleq\tleq$ region, then
      there is no $\vs$ from a red $\tlneq \tlneq$ region which is
      $\tleq$-related to $\vr$. This is so because, if $\vs$ is
      timelike related to $\vr$, then it is also timelike related to
      $\vp$ (if $\vs$ is in the future of $\vr$) or timelike related
      to $\vq$ (if $\vs$ is in the past of $\vr$). Thus $\vs$ cannot
      be $\tlneq$-related to both $\vp$ and $\vq$; and hence it cannot
      be from a red $\tlneq \tlneq$ region.
    \item If $\vr$ is in the upper brown $\tleq\tleq$ region, then
      there is no $\vz$ from a green $\tlneq \tleq$ region, which is
      $\tlneq$-related to $\vr$ for a completely analogous reason as
      in the first case.
   \end{itemize}
  Hence there is no appropriate $\vr$ required by formula
  $\ettos$. Therefore, $\ettos(\vp,\vq)$ cannot hold if $\vp$ and
  $\vq$ are timelike related.

  If $\vp\lleq \vq$, then basically the same argument works but it is
  simpler because there is no middle brown $\tleq\tleq$ region, \cf
  Figure~\ref{fig-ettos}.  Here, instead of the transitivity of the
  timelike past and future relations, we should use the fact that the
  timelike past of points from the causal past of point $\vq$ is in
  the timelike past of $\vq$, and the analogous fact that we get from
  this one if we replace past with future and $\vq$ with $\vp$. If
  $\vp=\vq$, then the situation is even simpler because, then there
  are no green $\tleq \tlneq$ and blue $\tlneq \tleq$ regions at
  all. So, if $\vp\sleq\vq$ does not hold, then $\ettos(\vp,\vq)$ does
  not hold either. This completes the proof of this direction and
  hence the proof of Item~\eqref{e2ttos}.

  Item \eqref{u2ttol} follows from Item~\eqref{e2ttos} and the fact
  that exactly one of relations $\vp\tleq\vq$, $\vp\lleq\vq$,
  $\vp\sleq\vq$, and $\vp=\vq$ holds for every $\vp,\vq\in\Q^n$.
\end{proof}

Because map $\alpha: (\nt,\nx)\mapsto (\nx,\nt)$ is an isomorphism
between structures $(\Q^2,\tleq,\sleq)$ and $(\Q^2,\sleq,\tleq)$,
existential formula
\begin{equation*}
  \estot(p,q) \defeq \exists r\exists x\exists s \exists z(r\sleq\sleq
  pq \kaj x \sleq \slneq pq \kaj s \slneq \slneq pq\kaj z \slneq \sleq
  pq\kaj r\slneq\sleq\slneq xsz)
\end{equation*}
defines timelike relatedness $\tleq$ from spacelike relatedness
$\sleq$ in $(\Q^2,\sleq)$, and hence,
\begin{equation*}
  \ustol(p,q) \defeq \lnot\estot(p,q) \kaj p\neq q \kaj p\slneq q  
\end{equation*}
gives a universal definition of lightlike relatedness from $\sleq$ in
$(\Q^2,\sleq)$. In other words, following is an immediate corollary of
Theorem~\ref{thm-edef2}.

\begin{corollary}[$n=2$]\label{cor-e2stot}
  Let $(\Q,+,\cdot,\le)$ be an arbitrary ordered field, then:
  \begin{enumerate}
  \item\label{e2stot} Timelike relatedness $\tleq$ can be defined
    existentially from spacelike relatedness $\sleq$ using only 6
    variables by $\exists^4$-formula $\estot$ in $(\Q^2,\sleq)$.
  \item\label{u2stol} Lightlike relatedness $\lleq$ can be defined
    universally from spacelike relatedness $\sleq$ using only 6
    variables by $\forall^4$-formula $\ustol$  in $(\Q^2,\sleq)$. \qed
  \end{enumerate}
\end{corollary}

\begin{remark}[$n\ge3$]
The assumption $n=2$ cannot be omitted from
Corollary~\ref{cor-e2stot}, \ie $\exists^4$-formula $\estot$ does not
define $\tleq$ in $(\Q^n,\sleq)$ if $n\ge3$. Spacelike related points
$\vp$ and $\vq$ satisfying $\estot(p,q)$ can easily be found searching
them in horizontal slices of $(\Q^n,\sleq)$ looked from above, \cf
Figure~\ref{fig-estot}. That $\estot$ can also be satisfied by
lightlike related points in $(\Q^n,\sleq)$ if $n\ge3$ can be shown by
checking the following example: $\vp=(-2,-2,0)$, $\vs=(0,0,0)$,
$\vq=(2,2,0)$, $\vx=(-2,0,0)$, $\vz=(2,0,0)$ and $\vr=(0,0,1)$. Hence,
$\forall^4$-formula $\ustol$ does not define $\lleq$ if $n\ge 3$
because the relation defined by it does not contain lightlike related
point pair $\vp=(-2,-2,0)$ and $\vq=(2,2,0)$.
  \begin{figure}[!hbt]
    \centering \input{eStoT.tikz}
    \caption{This figure illustrates why $\exists^4$-formula $\estot$
      does not define timelike relatedness $\tleq$ in $(\Q^n,\sleq)$
      if $n\ge3$. It shows a horizontal slice of $(\Q^3,\sleq)$ viewed
      from above together with the various regions of this plane
      colored and labeled how the points from there are related to the
      horizontally related points $\vp=(0,-2,0)$ and $\vq=(0,2,0)$
      below this plane. Points $\vx=(3,2,-2)$, $\vs=(3,0,1)$ and
      $\vz=(3,-2,-2)$ are in this horizontal plane, points $\vp$,
      $\vq$ and $\vr=(0,0,-3)$ are below in a parallel plane. Checking
      that these points show that spacelike related points $\vp$ and
      $\vq$ satisfy $\estot(p,q)$ is straightforward both by
      calculation and by the figure using the symmetries of the
      construction. \label{fig-estot}}
  \end{figure}
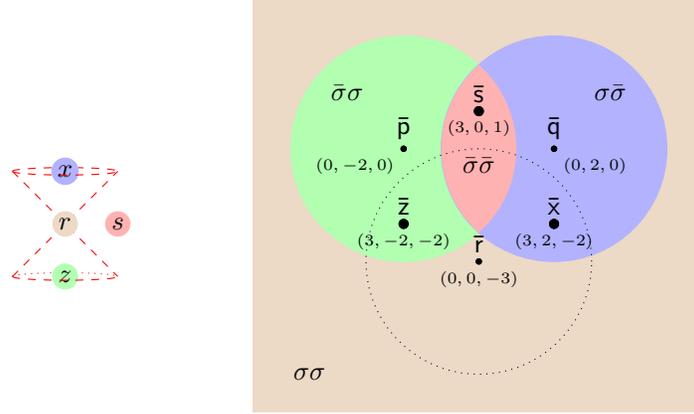
\end{remark}

Using $p$ and $q$ as free variables, up to logical equivalence, any
$\exists^k$-existential definition of a binary relation from a binary
relation $\aleq$ is of the following form $$\exists z_1\exists
z_2\ldots\exists z_k\; B(p,q,z_1,\ldots,z_k),$$ where $B$ is some
Boolean combination of relations $\aleq$ and $=$ between variables
$z_1,\ldots,z_k$, $p$ and $q$. Because $B$ can be written as a
disjunctive normal form and $\exists z (\varphi \lor \psi)$ is
logically equivalent to $(\exists z \varphi)\lor (\exists z \psi)$,
every such definition is equivalent to a disjunction
$\varphi_1\lor\ldots\lor\varphi_m$ of existential formulas $\varphi_i$
each of which requires the existence of $k$-many points
$z_1,\ldots,z_k$ and some relations of $\aleq$, $\alneq$, $\neq$, and
$=$ between $z_1,\ldots,z_k$, $p$ and $q$.

Without loosing generality, we can assume that none of the
requirements between two variables in $\varphi_i$ is the equality
because, in that case, either the same can be defined by using
less variables or $\varphi(p,q)$ is equivalent to $p=q$. We can
also assume that these requirements are \emph{non-self-contradicting}
(\ie we do not require a direct contradiction between variables
$z_1,\ldots,z_k$, $p$ and $q$, \eg if we require $z_1\aleq p$, then we
do not also require $z_1\alneq p$) because then $\varphi_i$ clearly
defines the empty relation. This means that we can assume that between
any two variables either nothing or one of relations $\aleq$,
$\alneq$, $\neq$, $\aleq_{\tinyneq}$, and $\alneq_{\tinyneq}$ is
required.

So, to understand relations definable by $\exists^k$-existential
formulas, it is enough to understand the relations which are definable
by an $\exists^k$-existential formula of the form $\exists z_1\exists
z_2\ldots\exists z_k\; C(p,q,z_1,\ldots,z_k)$, where
$C(p,q,z_1,\ldots,z_k)$ is a non-self-contradicting conjunction of
relations $\aleq$, $\alneq$, $\neq$ between variables
$z_1,\ldots,z_k$, $p$ and $q$ (since every nonempty definable relation
is the union of ones definable by such formulas). Let us call such
formulas \semph{basic $\exists^k$-formulas}.\footnote{Note that
  everything which is definable by a basic $\exists^k$ formula
  $\psi$ can easily be defined by a basic $\exists^{k+1}$
  formula, \eg $\exists v \psi$ defines the same if variable $v$
    does not occur in $\psi$ at all.}  The relation defined by a
basic $\exists^k$-formula $\exists z_1\exists z_2\ldots\exists z_k\;
C(p,q)\land D(p,q,z_1,\ldots,z_k)$ is the intersection of relation
$C(p,q)$ and the relation defined by basic $\exists^k$-formula
$\exists z_1\exists z_2\ldots\exists z_k\; D(p,q,z_1,\ldots,z_k)$. We
are going to call a basic $\exists^k$-formula \semph{non-requiring} if
it does not require anything between its free variables. So every
relation definable by a basic $\exists^k$-formula can be constructed
using intersections from $\aleq$, $\alneq$, $\neq$ and a relation
definable by a non-requiring basic $\exists^k$-formula.

It is the most difficult to satisfy those basic $\exists^k$-formulas,
which require either $\aleq$ or $\alneq_{\tinyneq}$ between any two of
variables $z_1,\ldots,z_k$, $p$ and $q$. We are going to call such
basic $\exists^k$-formulas \semph{fastidious}. By a non-requiring
fastidious basic $\exists^k$-formula, we mean one that requires
either $\aleq$ or $\alneq_{\tinyneq}$ between any two
  variables except $p$ and $q$, and there is no requirement
  between $p$ and $q$. To gain a picturesque understanding of their
satisfiability, we introduce mappings of some edge-labeled graphs to
$\Q^n$.

Let $G=(V,E)$ be a simple graph whose edges are labeled by some
labeling $l$ that maps the edges of $G$ to a set of symmetric binary
relations on $\Q^n$. By an \semph{embedding} \emph{of $(G,l)$ to
  $\Q^n$}, we understand a one-to-one map $f:V\to\Q^n$ such that, for
all vertices $x,y\in V$, if edge $(x,y)\in E$ is labeled by relation
$\aleq$, then $f(x)\aleq f(y)$ holds. We say that edge-labeled graph
$(G,l)$ is \semph{embeddable} \emph{to $\Q^n$} if there is an
embedding of $(G,l)$ to $\Q^n$.

To every basic $\exists^k$-formula $\varphi$, we can associate an
edge-labeled simple graph whose vertices are the variables of
$\varphi$, two variables $x$ and $y$ are connected by an edge if
$\varphi$ requires something between variables $x$ and $y$, and this
edge $(x,y)$ is labeled by the binary relation that $\varphi$ requires
between $x$ and $y$.  A basic $\exists^k$-formula is satisfiable in
model $(\Q^n,\tleq)$ if{}f the corresponding edge-labeled graph is
embeddable to $\Q^n$.

Let us start with some simple observations about these graph
embeddings.  A restriction of an embedding to a subset $S\subseteq V$
of vertices is an embedding of the subgraph of $G$ induced by $S$. This
gives us the following:
\begin{description}
\item[O1] If an edge-labeled graph is embeddable to $\Q^n$, then so are
  its induced subgraphs.
\end{description}
Since it is more difficult to satisfy stronger requirements, we have:
\begin{description}
\item[O2] If $\aleq\subseteq\bleq$ are symmetric binary relations on
  $\Q^n$ and edge-labeled graph $(G,l)$ is embeddable to $\Q^n$, then
  so is the one which we get from $(G,l)$ by replacing labels $\aleq$
  with $\bleq$.
\end{description}
Since $(\Q^2,\tleq,\lleq,\sleq)$ can be embedded as a substructure to
$(\Q^n,\tleq,\lleq,\sleq)$ for all $n\ge2$, we have that:
\begin{description}
\item[O3] If an edge-labeled graph labeled by relations $\tleq$,
  $\tlneq_{\tinyneq}$, $\lleq$, $\sleq$ is embeddable to $\Q^2$, then
  so is to $\Q^n$ for all $n\ge2$.
\end{description}
Because map $\alpha:(\nt,\nx)\mapsto(\nx,\nt)$ preserves relation
$\lleq$ but interchanges $\tleq$ and $\sleq$ in
$(\Q^2,\tleq,\lleq,\sleq)$, we have that:
\begin{description}
\item[O4] If $(G,l)$ labeled by $\tleq$, $\lleq$ and $\sleq$ is
  embeddable to $\Q^2$, then so is the one that we get from $(G,l)$ by
  interchanging labels $\tleq$ and $\sleq$.
\end{description}

\begin{figure}[!ht]
  \centering
  \input{nrfe2-list.tikz}
\caption{The figure illustrates the edge-labeled graphs of all the
  non-requiring fastidious $\exists^2$-formulas and the relations what
  the corresponding formulas define if $n=2$ or $(\Q,+,\cdot,\le)$ is
  a Euclidean field.\label{fig-nrfe2}}
\end{figure}
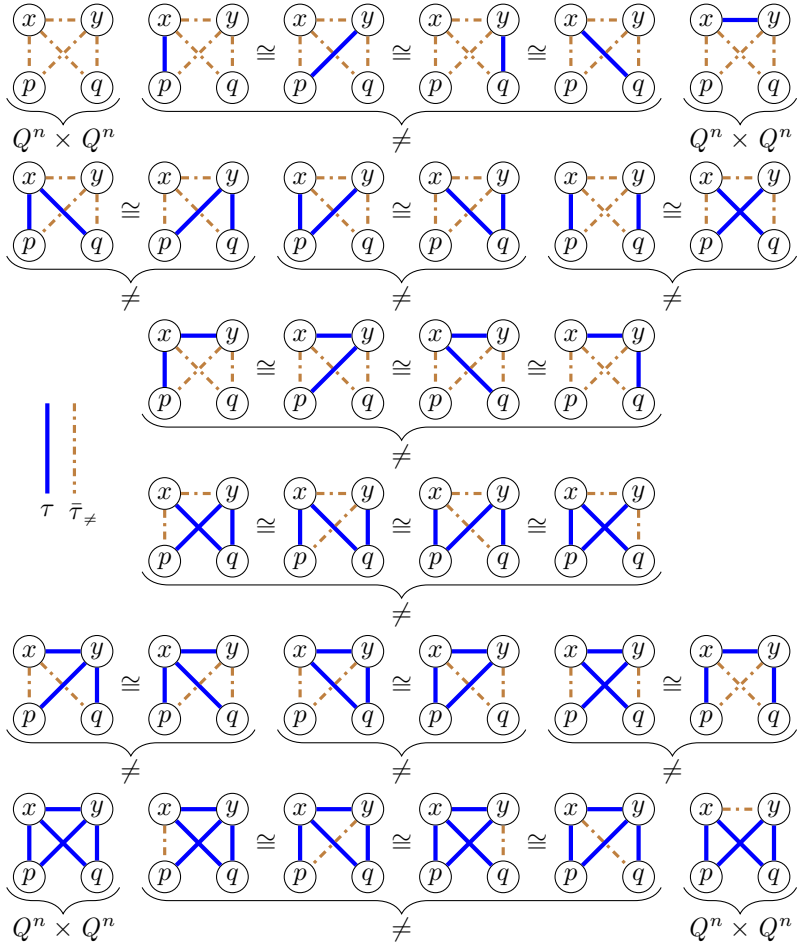

\begin{figure}[!ht]
  \centering \input{Embeddings.tikz}
\caption{The figure illustrates how one can satisfy any basic
  non-requiring fastidious $\exists^2$-formula $\varphi$ in
  $(\Q^2,\tleq)$ mapping its free variables to arbitrary two distinct
  points.\label{fig-Embeddings}}
\end{figure}
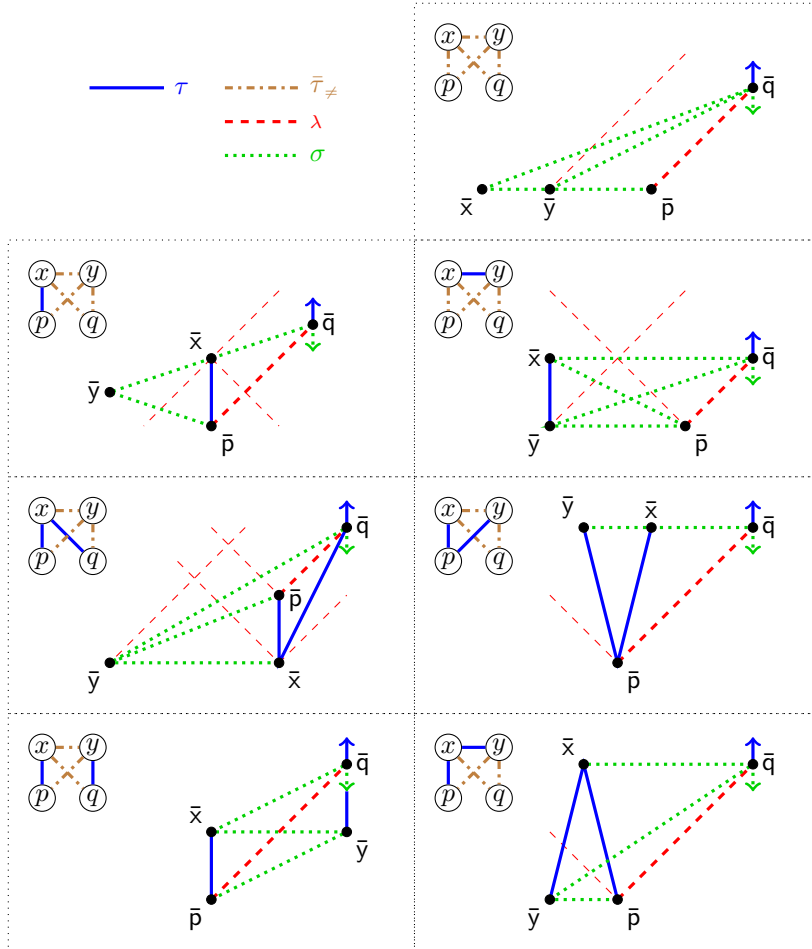

Now using these observations, we are going to show some properties of
relations definable from $\tleq$ or $\sleq$ by a non-requiring
$\exists^2$-formula.
\begin{lemma}[$n\ge2$]\label{lem-nre2}
  Let $(\Q,+,\cdot,\le)$ be an arbitrary ordered field, then:
\begin{enumerate}
  \item\label{l1} Every relation defined by a non-requiring basic
    $\exists^2$-formula contains timelike, spacelike and
    lightlike related pairs of points both in $(\Q^n,\tleq)$ and
    $(\Q^n,\sleq)$.
  \item\label{l2} If $n=2$ or
  $(\Q,+,\cdot,\le)$ is a Euclidean field, then every non-requiring
  basic $\exists^2$-formula defines either $\neq$ or the universal
  relation $\Q^n\times\Q^n$ both in $(\Q^n,\tleq)$ and $(\Q^n,\sleq)$.
\end{enumerate}
\end{lemma}

\begin{proof}
  The proofs of both items in the cases $(\Q^n,\tleq)$ and
  $(\Q^n,\sleq)$ are completely analogous by symmetries and because
  every construction used in the proof is two-dimensional. Therefore,
  we only show the statement in $(\Q^n,\tleq)$.  It is easy to check that
  the non-requiring fastidious basic $\exists^2$-formulas correspond
  exactly to the edge-labeled graphs listed in
  Figure~\ref{fig-nrfe2}. 

  Figure~\ref{fig-Embeddings} shows how the edge-labeled graphs
  corresponding to non-requiring fastidious basic $\exists^2$-formulas
  can be embedded to $\Q^2$ such that the $\tlneq_{\tinyneq}$ edges
  are mapped to spacelike related points and $p$ and $q$ are mapped to
  lightlike related if these formulas have less than three
  $\tleq$-requirements.  The constructions illustrated by
  Figure~\ref{fig-Embeddings} can easily be modified by
  rising/lowering $\vq$ a little bit to turn the relation between
  $\vp$ and $\vq$ into timelike/spacelike one without changing the
  timelike and spacelike relatedness between any other pair of points
  from $\vx$, $\vy$, $\vp$ and $\vq$.  In other words, we get an
  edge-labeled graph embeddable to $\Q^2$ from any of the graphs from
  the top half of Figure~\ref{fig-nrfe2} even if we strengthen the
  $\tlneq_{\tinyneq}$-labels to $\sleq$ and connect the vertices $p$
  and $q$ labeling them by any of relations $\tleq$, $\sleq$ and
  $\lleq$. So by observation O4, the same holds for the edge-labeled
  graphs from the bottom half of Figure~\ref{fig-nrfe2}.  By O3, this
  implies the statement of Item~\ref{l1} for relations definable by a
  non-requiring fastidious basic $\exists^2$-formula. From this,
  Item~\ref{l1} of the lemma follows because a non-requiring basic
  $\exists^2$-formula can only define something larger than its
  non-requiring fastidious strengthenings.

  Now Item~\ref{l2} follows from Item~\ref{l1} by
  Corollary~\ref{cor-aut} and the fact that no relation between $\neq$
  and $\Q^n\times\Q^n$ is definable from $\tleq$ (as no relation
  between $\emptyset$ and $=$ is definable from $\tleq$).
\end{proof}

\begin{theorem}[$n\ge2$]\label{cor-nre2}
  Let $(\Q,+,\cdot,\le)$ be an arbitrary ordered field, then:
  \begin{enumerate}
  \item Neither spacelike nor lightlike relatedness can be defined
    existentially or universally using only 4 variables (\ie by an
    $\exists^2$-formula or a $\forall^2$-formula) in $(\Q^n,\tleq)$.
  \item Neither timelike nor lightlike relatedness can be defined
    existentially or universally using only 4 variables (\ie by an
    $\exists^2$-formula or a $\forall^2$-formula) in $(\Q^n,\sleq)$.
    \qed
  \end{enumerate}
\end{theorem}

\begin{proof}
  In both cases, it is enough to show the nonexistence of existential
  definitions because that implies the nonexistence of universal
  ones. This is so because exactly one of relations $\tleq$, $\lleq$,
  $\sleq$ and $=$ holds. Hence, for example, the negation of a
  universal definition of $\sleq$ defined from $\tleq$ can be turned
  into an existential definition of $\lleq$ from $\tleq$ by
  intersecting it with $\neq$ and $\tlneq$.

  By Item~\ref{l1} of Lemma~\ref{lem-nre2}, a relation defined from
  $\tleq$ by a non-requiring basic $\exists^2$-formula contains
  timelike, lightlike and spacelike pairs of points alike. So every
  relation which is definable from $\tleq$ by a basic
  $\exists^2$-formula contains a spacelike related point pair exactly
  if it contains a lightlike related one because every such relation
  is the intersections of a relation definable from $\tleq$ by a
  non-requiring basic $\exists^2$-formula and some of relations
  $\tleq$, $\tlneq$, $\neq$ and $=$. Since every relation definable by
  an $\exists^2$-formula is the union of relations definable by basic
  $\exists^2$-formulas, every relation which is definable from $\tleq$
  by a $\exists^2$-formula contains a spacelike related point pair
  exactly if it contains a lightlike related one. Consequently,
  neither $\sleq$ nor $\lleq$ is definable from $\tleq$ by an
  $\exists^2$-formula.

  The proof of that neither $\tleq$ nor $\lleq$ is definable from
  $\sleq$ by an $\exists^2$-formula is completely analogous.
\end{proof}

\noindent
If the $n=2$ or $(\Q,+,\cdot,\le)$ is a Euclidean field, then we can
prove more:

\begin{theorem}[$n\ge2$; \ref{eq-Eucl} or $n=2$]\label{thm-noedef2}
   Assume that $n=2$ or that $(\Q,+,\cdot,\le)$ is a Euclidean field.
   Then we have the followings:
  \begin{enumerate}
  \item\label{item-noe3tldef} If a binary relation $R$ is definable
    existentially or universally using only 4 variables (\ie by an
    $\exists^2$-formula or a $\forall^2$-formula) in $(\Q^n,\tleq)$,
    then $R$ is also definable by a quantifier free formula; in other
    words, $R$ has to be one of the $8$ relations built up from
    $\tleq$ and $=$ by Boolean operations, \cf Figure~\ref{fig-BA}.
  \item\label{item-noe3sldef} If a binary relation $R$ is definable
    existentially or universally using only 4 variables (\ie by an
    $\exists^2$-formula or a $\forall^2$-formula) in $(\Q^n,\sleq)$,
    then $R$ is also definable by a quantifier free formula; in other
    words, $R$ has to be one of the $8$ relations built up from
    $\sleq$ and $=$ by Boolean operations, \cf Figure~\ref{fig-BA}.
  \end{enumerate}
\end{theorem}

\begin{proof}
  Because the negation of a quantifier free formula is also quantifier
  free and every $\forall^2$-formula is equivalent to the negation of
  a $\exists^2$-formula, it is enough to show the existential case of
  both items. We are going to prove the two items simultaneously. Let
  $\aleq$ be one of the relations $\tleq$ and $\sleq$.

  Since relations definable by $\exists^2$-formulas are unions of ones
  definable by basic $\exists^2$-formulas, it is enough to see that
  the statement holds for basic $\exists^2$-formulas. Since every
  relation definable by basic $\exists^2$-formula $\varphi$ is the
  intersection of a relation definable by a non-requiring basic
  $\exists^2$-formula and the Boolean-definable relation that
  $\varphi$ requires between its free variables, it is enough to see
  that every relation definable by a non-requiring basic
  $\exists^2$-formula is also Boolean definable from $\aleq$ and
  $=$. This follows by Item~\ref{l2} of Lemma~\ref{lem-nre2}, which
  completes the proof of the theorem.
\end{proof}

Using this colorful graph embedding perspective, we are going to show
that the following $\exists^3$-formula defines $\sleq$ from $\tleq$ in
$(\Q^n,\tleq)$:
\begin{equation*}
  \newettos(p,q)  \defeq\exists x\exists y\exists z
    (x\tleq\tleq\tlneq_{\tinyneq}\tlneq_{\tinyneq} pyqz \kaj y
    \tlneq_{\tinyneq}\tlneq_{\tinyneq}\tleq pqz \kaj
    z\tlneq_{\tinyneq}\tleq pq).
\end{equation*}
From this, we will see that the following $\forall^3$-formula defines
$\lleq$ from $\tleq$ in $(\Q^n,\tleq)$:
\begin{equation*}
\newuttol(p,q) \defeq \lnot\newettos(p,q) \kaj p\neq q \kaj p\tlneq q.
\end{equation*}

\begin{theorem}[$n\ge2$; \ref{eq-Eucl} or $n=2$]\label{thm-newedef2}
  Assume that $n=2$ or that $(\Q,+,\cdot,\le)$ is a Euclidean field,
  then:
  \begin{enumerate}
  \item\label{newe2ttos} Spacelike relatedness $\sleq$ can be defined
    existentially from timelike relatedness $\tleq$ using only 5
    variables by $\exists^3$-formula $\newettos$ in $(\Q^n,\tleq)$.
  \item\label{newu2ttol} Lightlike relatedness $\lleq$ can be defined
    universally from timelike relatedness $\tleq$ using only 5
    variables by $\forall^3$-formula $\newuttol$ in $(\Q^n,\tleq)$.
  \end{enumerate}
\end{theorem}

\begin{proof}
  It is enough to prove Item~\ref{newe2ttos} because
  Item~\ref{newu2ttol} follows from Item~\ref{newe2ttos} and the fact
  that exactly one of relations $\vp\tleq\vq$, $\vp\lleq\vq$,
  $\vp\sleq\vq$, and $\vp=\vq$ holds for every $\vp,\vq\in\Q^n$.

  To prove Item~\ref{newe2ttos}, let us assume first that
  $\newettos(\vp,\vq)$ holds for some $\vp,\vq\in\Q^n$ and show that
  then $\vp$ and $\vq$ has to be spacelike related. Let $\vx$, $\vy$,
  and $\vz$ be such points that they show the validity of
  $\newettos(\vp,\vq)$. Clearly, $\vp$ cannot be equal to $\vq$ since
  $\vz\tlneq_{\tinyneq} \vp$ but $\vz\tleq \vq$.

  Let us now show that points $\vp$ and $\vq$ cannot be timelike or
  lightlike related. If they were such, we could assume, without
  loosing generality, that $\vq$ is in the causal past of
  $\vp$. Then $\vz$ should be in the timelike future of $\vq$ otherwise
  $\vz$ and $\vp$ were timelike related contrary to the requirement
  $z\tlneq_{\tinyneq}p$ of $\newettos$. For analogous reasons, the
  relation being in the future/past should alternate between the
  consecutive points along the circle $\vq-\vz-\vy-\vx-\vp-\vq$ which
  is not possible because there are odd many edges between them, \ie
  $\vy$ has to be in the past of $\vz$, $\vx$ has to be in the future
  of $\vy$, $\vp$ has to be in the past of $\vx$, and finally $\vq$
  has to be in the future of $\vp$ contradicting that $\vq$ is in the
  past of $\vp$. Hence $\vp$ and $\vq$ has to be spacelike related.
  
  To show the other direction, let $\vp$ and $\vq$ be two arbitrary
  spacelike related points. We can assume that $\vp$ and $\vq$ are the
  horizontally related, in the \txplane by the automorphism argument,
  \cf Proposition~\ref{prop-aut}. Then we can choose $\vy$ to be the
  midpoint of $\vp$ and $\vq$, and choose $\vx$ and $\vz$ respectively
  above $\vp$ and $\vq$ high enough but not too high, \cf
  Figure~\ref{fig-neweTtoS}.
  \begin{figure}
    \centering \input{neweTtoS.tikz}
    \caption{The figure illustrates the proof of that the relation
      defined by $\exists^3$-formula $\newettos$ defines $\sleq$ if
      $n=2$ or $(\Q,+,\cdot,\le)$ is a Euclidean
      field.\label{fig-neweTtoS}}
  \end{figure}
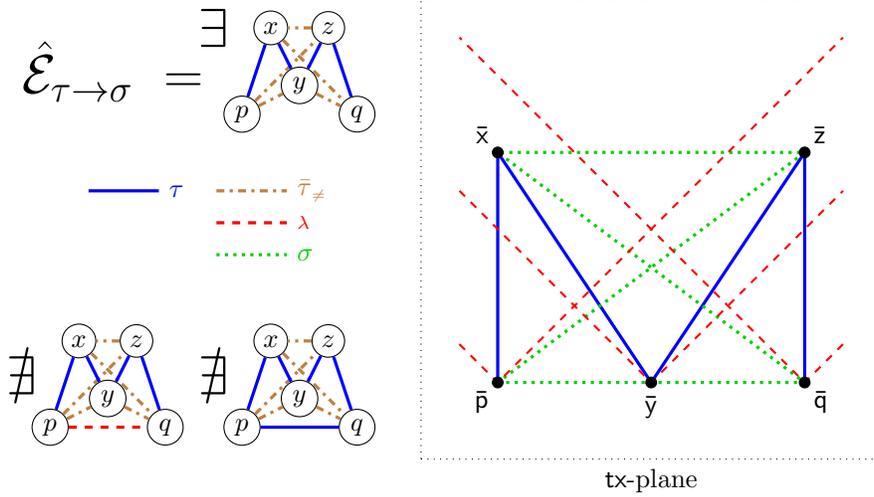
\end{proof}

\begin{remark}[$n\ge3$]
  Regarding Theorem~\ref{thm-newedef2}, the proof of the direction
  showing that $\sleq$ contains the relation defined by $\newettos$
  works over arbitrary ordered field.  The other direction can be
  generalized to Archimedean ordered fields using the facts that (a)
  every Archimedean ordered field is isomorphic to a subfield of the
  reals and (b) every subfield of the reals is dense in the field of
  reals. However, this proof idea does not work in general, and hence,
  we do not know whether it is possible that $\newettos$ defines a
  nonempty relation strictly smaller than $\sleq$ in some
  non-Archimedean ordered fields.
\end{remark}

Let us now replace $\tleq$ with $\sleq$ in formulas $\newettos$ and
$\newuttol$ and see when do they work as respective definitions of
$\tleq$ and $\lleq$. So let
\begin{equation*}
  \newestot(p,q)  \defeq  \exists x\exists y\exists z
    (x\sleq\sleq\slneq_{\tinyneq}\slneq_{\tinyneq} pyqz \kaj y
    \slneq_{\tinyneq}\slneq_{\tinyneq}\sleq pqz \kaj
    z\slneq_{\tinyneq}\sleq pq)
\end{equation*}
and let
\begin{equation*}
\newustol(p,q) \defeq \lnot\newestot(p,q) \kaj p\neq q \kaj p\slneq q.
\end{equation*}

As before, by the isomorphism between $(\Q^2,\tleq,\sleq)$ and
$(\Q^2,\sleq,\tleq)$, we get the following corollary of
Theorem~\ref{thm-newedef2}.

\begin{corollary}[$n=2$]\label{cor-newe2stot}
  Let $(\Q,+,\cdot,\le)$ be an arbitrary ordered field, then:
  \begin{enumerate}
  \item\label{newe2stot} Timelike relatedness $\tleq$ can be defined
    existentially from spacelike relatedness $\sleq$ using only 5
    variables by $\exists^3$-formula $\newestot$ in $(\Q^2,\sleq)$.
  \item\label{newu2stol} Lightlike relatedness $\lleq$ can be defined
    universally from spacelike relatedness $\sleq$ using only 5
    variables by $\forall^3$-formula $\newustol$  in $(\Q^2,\sleq)$. \qed
  \end{enumerate}
\end{corollary}

\begin{remark}[$n\ge3$]
The assumption $n=2$ cannot be omitted from
Corollary~\ref{cor-newe2stot}, \ie $\exists^3$-formula $\newestot$
does not define $\tleq$ in $(\Q^n,\sleq)$ if $n\ge3$. It is enough to
see this in case $n=3$ since $(\Q^3,\sleq)$ can be embedded as a
submodel to $(\Q^n,\sleq)$ for all $n\ge3$. To show it for $n=3$,
consider the following lightlike related points $\vp=(-2,-2,0)$ and
$\vq=(2,2,0)$. It is straightforward to check that points
$\vx=(-2,0,3)$, $\vy=(0,0,0)$ and $\vz=(2,0,3)$ are such that they
show the validity of $\newestot(\vp,\vq)$, see
Figure~\ref{fig-newestot} for a picturesque justification.  Hence,
$\forall^4$-formula $\newustol$ does not define $\lleq$ if $n\ge 3$
because the relation defined by it does not contain lightlike related
point pair $\vp=(-2,-2,0)$ and $\vq=(2,2,0)$.
  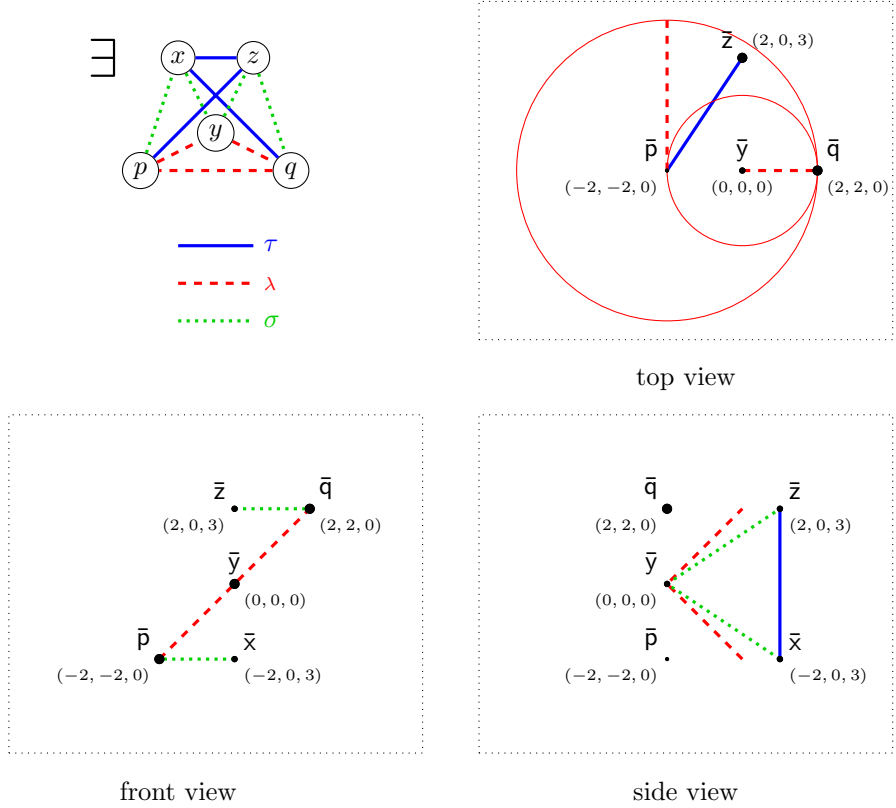
\begin{figure}[!hbt]
    \centering \input{neweStoT.tikz}
    \caption{This figure illustrates why $\exists^3$-formula
      $\newestot$ does not define timelike relatedness $\tleq$ in
      $(\Q^n,\sleq)$ if $n\ge3$.\label{fig-newestot}}
  \end{figure}
\end{remark}

\begin{theorem}[$n\ge2$]\label{thm-noedef}
  Let $n\ge 2$.  Relation $\lleq$ is not definable existentially from
  $\tleq$ in $(\Q^n,\tleq)$. Similarly, $\lleq$ is not definable
  existentially from $\sleq$ in $(\Q^n,\sleq)$. Thus, $\sleq$ is not
  definable universally from $\tleq$ in $(\Q^n,\tleq)$ and $\tleq$ is
  not definable universally from $\sleq$ in $(\Q^n,\sleq)$.
\end{theorem}

\begin{proof}
  Let formula $\varphi(x,y)$ be an arbitrary existential formula in
  the language of $(\Q^n,\tleq)$. Then $\varphi(x,y)$ is logically
  equivalent to a formula of the following form $$\exists z_1\exists
  z_2\ldots\exists z_k\; B(x,y,z_1,\ldots,z_k)$$ such that $B$ is a
  Boolean combination of relations $\tleq$ and $=$ between
    variables $z_1,\ldots,z_k$, $x$ and $y$.

  We are going to show that if the relation defined by $\varphi$ holds
  for some lightlike related points $\vx$ and $\vy$, then it also
  holds for some spacelike related points $\vx'$ and $\vy'$, and hence
  $\varphi$ cannot be a definition of lightlike relatedness.
  To see this, let $\vx$ and $\vy$ be lightlike related points of
  $\Q^n$ for which $\varphi(\vx,\vy)$ holds. Then there are points
  $\vz_1,\ldots,\vz_k\in\Q^n$ such that $B(\vx,\vy,\vz_1,\ldots,\vz_k)$
  holds. There is a small enough $\varepsilon>0$ such that map
  $$T_\varepsilon:(\nr_0,\nr_1,\ldots,\nr_{n-1})\mapsto
  ((1-\varepsilon)\nr_0,\nr_1,\ldots,\nr_{n-1})$$ scaling time down
  does not change the $\tleq$, $\tlneq$, $=$ and $\neq$ relations
  between points $\vz_1,\ldots,\vz_k$, $\vx$ and $\vy$. This is so
  because of the followings. For fixed timelike related points $\vp$
  and $\vq$, there is a $\delta>0$ such that $T_\varepsilon(\vp)$ and
  $T_\varepsilon(\vq)$ are also timelike related for all
  $0<\varepsilon<\delta$. Since there are only finitely many pairs of
  points to consider, there is an appropriate $\varepsilon$ preserving
  the timelike relatedness relations between points $\vx$, $\vy$,
  $\vz_1,\ldots$, $\vz_k$. Let us fix such an $\varepsilon$. Then map
  $T_\varepsilon$ takes $\tlneq$-related points to $\tlneq$-related
  ones because it decreases the time difference but does not change
  the spatial distance. Also $T_\varepsilon$ maps different points to
  different ones because it is a bijection. Let $\vz_1',\ldots$,
  $\vz_k'$, $\vx'$ and $\vy'$ be the $T_\varepsilon$-images of points
  $\vz_1,\ldots,\vz_k$, $\vx$, and $\vy$, respectively. Then
  $B(\vx',\vy',\vz_1',\ldots,\vz_k')$ holds by the above properties of
  $T_\varepsilon$, and hence $\varphi(\vx',\vy')$ also holds. We also
  have that $\vx'$ and $\vy'$ are spacelike related because
  $T_\varepsilon$ decreases the time difference but does not change
  the spatial distance.  This completes the proof of that relation
  $\lleq$ is not definable existentially from $\tleq$ in
  $(\Q^n,\tleq)$.

  The proof of that $\lleq$ is not definable existentially from
  $\sleq$ in $(\Q^n,\sleq)$ is completely analogous but using a map
  scaling time up.

  Finally, that $\sleq$ is not definable universally from $\tleq$ in
  $(\Q^n,\tleq)$ and $\tleq$ is not definable universally from $\sleq$
  in $(\Q^n,\sleq)$ follows as before from that exactly one of
  relations $\tleq$, $\lleq$, $\sleq$ and $=$ holds between any two
  points.
\end{proof}

\begin{theorem}[$n\ge 2$]\label{thm-noedefll}
  Let $n\ge 2$.  Neither timelike relatedness $\tleq$ nor spacelike
  relatedness $\sleq$ is definable existentially or universally from
  lightlike relatedness $\lleq$ in $(\Q^n,\lleq)$.
\end{theorem}

We have already seen that Theorem~\ref{thm-noedefll} holds in case
$n=2$, see Remark~\ref{rem-2d}. Now we are going to give a proof that
works for all $n\ge2$.

\begin{proof}
  As in the proof of Theorem~\ref{cor-nre2} and for the very same
    reason, is enough to show the nonexistence of existential
    definitions.
  To show the nonexistence of existential definitions, let us consider
  the following map:
  $$h:(\nr_0,\nr_1,\ldots,\nr_{n-1})\mapsto
  \frac{1}{\nr_0^2-\nr_1^2-\ldots -\nr_{n-1}^2}
  (\nr_0,\nr_1,\ldots,\nr_{n-1}).$$
  It is known, see \eg \cite{Lester}, and also straightforward to
  check that $h$ is a bijection of the complement of the light cone
  through the origin $\vo=(0,0,\ldots,0)$ to itself and $h$ maps
  lightlike related pairs of points to lightlike related ones. By the
  definition of $h$, we also have that $\vr$, $\vo$ and $h(\vr)$ are
  on the same line; $h$ leaves the hyperboloid defined by equation
  $\nr_0^2-\nr_1^2-\ldots -\nr_{n-1}^2=1$ pointwise fixed; $h$ maps
  every point of the hyperboloid defined by equation
  $\nr_0^2-\nr_1^2-\ldots -\nr_{n-1}^2=-1$ to its opposite; and $h$
  maps the region inside these hyperboloids to the region outside of
  them.

  \begin{figure}[!htb]
    \centering \input{HyperbInv.tikz}
    \caption{The figure illustrates that $h$ interchanges the
      relations $\tleq$ and $\sleq$ between points $\vp$ and $\vq$ if
      they are from the appropriate regions given in the proof of
      Theorem~\ref{thm-noedefll}.\label{fig-HyperbInv}}
  \end{figure}
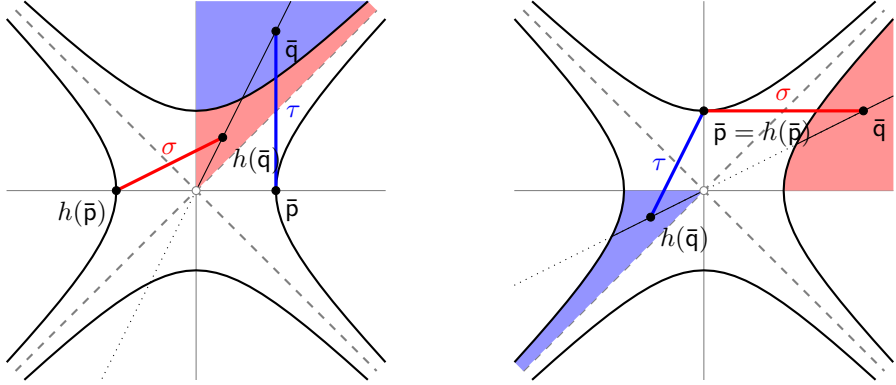
  
  Let  $\varphi(p,q)$ be an arbitrary existential formula in
  the language of $(\Q^n,\lleq)$. Then $\varphi(p,q)$ is of the form
  $\exists z_1\exists z_2\ldots\exists z_k\; B(p,q,z_1,\ldots,z_k)$,
  where $B$ is a Boolean combination of relations $\lleq$ and $=$
  between variables $z_1,\ldots,z_k$, $p$ and $q$.

  Using the fact that $h$ is a bijection preserving $\lleq$, we are
  going to show that the relation defined by $\varphi$ holds for some
  timelike related points if{}f it holds for some spacelike related
  points. So $\varphi$ can define neither $\tleq$ nor $\sleq$.

  Let $\vp$ and $\vq$ be points such that $\varphi(\vp,\vq)$
  holds. Then there are points $\vz_1,\ldots,\vz_k$ such that
  $B(\vp,\vq,\vz_1,\ldots,\vz_k)$ holds. Using automorphisms of
  $(\Q^n,\lleq)$, we are going to show that we can assume, without
  loosing generality, that
  \begin{enumerate}
  \item\label{unu} points $\vz_1,\ldots,\vz_k$, $\vp$, and $\vq$ are
    all in the domain of $h$,
  \item\label{du} if $\vp\tleq\vq$, then $\vp=(0,1,0,\ldots,0)$ and
    there is a $\nt>0$ such that $\vq=(\nt,1,0,\ldots,0)$ and
    $\nt^2-1>1$ (\ie $\vq$ is above the upper half of the hyperbola
    $\nt^2-\nx^2=1$ in the \txplane),
  \item\label{tri} if $\vp\sleq\vq$, then $\vp=(1,0,\ldots,0)$ and
    there is an $\nx>0$ such that $\vq=(1,\nx,0,\ldots,0)$ and
    $1-\nx^2<-1$ (\ie $\vq$ is on the right of the right half of the
    hyperbola $\nt^2-\nx^2=-1$ in the \txplane).
  \end{enumerate}
  In this setting, the $h$-images of $\vp$ and $\vq$ are timelike if
  $\vp$ and $\vq$ are spacelike related, and they are spacelike
  related if $\vp$ and $\vq$ are timelike related, see
  Figure~\ref{fig-HyperbInv}. Since $h$ preserves the $\lleq$,
  $\llneq$, $=$ and $\neq$ relations between points
  $\vz_1,\ldots,\vz_k$, $\vp$, and $\vq$, this shows that there are
  timelike related points in $\varphi$ relation exactly if there are
  spacelike related ones in $\varphi$ relation.
  
  The only remaining thing to show is that points
  $\vz_1,\ldots,\vz_k$, $\vp$, and $\vq$ can indeed be transformed by
  an automorphism of $(\Q^n,\lleq)$ satisfying \ref{unu}, \ref{du} and
  \ref{tri}. By Proposition~\ref{prop-aut}, we can transform them such
  that $\vp$ and $\vq$ are situated appropriately in the \txplane. The
  only problem is that maybe some of points $\vz_1,\ldots,\vz_k$ are
  not in the domain of $h$. If that is so, then by applying an
  appropriate automorphism $\alpha$ enough many time, we can move
  everything into the domain of $h$ while keeping points $\vp$ and
  $\vq$ in the right places.

  In case $\vp\tleq\vq$, we can choose this automorphism $\alpha$ to
  be the composition of a uniform scaling (say by factor 2) and an
  appropriate horizontal translation (in the negative direction along
  the $\nx$-axis). Of course one such step may create new problematic
  points, but iterating this step will put every points to the right
  place after finitely many iterations because there are only finitely
  many points to deal with and a point cannot become problematic more
  than twice. The later is so because the sequence of points $\vz$,
  $\alpha(\vz)$, $\alpha\big(\alpha(\vz)\big)$, \ldots are different
  points on the same line, and a line cannot intersect the set of
  problematic points (\ie the light cone through the origin) more than
  twice.\footnote{More precisely, the considered automorphism $\alpha$
    maps point $\vz=(z_0,z_1,z_2,\ldots z_{n-1})$ to
    $(2z_0,2z_1-1,2z_2,\ldots, 2z_{n-1})$. Hence, points $\vz$,
    $\alpha(\vz)$ and $\alpha\big(\alpha(\vz)\big)$ are on the same
    line because
    $\alpha\big(\alpha(\vz)\big)-\alpha(\vz)=2\big(\alpha(\vz)-\vz\big)$,
    which can be checked by a straightforward calculation. Checking
    that this $\alpha$ keeps points $\vp$ and $\vq$ in the right
    places is also straightforward.}

  In case $\vp\sleq\vq$, the same idea works but here after scaling up
  we should translate the points downwards (\ie in the negative
  direction along the $\nt$-axis).
\end{proof}

Based on the definitions given by Winnie~\cite{winnie}, we get the
following formulas which respectively define $\lleq$ and $\tleq$ from
$\sleq$:
\begin{align*}
  \wstol(x,y) &\defeq x\slneq y\kaj x\neq y\kaj \forall u\forall v
  \exists z_u\exists z_v (z_u\slneq\sleq\sleq uxy \lor
  z_v\slneq\sleq\sleq vxy\lor u\slneq v)\\ \wstot(x,y) &\defeq \lnot
  \wstol(x,y) \kaj x\slneq y \kaj x\neq y.
\end{align*}

\begin{theorem}[$n\ge 2$; \ref{eq-Eucl} or $n=2$]\label{thm-w1}
  Assume that $n=2$ or that $(\Q,+,\cdot,\le)$ is a Euclidean field.
  Then in model $(\Q^n,\sleq)$, lightlike relatedness $\lleq$ can be
  defined from spacelike relatedness $\sleq$ using only 6 variables by
  $\forall^2\exists^2$-formula $\wstol$. Hence timelike relatedness
  $\tleq$ can be also defined from $\sleq$ using only 6 variables by
  $\exists^2\forall^2$-formula $\wstot$.
\end{theorem}

\begin{proof}
  Let us first show that if points $\vx$ and $\vy$ are lightlike
  related, then they satisfy formula $\wstol$. So let $\vx$ and $\vy$
  be arbitrary two lightlike related points. Since then $\vx$ and
  $\vy$ are clearly neither equal nor spacelike related, we only have
  to show that, for all points $\vu$ and $\vv$, there are points
  $\vz_u$ and $\vz_v$ such that $\vz_u\slneq\vu$,
  $\vz_u\sleq\vx$ and $\vz_u\sleq \vy$ hold, $\vz_v\slneq\vv$,
  $\vz_v\sleq\vx$ and $\vz_v\sleq \vy$ hold or $\vu$ and $\vv$ are not
  spacelike related. Unless $\vu$ is in the segment $\vx\vy$, there is
  an appropriate $\vz_u$. This can be checked by considering the
  horizontal hyperplanes through $\vx$ and $\vy$, see the left hand side
  of Figure~\ref{fig-w2}. The light cone through $\vu$ intersects these
  hyperplanes in two spheres (at most one of which is degenerate as
  $\vu$ cannot be in both hyperplanes). It is straightforward to check
  that there is a point $\vz_u$ inside one of these spheres which is
  spacelike related to both $\vx$ and $\vy$, unless $\vu$ is in the
  lightlike segment $\vx\vy$.
  \begin{figure}[!htb]
    \centering \input{Winnie2.tikz}
    \caption{\label{fig-w2}}
  \end{figure}
  Similarly, unless $\vv$ is in the segment $\vx\vy$, there is an
  appropriate $\vz_v$.  If both $\vu$ and $\vv$ are in the lightlike
  segment $\vx\vy$, then they are either equal or lightlike related,
  and hence $\vu\slneq\vv$ holds. So lightlike relatedness implies
  $\wstol$-relatedness.

  Let us now show that if points $\vx$ and $\vy$ are $\wstol$-related,
  then they have to be lightlike related. Clearly, $\vx$ and $\vy$
  have to be different and cannot be spacelike related. Thus, we only
  have to show that they cannot be timelike related. So let $\vx$ and
  $\vy$ be two arbitrary timelike related points.  We should show that
  they are not $\wstol$-related.

  By Proposition~\ref{prop-aut}, we can assume that $\vx$ and $\vy$
  are vertically related and $\vy$ is in the (causal) future of
  $\vx$. To show that they are not $\wstol$-related, we should find
  spacelike related points $\vu$ and $\vv$ such that at least one of
  $\vz_u\sleq\vu$, $\vz_u\slneq\vx$ and $\vz_u\slneq\vy$ holds and at
  least one of $\vz_v\sleq\vv$, $\vz_v\slneq\vx$ and $\vz_v\slneq\vy$
  holds for all points $\vz_u$ and $\vz_v$. Let $\vu$ and $\vv$ be
  spacelike related points such that they are both in the lightlike
  past of $\vy$ and in the lightlike future of $\vx$. There are such
  points, see the right hand side of Figure~\ref{fig-w2}. Then the
  causal future of $\vu$ is in the causal future of $\vx$ and the
  causal past of $\vu$ is in the causal past of $\vy$ by the
  transitivity of the causal past and future relations. Hence, if
  $\vz\slneq \vu$, then $\vz\slneq \vx$ or $\vz\slneq\vy$. So every
  point is $\slneq$-related to $\vx$ or $\vy$, or it is
  $\sleq$-related to $\vu$; and the same holds for $\vv$. Thus $\vx$
  and $\vy$ are not $\wstol$-related, and this is what we wanted to
  prove.

  From this, it follows that $\wstot$ defines timelike relatedness as
  in the previous proofs because exactly one of relations $\tleq$,
  $\lleq$, $\sleq$ and $=$ holds.  
\end{proof}

Let $\wstols$ and $\wstots$ be the formulas that we respectively get
from $\wstol$ and $\wstot$ when replacing $\sleq$ with $\tleq$. By the
isomorphism of $(\Q^2,\tleq,\lleq)$ and $(\Q^2,\sleq,\lleq)$, the
following is a corollary of Theorem~\ref{thm-w1}:

\begin{corollary}[$n=2$]\label{thm-w2}
  Let $(\Q,+,\cdot,\le)$ be an arbitrary ordered field.
  Then in model $(\Q^2,\tleq)$, lightlike relatedness $\lleq$ can be
  defined from timelike relatedness $\tleq$ using only 6 variables by
  $\forall^2\exists^2$-formula $\wstols$. Hence spacelike relatedness
  $\sleq$ can be also defined from $\tleq$ using only 6 variables by
  $\exists^2\forall^2$-formula $\wstots$.
\end{corollary}

\begin{remark}
   Even if the field $(\Q,+,\cdot,\le)$ is Euclidean, the statement of
   Corollary~\ref{thm-w2} does not hold if the dimension $n>2$. In
   this case, formulas $\wstols$ and $\wstots$ respectively define
   relations $\tlneq_{\tinyneq}$ and $\emptyset$. This is so because
   formula $$\phi(x,y):=\forall u\forall v \exists z_u\exists z_v
   (z_u\tlneq\tleq\tleq uxy \lor z_v\tlneq\tleq\tleq vxy\lor u\tlneq
   v)$$ defines $\tlneq$, which is the union of relations $\sleq$,
   $\lleq$ and $=$. That the relation defined by $\phi$ contains
   $\sigma$ can be shown based on the following observation: If point
   $\vu$ is not simultaneous to horizontally related points $\vx$ and
   $\vy$, then after (or before) a certain time the horizontal slices
   of the light cone through $\vu$ do not cover the corresponding
   horizontal slices of the intersection of the light cones through
   $\vx$ and $\vy$, and hence there is an appropriate point $\vz_u$
   for which $\vz_u\tlneq\tleq\tleq \vu\vx\vy$ holds.
\end{remark}

\section{Concluding Remarks and Open Problems}
\label{sec-open}

We have carefully investigated the interdefinability between timelike,
lightlike and spacelike relatedness of Minkowski spacetime in various
dimensions over Euclidean fields and in some cases over arbitrary
ordered fields. We aimed to find the simplest possible definitions. We
have shown that in terms of number of variables 4 is minimal, but in
terms of quantifier complexity, we left some natural questions open.

For all $n\ge2$, it is open if universal-existential definitions
$\wstots$ and $\wstot$ (and existential-universal definitions
$\wstols$ and $\wstol$) are minimal in the number of used variables;
in other words, it is open if there are corresponding
universal-existential (existential-universal) definitions using fewer
quantifiers. In the case $n>2$, it is open if there is an
existential-universal definition of timelike relatedness from
lightlike relatedness which works (at least) over Euclidean ordered
field, or equivalently if there is a universal-existential definition
of spacelike relatedness from lightlike relatedness.

\begin{table}[!htb]
  \caption{The table summarizes the results and open problems when the
    dimension $n=2$.
    \label{open-table2}}
  \begin{tabular}{|c|c;{1pt/2pt}c|c;{1pt/2pt}c|c;{1pt/2pt}c|}
    \hline $n=2$ & $\tleq\to\sleq$ & $\tleq\to\lleq$ &
    $\sleq\to\tleq$ & $\sleq\to\lleq$& $\lleq\to\tleq$&
    $\lleq\to\sleq$ \\

    \hline $\exists^2$ / $\forall^2$ &
    \multicolumn{6}{c|}{\cellcolor{red!42} $\nexists$ (not possible)}\\

    \hline $\exists^3$ & \cellcolor{verda}$\newettos$ &
    \cellcolor{red!42} $\nexists$ &  \cellcolor{verda}$\newestot$  &
    \multicolumn{3}{c|}{\cellcolor{red!42}  $\nexists$ (not possible)}\\

    \hline $\forall^3$ & \cellcolor{red!42} $\nexists$ &
    \cellcolor{verda}$\newuttol$ & \cellcolor{red!42} $\nexists$ &
    \cellcolor{verda} $\newustol$ &
    \multicolumn{2}{c|}{\cellcolor{red!42} $\nexists$}\\
    
    \hline $\exists^4$ & \cellcolor{verda}$\ettos$ &
    \cellcolor{red!42} $\nexists$ & \cellcolor{verda}$\estot$ &
    \multicolumn{3}{c|}{\cellcolor{red!42}  $\nexists$ (not possible)}\\

    \hline $\forall^4$ & \cellcolor{red!42} $\nexists$ &
    \cellcolor{verda}$\uttol$ & \cellcolor{red!42} $\nexists$ &
    \cellcolor{verda}$\ustol$ &
    \multicolumn{2}{c|}{\cellcolor{red!42} $\nexists$}\\



    \hline $\exists^1\forall^1$ & \open &\cellcolor{verda}$\ttol$ &
    \open &\cellcolor{verda}$\stol$ &
    \multicolumn{2}{c|}{\cellcolor{red!42} $\nexists$} \\
    
    \hline $\forall^1\exists^1$ &\cellcolor{verda}$\ttos$ & \open &
    \cellcolor{verda}$\stot$ & \open &
    \multicolumn{2}{c|}{\cellcolor{red!42} $\nexists$}\\

    \hline $\exists^2\forall^1$ / $\exists^1\forall^2$ & \open &\done &
    \open &\done &
    \multicolumn{2}{c|}{\cellcolor{red!42} $\nexists$} \\
    
    \hline $\forall^2\exists^1$ / $\forall^1\exists^2$  &\done & \open &
    \done & \open &
    \multicolumn{2}{c|}{\cellcolor{red!42} $\nexists$}\\

    \hline $\exists^2\forall^2$
    &\cellcolor{verda}$\wstots$ &\done
    &\cellcolor{verda}$\wstot$ &\done &
    \multicolumn{2}{c|}{\cellcolor{red!42} $\nexists$}\\
    
    \hline $\forall^2\exists^2$ &\done
    &\cellcolor{verda}$\wstols$ &\done &
    \cellcolor{verda}$\wstol$ &\multicolumn{2}{c|}{\cellcolor{red!42}
      $\nexists$} \\
    
    \hline
  \end{tabular}
\end{table}
\begin{table}[!htb]
  \caption{The table summarizes the results and open problems
    if $n>2$ and the underlying field is Euclidean.
   \label{open-table3}}
  \begin{tabular}{|c|c;{1pt/2pt}c|c;{1pt/2pt}c|c;{1pt/2pt}c|}
    \hline $n>2\kaj$ \eqref{eq-Eucl} & $\tleq\to\sleq$ &
    $\tleq\to\lleq$ & $\sleq\to\tleq$ & $\sleq\to\lleq$&
    $\lleq\to\tleq$& $\lleq\to\sleq$ \\
    
    \hline $\exists^2$ or $\forall^2$ &
    \multicolumn{6}{c|}{\cellcolor{red!42} $\nexists$ (not
      possible)}\\

    \hline $\exists^3$ & \cellcolor{verda}$\newettos$ &
    \cellcolor{red!42} $\nexists$ & \open &
    \multicolumn{3}{c|}{\cellcolor{red!42}  $\nexists$ (not possible)}\\

    \hline $\forall^3$ & \cellcolor{red!42} $\nexists$ &
    \cellcolor{verda}$\newuttol$ & \cellcolor{red!42} $\nexists$ &
    \open &
    \multicolumn{2}{c|}{\cellcolor{red!42} $\nexists$}\\

    \hline $\exists^4$ & \cellcolor{verda}$\ettos$ &
    \cellcolor{red!42} $\nexists$ & \open &
    \multicolumn{3}{c|}{\cellcolor{red!42}  $\nexists$ (not possible)}\\

    \hline $\forall^4$ & \cellcolor{red!42} $\nexists$ &
    \cellcolor{verda}$\uttol$ & \cellcolor{red!42} $\nexists$ &
    \open &
    \multicolumn{2}{c|}{\cellcolor{red!42} $\nexists$}\\

    \hline $\exists^*$ & \done &
    \cellcolor{red!42} $\nexists$ & \open &
    \multicolumn{3}{c|}{\cellcolor{red!42}  $\nexists$ (not possible)}\\

    \hline $\forall^*$ & \cellcolor{red!42} $\nexists$ & \done &
    \cellcolor{red!42} $\nexists$ & \open &
    \multicolumn{2}{c|}{\cellcolor{red!42} $\nexists$}\\

    \hline $\exists^1\forall^1$ & \open & \cellcolor{verda}$\ttol$ &
    \open & \cellcolor{verda}$\stol$ & \open &
    \cellcolor{verda}$\ltos$ \\

    \hline $\forall^1\exists^1$ & \cellcolor{verda}$\ttos$ & \open &
    \cellcolor{verda}$\stot$ & \open & \cellcolor{verda}$\ltot$ &
    \open \\

    \hline $\exists^2\forall^1$ /  $\exists^1\forall^2$  & \open & \done &
    \open & \done & \open &
    \done \\

    \hline $\forall^2\exists^1$ / $\forall^1\exists^2$ & \done & \open
    & \done & \open & \done & \open \\

    \hline $\exists^2\forall^2$ &\open  &\done
    &\cellcolor{verda}$\wstot$ &\done & \open &\done \\
    
    \hline $\forall^2\exists^2$ &\done &\open &\done &
    \cellcolor{verda}$\wstol$ &\done & \open \\

    \hline $\exists^*\forall^*$ & \done &\done & \done &\done & \open
    &\done \\

    \hline $\forall^*\exists^*$ &\done & \done &\done & \done
    &\done& \open \\ \hline
  \end{tabular}
\end{table}
\clearpage

In Tables \ref{open-table2} and \ref{open-table3}, we summarize our
results and open problems related to the simplest possible quantifier
complexity. For example, we do not know whether $\wstot$ and $\wstots$
are the simplest existential-universal definitions or not, see
Table~\ref{open-table2}. In case the dimension $n\ge3$, we also do not
know if timelike relatedness can or cannot be defined from lightlike
relatedness by an existential-universal formula, or equivalently if
spacelike relatedness can be defined from lightlike relatedness by a
universal-existential formula, see Table~\ref{open-table3}.

Here we were using a quite direct approach both to find defining
formulas and to show their nonexistence. Another possible indirect
approach is trying to use \L o\'s--Tarski theorem (see, \eg
\cite[Thm.6.5.4]{Ho93}, \cf also \cite{SAC16}) and related
preservation theorems for solving the above problems.

\bibliographystyle{asl}
\bibliography{LogRel12019}

\end{document}

%% file: tlseq.tikz
\begin{tikzpicture}[]

\pgfmathsetmacro{\conesize}{1}

\newcommand{\lcone}[4]{
\draw[red]   (#1-2*\conesize,#2+2*\conesize) to (#1,#2) to (#1+2*\conesize,#2+2*\conesize);
\fill[#4, opacity=.1] (#1,#2) to (#1-2*\conesize,#2+2*\conesize) to (#1+2*\conesize,#2+2*\conesize) to cycle ;

\draw[red, fill=#4, fill opacity=0.1] (#1+2*\conesize,#2+2*\conesize) arc [start angle=0,end angle=180,x radius=2*\conesize, y radius=.1*2*\conesize];
\draw[red] (#1-2*\conesize,#2+2*\conesize) arc [start angle=180,end angle=360,x radius=2*\conesize, y radius=.1*2*\conesize];

\draw[red]   (#1-\conesize,#2-\conesize) to (#1,#2) to (#1+\conesize,#2-\conesize);

\fill[#4, opacity=.1] (#1,#2) to (#1-\conesize,#2-\conesize) to (#1+\conesize,#2-\conesize) to cycle ;

\draw[dashed, red] (#1+\conesize,#2-\conesize) arc [start angle=0,end angle=180,x radius=\conesize, y radius=.1*\conesize];
\draw[red, fill=#4, fill opacity=0.1] (#1-\conesize,#2-\conesize) arc [start angle=180,end angle=360,x radius=\conesize, y radius=.1*\conesize];

\draw[fill] (#1,#2)  circle [radius=0.05] node[left,yshift=-1]{#3};

}

\lcone{0}{0}{$\vp$}{blue}

\coordinate (t) at (0.3,1.5);
\coordinate (l) at (1.2,1.2);
\coordinate (s) at (1.5,0.3);

\draw[ultra thick,blue] (0,0) to node[left,black] {$\tleq$} (t);
\draw[ultra thick,red] (0,0) to node[right,black] {$\lleq$} (l);
\draw[ultra thick,green!84!black] (0,0) to node[below,black] {$\sleq$} (s);

\node[right] at (2,2) {$\Lambda_{\vp}$};
\draw[fill] (t)  circle [radius=0.05] node[below right]{$\vq$};
\draw[fill] (l)  circle [radius=0.05] node[below right]{$\vq'$};
\draw[fill] (s)  circle [radius=0.05] node[below right]{$\vq''$};
\draw[fill] (0,0)  circle [radius=0.05];

\end{tikzpicture}

%% file: BA2relat.tikz
\usetikzlibrary{calc}
\begin{tikzpicture}[scale=1.5,>=stealth]

\pgfmathsetmacro{\xshift}{3}
\pgfmathsetmacro{\yshift}{1}

  \coordinate (A1) at (0,0);
  \coordinate(A2) at (-1,1); 
  \coordinate (A3) at (0,2); 
  \coordinate (A4) at (1,1); 

\begin{scope}[shift={(0,1)}]
  \coordinate (A6) at (0,0); 
  \coordinate (A7) at (-1,1);
  \coordinate (A8) at (0,2);
  \coordinate (A5) at (1,1);
\end{scope}

\begin{scope}[shift={(\xshift,\yshift)}]
  \coordinate (B1) at (0,0); 
  \coordinate (B2) at (-1,1);
  \coordinate (B3) at (0,2);
  \coordinate (B4) at (1,1);
\end{scope}

\begin{scope}[shift={(\xshift,\yshift+1)}]
  \coordinate (B6) at (0,0); 
  \coordinate (B7) at (-1,1);
  \coordinate (B8) at (0,2);
  \coordinate (B5) at (1,1);
\end{scope}

\foreach \i in {1,2,...,8}
{
\draw[dashed, thick, gray] (A\i) to (B\i);
}

\draw[thick, gray] (A1) to (B1);
\draw[thick, gray] (A8) to (B8);
\draw[thick, gray] (A5) to (B5);
\draw[thick, gray] (A6) to (B6);
\draw[thick, gray] (A7) to (B7);

\draw[very thick, darkgray] (A5) -- (A8);
\draw[very thick, darkgray,dashed] (A3) -- (A8);
\draw[very thick, darkgray] (A7) -- (A8);
\draw[very thick, darkgray] (A1) -- (A2);
\draw[very thick, darkgray,dashed] (A3) -- (A4);
\draw[very thick, darkgray] (A7) -- (A6);
\draw[very thick, darkgray] (A1) -- (A4);
\draw[very thick, darkgray] (A1) -- (A6);
\draw[very thick, darkgray,dashed] (A2) -- (A3);
\draw[very thick, darkgray] (A2) -- (A7);
\draw[very thick, darkgray] (A4) -- (A5);
\draw[very thick, darkgray] (A6) -- (A5);

\draw[very thick, darkgray] (B5) -- (B8);
\draw[very thick, darkgray,dashed] (B3) -- (B8);
\draw[very thick, darkgray] (B7) -- (B8);
\draw[very thick, darkgray] (B1) -- (B2);
\draw[very thick, darkgray,dashed] (B3) -- (B4);
\draw[very thick, darkgray] (B7) -- (B6);
\draw[very thick, darkgray] (B1) -- (B4);
\draw[very thick, darkgray] (B1) -- (B6);
\draw[very thick, darkgray,dashed] (B2) -- (B3);
\draw[very thick, darkgray] (B2) -- (B7);
\draw[very thick, darkgray] (B4) -- (B5);
\draw[very thick, darkgray] (B6) -- (B5);

\foreach \i in {1,2,...,8}
	{
	  \draw[fill=black] (A\i) circle (0.05);
 \draw[fill=black] (B\i) circle (0.05);
}

\node[below] at (A1) {$\emptyset$};
\node[below left] at (A2) {$\tleq$};
\node[below left] (tl) at (A6) {$\lleq$};
\node[above] (1) at (B8) {$\Q^n\times \Q^n$};
\node[below right] at (A4) {$\sleq$};
\node[below right] (neqsl) at (B1){$=$};
\node[above left] at (B7){$\slneq$};
\node[above left] at (A8){$\neq$};
\node[left] at (A7){$\slneq_{\tinyneq}$};
\node[above right] at (B5){$\tlneq$};
\node[above right] at (B3){$\llneq$};
\node[above right,xshift=-1] at (A3){$\llneq_{\tinyneq}$};
\node[below right,yshift=2] at (A5){$\tlneq_{\tinyneq}$};
\node[below right] at (B4) {$\sleq\etacup =$};
\node[right] at (B2) {$\tleq\etacup =$};
\node[right] at (B6) {$\lleq\etacup =$};

\end{tikzpicture}

%% file: ttos.tikz
\begin{tikzpicture}[]

\pgfmathsetmacro{\conesize}{2}

\newcommand{\lcone}[4]{
\draw[red]   (#1-\conesize,#2+\conesize) to (#1,#2) to (#1+\conesize,#2+\conesize);
\fill[#4, opacity=.1] (#1,#2) to (#1-\conesize,#2+\conesize) to (#1+\conesize,#2+\conesize) to cycle ;
\draw[red]   (#1-\conesize,#2-\conesize) to (#1,#2) to (#1+\conesize,#2-\conesize);
\fill[#4, opacity=.1] (#1,#2) to (#1-\conesize,#2-\conesize) to (#1+\conesize,#2-\conesize) to cycle ;
\draw[fill] (#1,#2)  circle [radius=0.05] node[left]{#3};
}

\lcone{0}{0}{$x$}{blue}
\lcone{3}{0}{$y$}{green}
\lcone{1.3}{0.5}{$\forall z$}{magenta}
\lcone{1.7}{-0.7}{$\exists u$}{brown}
\end{tikzpicture}

%% file: ttos1.tikz
\begin{tikzpicture}[]

\pgfmathsetmacro{\conesize}{2}

\newcommand{\flcone}[3]{
\draw[red]   (#1-\conesize,#2+\conesize) to (#1,#2) to (#1+\conesize,#2+\conesize);
\fill[#3, opacity=.3] (#1,#2) to (#1-\conesize,#2+\conesize) to (#1+\conesize,#2+\conesize) to cycle ;
}

\newcommand{\plcone}[3]{
\draw[red]   (#1-\conesize,#2-\conesize) to (#1,#2) to (#1+\conesize,#2-\conesize);
\fill[#3, opacity=.3] (#1,#2) to (#1-\conesize,#2-\conesize) to (#1+\conesize,#2-\conesize) to cycle ;
}

\begin{scope}[shift={(-3,0)}]
\flcone{0}{0}{blue!50}
\flcone{1.6}{1.6}{green!50}
\flcone{0.8}{0.8}{magenta!50}
 \draw[fill] (0,0)  circle [radius=0.05] node[below right]{$\vx$};
 \draw[fill] (0.8,0.8)  circle [radius=0.05] node[below right]{$\vz$};
 \draw[fill] (1.6,1.6)  circle [radius=0.05] node[below right]{$\vy$};
 \draw[fill] (0.7,1.6)  circle [radius=0.05] node[above left]{$\vu$};
\end{scope}

\begin{scope}[shift={(3,1.5)}]
\plcone{0}{0}{blue!50}
\plcone{1.6}{1.6}{green!50}
\plcone{0.8}{0.8}{magenta!50}
 \draw[fill] (0,0)  circle [radius=0.05] node[above left]{$\vx$};
 \draw[fill] (0.8,0.8)  circle [radius=0.05] node[above left]{$\vz$};
 \draw[fill] (1.6,1.6)  circle [radius=0.05] node[above left]{$\vy$};
 \draw[fill] (0.7,0)  circle [radius=0.05] node[below right]{$\vu$};
\end{scope}

\end{tikzpicture}

%% file: ttos2.tikz
\begin{tikzpicture}[scale=0.75]

\pgfmathsetmacro{\conesize}{2}

\newcommand{\flcone}[3]{
\draw[red]   (#1-\conesize,#2+\conesize) to (#1,#2) to (#1+\conesize,#2+\conesize);
\fill[#3, opacity=.1] (#1,#2) to (#1-\conesize,#2+\conesize) to (#1+\conesize,#2+\conesize) to cycle ;

\draw[red, fill=#3, fill opacity=0.1] (#1+\conesize,#2+\conesize) arc [start angle=0,end angle=180,x radius=\conesize, y radius=.1*\conesize];
\draw[red] (#1-\conesize,#2+\conesize) arc [start angle=180,end angle=360,x radius=\conesize, y radius=.1*\conesize];
}

\newcommand{\plcone}[3]{
\draw[red]   (#1-\conesize,#2-\conesize) to (#1,#2) to (#1+\conesize,#2-\conesize);

\fill[#3, opacity=.1] (#1,#2) to (#1-\conesize,#2-\conesize) to (#1+\conesize,#2-\conesize) to cycle ;

\draw[dashed, red] (#1+\conesize,#2-\conesize) arc [start angle=0,end angle=180,x radius=\conesize, y radius=.1*\conesize];
\draw[red, fill=#3, fill opacity=0.1] (#1-\conesize,#2-\conesize) arc [start angle=180,end angle=360,x radius=\conesize, y radius=.1*\conesize];
}

\newcommand{\lcone}[4]{
\draw[red]   (#1-\conesize,#2+\conesize) to (#1,#2) to (#1+\conesize,#2+\conesize);
\fill[#4, opacity=.1] (#1,#2) to (#1-\conesize,#2+\conesize) to (#1+\conesize,#2+\conesize) to cycle ;

\draw[red, fill=#4, fill opacity=0.1] (#1+\conesize,#2+\conesize) arc [start angle=0,end angle=180,x radius=\conesize, y radius=.1*\conesize];
\draw[red] (#1-\conesize,#2+\conesize) arc [start angle=180,end angle=360,x radius=\conesize, y radius=.1*\conesize];

\draw[fill] (#1,#2)  circle [radius=0.05] node[left]{#3};

\draw[red]   (#1-\conesize,#2-\conesize) to (#1,#2) to (#1+\conesize,#2-\conesize);

\fill[#4, opacity=.1] (#1,#2) to (#1-\conesize,#2-\conesize) to (#1+\conesize,#2-\conesize) to cycle ;

\draw[dashed, red] (#1+\conesize,#2-\conesize) arc [start angle=0,end angle=180,x radius=\conesize, y radius=.1*\conesize];
\draw[red, fill=#4, fill opacity=0.1] (#1-\conesize,#2-\conesize) arc [start angle=180,end angle=360,x radius=\conesize, y radius=.1*\conesize];
}

\begin{scope}[]
\draw[white, fill opacity=0.2, fill=darkgray] (-3,-1) to (-1,1)to (5,1)  to (3,-1) node[black,fill opacity=1,right]{$H$} to cycle;

\plcone{1.2}{1.5}{brown}

\draw[fill] (-0.3,0.5)  circle [radius=0.05] node[below left]{$\vx$};
\draw[fill] (3.3,0.5)  circle [radius=0.05] node[below left]{$\vy$};
\draw[fill] (1.2,1.5)  circle [radius=0.05] node[above left]{$\vz$};
\draw[fill] (1.2,-0.5)  circle [radius=0.05] node[above right]{$\vu$};
\end{scope}

\begin{scope}[shift={(8,0)}]
\draw[white, fill opacity=0.2, fill=darkgray] (-3,-1) to (-1,1)to (5,1)  to (3,-1) node[black,fill opacity=1,right]{$H$} to cycle;

\flcone{-0.3}{0.5}{blue}
\flcone{3.3}{.5}{green}
\flcone{1.5}{-0.5}{brown}

\draw[fill] (-0.3,0.5)  circle [radius=0.05] node[below left]{$\vx$};
\draw[fill] (3.3,0.5)  circle [radius=0.05] node[below left]{$\vy$};
\draw[fill] (1.5,-0.5)  circle [radius=0.05] node[below left]{$\vz$};
\draw[fill] (1.5,0)  circle [radius=0.05] node[above right]{$\vu$};
\end{scope}

\end{tikzpicture}

%% file: stot.tikz
\begin{tikzpicture}[]

\pgfmathsetmacro{\conesize}{2}

\newcommand{\lcone}[4]{
\draw[red]   (#1-\conesize,#2+\conesize) to (#1,#2) to (#1+\conesize,#2+\conesize);
\fill[#4, opacity=.1] (#1,#2) to (#1-\conesize,#2+\conesize) to (#1+\conesize,#2+\conesize) to cycle ;
\draw[red]   (#1-\conesize,#2-\conesize) to (#1,#2) to (#1+\conesize,#2-\conesize);
\fill[#4, opacity=.1] (#1,#2) to (#1-\conesize,#2-\conesize) to (#1+\conesize,#2-\conesize) to cycle ;
\draw[fill] (#1,#2)  circle [radius=0.05] node[left]{#3};
}

\lcone{0}{0}{$x$}{blue}
\lcone{0}{2}{$y$}{green}
\lcone{-1}{1.7}{$\forall z$}{magenta}
\lcone{0.2}{1.15}{$\exists u$}{brown}

\end{tikzpicture}

%% file: stot2.tikz
\begin{tikzpicture}[scale=0.75]
\pgfmathsetmacro{\conesize}{1.5}

\newcommand{\flcone}[3]{
\draw[red]   (#1-\conesize,#2+\conesize) to (#1,#2) to (#1+\conesize,#2+\conesize);
\fill[#3, opacity=.1] (#1,#2) to (#1-\conesize,#2+\conesize) to (#1+\conesize,#2+\conesize) to cycle ;

\draw[red, fill=#3, fill opacity=0.1] (#1+\conesize,#2+\conesize) arc [start angle=0,end angle=180,x radius=\conesize, y radius=.1*\conesize];
\draw[red] (#1-\conesize,#2+\conesize) arc [start angle=180,end angle=360,x radius=\conesize, y radius=.1*\conesize];
}

\newcommand{\plcone}[3]{
\draw[red]   (#1-\conesize,#2-\conesize) to (#1,#2) to (#1+\conesize,#2-\conesize);

\fill[#3, opacity=.1] (#1,#2) to (#1-\conesize,#2-\conesize) to (#1+\conesize,#2-\conesize) to cycle ;

\draw[dashed, red] (#1+\conesize,#2-\conesize) arc [start angle=0,end angle=180,x radius=\conesize, y radius=.1*\conesize];
\draw[red, fill=#3, fill opacity=0.1] (#1-\conesize,#2-\conesize) arc [start angle=180,end angle=360,x radius=\conesize, y radius=.1*\conesize];
}

\newcommand{\lcone}[4]{
\draw[red]   (#1-\conesize,#2+\conesize) to (#1,#2) to (#1+\conesize,#2+\conesize);
\fill[#4, opacity=.1] (#1,#2) to (#1-\conesize,#2+\conesize) to (#1+\conesize,#2+\conesize) to cycle ;

\draw[red, fill=#4, fill opacity=0.1] (#1+\conesize,#2+\conesize) arc [start angle=0,end angle=180,x radius=\conesize, y radius=.1*\conesize];
\draw[red] (#1-\conesize,#2+\conesize) arc [start angle=180,end angle=360,x radius=\conesize, y radius=.1*\conesize];

\draw[fill] (#1,#2)  circle [radius=0.05] node[left]{#3};

\draw[red]   (#1-\conesize,#2-\conesize) to (#1,#2) to (#1+\conesize,#2-\conesize);

\fill[#4, opacity=.1] (#1,#2) to (#1-\conesize,#2-\conesize) to (#1+\conesize,#2-\conesize) to cycle ;

\draw[dashed, red] (#1+\conesize,#2-\conesize) arc [start angle=0,end angle=180,x radius=\conesize, y radius=.1*\conesize];
\draw[red, fill=#4, fill opacity=0.1] (#1-\conesize,#2-\conesize) arc [start angle=180,end angle=360,x radius=\conesize, y radius=.1*\conesize];
}

\begin{scope}[]

\draw[dotted] (0,0.5)  to (0,1.2);
\draw[white, fill opacity=0.2, fill=darkgray] (-3,0.5) to (-1,1.5) to (5,1.5) to (3,0.5)  node[black,fill opacity=1,right]{$H$}  to cycle;

\flcone{0}{-1.2}{blue}
\plcone{0}{2.5}{green}
\flcone{2.7}{1}{brown}

\draw (0,1) to (0,3.5) node[below right]{$\ell$};
\draw (0,-1.5) to (0,0.5);
\draw[fill] (0,-1.2)  circle [radius=0.05] node[left]{$\vx$};
\draw[fill] (0,2.5)  circle [radius=0.05] node[left]{$\vy$};
\draw[fill] (2.7,1)  circle [radius=0.05] node[right]{$\vz$};
\draw[fill] (0,1)  circle [radius=0.05] node[right]{$\vu$};
\end{scope}

\begin{scope}[shift={(8,0)}]

\draw[dotted] (0,0.5)  to (0,1.2);
\draw[white, fill opacity=0.2, fill=darkgray] (-3,0.5) to (-1,1.5) to (4,1.5) to (2,0.5)  node[black,fill opacity=1,right]{$H$}  to cycle;

\flcone{0}{-1.2}{blue}
\plcone{0}{2.5}{green}

\draw (0,1) to (0,3.5) node[below right]{$\ell$};
\draw (0,-1.5) to (0,0.5);
\draw[fill] (0,-1.2)  circle [radius=0.05] node[left]{$\vx$};
\draw[fill] (0,2.5)  circle [radius=0.05] node[left]{$\vy$};
\draw[fill] (0,1)  circle [radius=0.05] node[right]{$\vz$};
\draw[fill] (-0.5,1)  circle [radius=0.05] node[left]{$\vu$};
\end{scope}

\end{tikzpicture}

%% file: stot1.tikz
\begin{tikzpicture}[scale=0.85]

\pgfmathsetmacro{\conesize}{1.5}

\newcommand{\flcone}[5]{
\draw[#3]   (#1-#5,#2+#5) to (#1,#2) to (#1+#5,#2+#5);
\fill[#3, opacity=#4] (#1,#2) to (#1-#5,#2+#5) to (#1+#5,#2+#5) to cycle ;

\draw[#3, fill=#3, fill opacity=#4] (#1+#5,#2+#5) arc [start angle=0,end angle=180,x radius=#5, y radius=.1*#5];
\draw[#3] (#1-#5,#2+#5) arc [start angle=180,end angle=360,x radius=#5, y radius=.1*#5];
}

\newcommand{\plcone}[5]{
\draw[#3]   (#1-#5,#2-#5) to (#1,#2) to (#1+#5,#2-#5);

\fill[#3, opacity=#4] (#1,#2) to (#1-#5,#2-#5) to (#1+#5,#2-#5) to cycle ;

\draw[dashed, #3] (#1+#5,#2-#5) arc [start angle=0,end angle=180,x radius=#5, y radius=.1*#5];
\draw[#3, fill=#3, fill opacity=#4] (#1-#5,#2-#5) arc [start angle=180,end angle=360,x radius=#5, y radius=.1*#5];
}

\begin{scope}[]

\flcone{0}{0}{blue}{.1}{2}
\plcone{0}{0}{blue}{0.1}{2}

\flcone{-0.5}{-0.5}{green!63!black}{0.1}{2.5}
\plcone{-0.5}{-0.5}{green!63!black}{0.1}{1.5}

\flcone{0.5}{0.5}{brown!63!black}{0.1}{1.5}
\plcone{0.5}{0.5}{brown!63!black}{0.1}{2.5}

\draw[fill] (0,0)  circle [radius=0.05] node[left]{$\vz$};
\draw[fill] (-0.5,-0.5)  circle [radius=0.05] node[left]{$\vx$};
\draw[fill] (0.5,0.5)  circle [radius=0.05] node[right]{$\vy$};
\draw[fill] (.5,2)  circle [radius=0.05] node[right]{$\vu$};
\end{scope}

\begin{scope}[shift={(6,0)}]

\draw[white, fill opacity=0.1, fill=gray] (-3,-0.5) to (-1,0.5) to (5,0.5) to (3,-0.5)  to cycle;

\flcone{1}{0}{blue}{0.1}{2}
\plcone{1}{0}{blue}{0.1}{2}

\flcone{0}{0}{green!63!black}{0.1}{2}
\plcone{0}{0}{green!63!black}{0.1}{2}

\flcone{2}{0}{brown!63!black}{0.1}{2}
\plcone{2}{0}{brown!63!black}{0.1}{2}


\draw[fill] (0,0)  circle [radius=0.05] node[left]{$\vx$};
\draw[fill] (2,0)  circle [radius=0.05] node[right]{$\vy$};
\draw[fill] (1,0)  circle [radius=0.05] node[right]{$\vz$};
\draw[fill] (1,2)  circle [radius=0.05] node[right]{$\vu$};
\end{scope}

\end{tikzpicture}

%% file: ltos.tikz
\begin{tikzpicture}[]

\pgfmathsetmacro{\conesize}{2}

\newcommand{\lcone}[4]{
\draw[red]   (#1-\conesize,#2+\conesize) to (#1,#2) to (#1+\conesize,#2+\conesize);
\fill[#4, opacity=.1] (#1,#2) to (#1-\conesize,#2+\conesize) to (#1+\conesize,#2+\conesize) to cycle ;

\draw[red, fill=#4, fill opacity=0.1] (#1+\conesize,#2+\conesize) arc [start angle=0,end angle=180,x radius=\conesize, y radius=.1*\conesize];
\draw[red] (#1-\conesize,#2+\conesize) arc [start angle=180,end angle=360,x radius=\conesize, y radius=.1*\conesize];

\draw[fill] (#1,#2)  circle [radius=0.05] node[left]{#3};

\draw[red]   (#1-\conesize,#2-\conesize) to (#1,#2) to (#1+\conesize,#2-\conesize);

\fill[#4, opacity=.1] (#1,#2) to (#1-\conesize,#2-\conesize) to (#1+\conesize,#2-\conesize) to cycle ;

\draw[dashed, red] (#1+\conesize,#2-\conesize) arc [start angle=0,end angle=180,x radius=\conesize, y radius=.1*\conesize];
\draw[red, fill=#4, fill opacity=0.1] (#1-\conesize,#2-\conesize) arc [start angle=180,end angle=360,x radius=\conesize, y radius=.1*\conesize];
}

\draw[white, fill opacity=0.1, fill=red] (-1.1-2,-2.2+0.3) to (1.6-2,3.2+0.3)to (3.6+2,2.9-0.3) node[below ,black,fill opacity=1]{tangent (hyper)plane}  to (0.9+2,-2.5-0.3)  to cycle;
\draw[red,dashed]  (-1.1,-2.2)to (1.6,3.2);
\draw[red,dashed, shift={(2,-0.3)}]  (-1.1,-2.2)to (1.6,3.2);

\lcone{0}{0}{$x$}{blue}
\lcone{2}{-0.3}{$y$}{green}
\lcone{.5}{1}{$\exists z$}{brown}
\draw (2+1.5,-0.3+3) circle [radius=.05] node[left]{$\nexists u$};
\draw[red,dotted,shorten <=-0.3cm,shorten >=-0.3cm] (0,0) to[out=45,in=200] (2+1.5,-0.3+3);
\draw[red,dotted,shorten <=-0.3cm,shorten >=-0.3cm] (0.5,1) to[out=70,in=200] (2+1.5,-0.3+3);

\end{tikzpicture}

%% file: deftriangle.tikz
\begin{tikzpicture}[>=stealth,scale=1.2]
\tikzstyle{nyil}=[->,->,bend right=17, thick]
\tikzstyle{bogyo}=[circle,, fill=blue!20]

\node[bogyo] (t) at (0,4.5){$\tleq$};
\node[bogyo] (l) at (3,0){$\lleq$};
\node[bogyo] (s) at (-3,0){$\sleq$};

\draw[nyil] (t) to node[above,sloped,] {$\ttos$} node[below,sloped]
{\tiny$(n\ge 2$; \ref{eq-Eucl} or $n=2)$ } (s);

\draw[nyil] (s) to node[below,sloped,] {$\stot$} node[above,sloped] {
\tiny$(n\ge 2$; \ref{eq-Eucl} or $n=2)$ } (t);

\draw[nyil] (s) to node[below, sloped] {$\stol$} node[above,sloped]
{\tiny$(n\ge 2$; \ref{eq-Eucl} or $n=2)$} (l);

\draw[nyil] (l) to node[above, sloped] {$\ltos$} node[below,sloped] {
\tiny$(n\ge 3$; \ref{eq-Eucl})} (s);

\draw[nyil] (l) to node[above,sloped] {$\ltot$} node[below,sloped] {
\tiny$(n\ge 3$; \ref{eq-Eucl})} (t);

\draw[nyil] (t) to node[below, sloped] {$\ttol$} node[above,sloped]
{\tiny$(n\ge 2$; \ref{eq-Eucl} or $n=2)$} (l);

\end{tikzpicture}

%% file: eTtoS.tikz
\begin{tikzpicture}[scale=.5,rotate=45]

\tikzstyle{ll}=[red]
\tikzstyle{TS}=[fill=blue!30,draw=blue!30]
\tikzstyle{SS}=[fill=red!30,draw=red!30]
\tikzstyle{ST}=[fill=green!20,draw=green!30]
\tikzstyle{TT}=[fill=brown!30,draw=brown!30]
\tikzstyle{pont}=[circle, inner sep=1]

\begin{scope}[]
\draw[SS] (2,2) rectangle node{$\tlneq\tlneq$} (4,4);
\draw[TT] (0,0) rectangle node{$\tleq\tleq$} (2,2);
\draw[TT] (4,4) rectangle node{$\tleq\tleq$} (6,6);
\draw[TS] (0,2) rectangle node{$\tleq\tlneq$} (2,4);
\draw[SS] (0,4) rectangle node{$\tlneq \tlneq$} (2,6);
\draw[ST] (2,0) rectangle node{$\tlneq \tleq$} (4,2);
\draw[ST] (4,2) rectangle node{$\tlneq \tleq$} (6,4);
\draw[TS] (2,4) rectangle node{$\tleq\tlneq$} (4,6);
\draw[SS] (4,0) rectangle node{$\tlneq\tlneq$} (6,2);
\draw[blue,thick] (4,4) to (4,6);
\draw[blue,thick] (0,2) to (2,2);
\draw[green,thick] (2,0) to (2,2);
\draw[green,thick] (4,4) to (6,4);
\draw[red!80,thick] (0,4) to (4,4) to (4,0) ;
\draw[red!80,thick] (2,6) to (2,2) to (6,2) ;

\draw[fill] (2,4) circle [radius=.07] node[below] {$\vp$};
\draw[fill] (4,2) circle [radius=.07] node[below] {$\vq$};

\draw[fill] (2.5,2.5) circle [radius=.04] node[below] {$\vs$};
\draw[fill] (1.3,1.3) circle [radius=.04] node[above] {$\vr$};
\draw[fill] (2.5,.1) circle [radius=.04] node[below right] {$\vz$};
\draw[fill] (0.1,2.5) circle [radius=.04] node[below left] {$\vx$};
\node at (3.5,6.5) {\huge\checkmark};

\end{scope}

\begin{scope}[shift={(8,-8)}]
\draw[TT] (2,2) rectangle node{$\tleq\tleq$} (4,4);
\draw[TT] (0,0) rectangle node{$\tleq\tleq$} (2,2);
\draw[TT] (4,4) rectangle node{$\tleq\tleq$} (6,6);
\draw[ST] (0,2) rectangle node{$\tlneq \tleq$} (2,4);
\draw[SS] (0,4) rectangle node{$\tlneq \tlneq$} (2,6);
\draw[ST] (2,0) rectangle node{$\tlneq \tleq$} (4,2);
\draw[TS] (4,2) rectangle node{$\tleq\tlneq$} (6,4);
\draw[TS] (2,4) rectangle node{$\tleq\tlneq$} (4,6);
\draw[SS] (4,0) rectangle node{$\tlneq\tlneq$} (6,2);
\draw[blue,thick] (2,4) to (6,4);
\draw[blue,thick] (4,2) to (4,6);
\draw[green,thick] (2,0) to (2,4);
\draw[green,thick] (0,2) to (4,2);
\draw[red!80,thick] (0,4) to (2,4) to (2,6);
\draw[red!80,thick] (6,2) to (4,2) to (4,0);

\draw[fill] (2,2) circle [radius=.07] node[below] {$\vp$};
\draw[fill] (4,4) circle [radius=.07] node[below] {$\vq$};

\node[cross out,draw,red,inner sep=0,very thick] at (7,3) {\textcolor{black}{\huge\checkmark}};
\end{scope}

\begin{scope}[shift={(8,0)}]
\draw[TT] (0,0) rectangle node{$\tleq\tleq$} (2,3);
\draw[SS] (0,3) rectangle node{$\tlneq \tlneq$} (2,6);
\draw[ST] (2,0) rectangle node{$\tlneq \tleq$} (4,3);
\draw[TT] (4,3) rectangle node{$\tleq\tleq$} (6,6);
\draw[TS] (2,3) rectangle node{$\tleq\tlneq$} (4,6);
\draw[SS] (4,0) rectangle node{$\tlneq\tlneq$} (6,3);
\draw[red!80,thick] (2,3) to (2,6);
\draw[red!80,thick] (4,3) to (4,0);
\draw[red!80,thick] (0,3) to (6,3);
\draw[blue,thick] (4,3) to (4,6);
\draw[green,thick] (2,0) to (2,3);

\draw[fill] (2,3) circle [radius=.07] node[below] {$\vp$};
\draw[fill] (4,3) circle [radius=.07] node[below] {$\vq$};

\node[cross out,draw,red,inner sep=0,very thick] at (7,3) {\textcolor{black}{\huge\checkmark}};
\end{scope}

\begin{scope}[shift={(4,-4)}]
\draw[dashed, red] (0,-2) to (0,2); 
\draw[dashed, red] (-2,0) to (2,0); 
\node[pont,TT] at (0,0) {$r$};
\node[pont,SS] at (1,1) {$s$};
\node[pont,TS] at (-1,1) {$x$};
\node[pont,ST] at (1,-1) {$z$};
\end{scope}

\end{tikzpicture}

%% file: eStoT.tikz
\begin{tikzpicture}[scale=.5]

\tikzstyle{ll}=[red]
\tikzstyle{TS}=[fill=blue!30,draw=blue!30]
\tikzstyle{SS}=[fill=red!30,draw=red!30]
\tikzstyle{ST}=[fill=green!30,draw=green!30]
\tikzstyle{TT}=[fill=brown!30,draw=brown!30]
\tikzstyle{pont}=[circle, inner sep=1]

\begin{scope}[shift={(7,2)}]

\draw[TT]  (-6,-7) rectangle (6,4);

\draw[ST] (-2,0) circle(3);
\draw[TS] (2,0) circle(3);

\begin{scope}
    \clip (-2,0) circle(3cm);
    \clip (2,0) circle(3cm);
\draw[SS] (0,0) circle (3);
\node[below] at (0,0) {$\slneq\slneq$};
\end{scope}

\node at (-4.5,-6) {$\sleq\sleq$};
\node at (-3.5,1.5) {$\slneq\sleq$};
\node at (3.5,1.5) {$\sleq\slneq$};

\draw[dotted] (0,-3) circle(3);

\draw[fill] (-2,0) circle [radius=.07] node[above] {$\vp$};
\draw[fill] (-2,0) circle [radius=.07] node[below left] {\tiny$(0,-2,0)$};

\draw[fill] (2,0) circle [radius=.07] node[above] {$\vq$};
\draw[fill] (2,0) circle [radius=.07] node[below right] {\tiny$(0,2,0)$};

\draw[fill] (0,-3) circle [radius=.07] node[above] {$\vr$};
\draw[fill] (0,-3) circle [radius=.07] node[below] {\tiny$(0,0,-3)$};

\draw[fill] (0,1) circle [radius=.12] node[above] {$\vs$};
\draw[fill] (0,1) circle [radius=.12] node[below] {\tiny$(3,0,1)$};

\draw[fill] (2,-2) circle [radius=.12] node[above] {$\vx$};
\draw[fill] (2,-2) circle [radius=.12] node[below] {\tiny$(3,2,-2)$};

\draw[fill] (-2,-2) circle [radius=.12] node[above] {$\vz$};
\draw[fill] (-2,-2) circle [radius=.12] node[below] {\tiny$(3,-2,-2)$};
\end{scope}

\begin{scope}[shift={(-4,0)}]
\draw[dashed, red] (-1.4,1.4) to (1.4,-1.4); 
\draw[dashed, red] (-1.4,-1.4) to (1.4,1.4); 
\node[pont,TT] at (0,0) {$r$};
\node[pont,SS] at (1.4,0) {$s$};
\draw[dashed,red] (-1.4,1.4) arc [start angle=180,end angle=0, x radius=1.4, y radius=0.1];
\draw[dotted,red] (-1.4,-1.4) arc [start angle=180,end angle=0, x radius=1.4, y radius=0.1];
\node[pont,ST] at (0,-1.4) {$z$};
\node[pont,TS] at (0,1.4) {$x$};
\draw[dashed,red] (-1.4,1.4) arc [start angle=180,end angle=360, x radius=1.4, y radius=0.1];
\draw[dashed,red] (-1.4,-1.4) arc [start angle=180,end angle=360, x radius=1.4, y radius=0.1];
\end{scope}

\end{tikzpicture}

%% file: nrfe2-list.tikz
\usetikzlibrary{decorations.pathreplacing}

\begin{tikzpicture}[scale=0.6]
\tikzstyle{poento}=[circle, draw, fill=white, inner sep=0,minimum size=12];
\tikzstyle{linio}=[very thick];
\tikzstyle{tl}=[blue,ultra thick]
\tikzstyle{ll}=[very thick, red, dashed]  
\tikzstyle{ntl}=[very thick, brown, dashdotted] 



\begin{scope}[shift={(0,0)}]
\node[poento] (x) at (0,1.5) {$x$};
\node[poento] (y) at (1.5,1.5) {$y$};
\node[poento] (p) at (0,0) {$p$};
\node[poento] (q) at (1.5,0) {$q$};

\draw[ntl] (x) to (p); 
\draw[ntl] (x) to (q); 
\draw[ntl] (x) to (y); 
\draw[ntl] (y) to (p);
\draw[ntl] (y) to (q); 

\draw [decorate,decoration={brace,amplitude=10pt,raise=4pt},yshift=0pt]
(2,0) -- (-0.5,0) node [below=12,black,midway] {$\Q^n\times \Q^n$};
\end{scope}

\begin{scope}[shift={(3,0)}]
\node[poento] (x) at (0,1.5) {$x$};
\node[poento] (y) at (1.5,1.5) {$y$};
\node[poento] (p) at (0,0) {$p$};
\node[poento] (q) at (1.5,0) {$q$};

\draw[tl] (x) to (p); 
\draw[ntl] (x) to (q); 
\draw[ntl] (x) to (y); 
\draw[ntl] (y) to (p);
\draw[ntl] (y) to (q); 
\node at (2.25,0.75) {$\cong$};
\draw [decorate,decoration={brace,amplitude=10pt,raise=4pt},yshift=0pt]
(11,0) -- (-0.5,0) node [below=12,black,midway] {$\neq$};
\end{scope}

\begin{scope}[shift={(6,0)}]
\node[poento] (x) at (0,1.5) {$x$};
\node[poento] (y) at (1.5,1.5) {$y$};
\node[poento] (p) at (0,0) {$p$};
\node[poento] (q) at (1.5,0) {$q$};

\draw[ntl] (x) to (p); 
\draw[ntl] (x) to (q); 
\draw[ntl] (x) to (y); 
\draw[tl] (y) to (p);
\draw[ntl] (y) to (q); 

\node at (2.25,0.75) {$\cong$};
\end{scope}

\begin{scope}[shift={(9,0)}]
\node[poento] (x) at (0,1.5) {$x$};
\node[poento] (y) at (1.5,1.5) {$y$};
\node[poento] (p) at (0,0) {$p$};
\node[poento] (q) at (1.5,0) {$q$};

\draw[ntl] (x) to (p); 
\draw[ntl] (x) to (q); 
\draw[ntl] (x) to (y); 
\draw[ntl] (y) to (p);
\draw[tl] (y) to (q); 

\node at (2.25,0.75) {$\cong$};
\end{scope}

\begin{scope}[shift={(12,0)}]
\node[poento] (x) at (0,1.5) {$x$};
\node[poento] (y) at (1.5,1.5) {$y$};
\node[poento] (p) at (0,0) {$p$};
\node[poento] (q) at (1.5,0) {$q$};

\draw[ntl] (x) to (p); 
\draw[tl] (x) to (q); 
\draw[ntl] (x) to (y); 
\draw[ntl] (y) to (p);
\draw[ntl] (y) to (q); 
\end{scope}

\begin{scope}[shift={(15,0)}]
\node[poento] (x) at (0,1.5) {$x$};
\node[poento] (y) at (1.5,1.5) {$y$};
\node[poento] (p) at (0,0) {$p$};
\node[poento] (q) at (1.5,0) {$q$};

\draw[ntl] (x) to (p); 
\draw[ntl] (x) to (q); 
\draw[tl] (x) to (y); 
\draw[ntl] (y) to (p);
\draw[ntl] (y) to (q); 
\draw [decorate,decoration={brace,amplitude=10pt,raise=4pt},yshift=0pt]
(2,0) -- (-0.5,0) node [below=12,black,midway] {$\Q^n\times \Q^n$};
\end{scope}


\begin{scope}[shift={(0,-3.5)}]
\node[poento] (x) at (0,1.5) {$x$};
\node[poento] (y) at (1.5,1.5) {$y$};
\node[poento] (p) at (0,0) {$p$};
\node[poento] (q) at (1.5,0) {$q$};

\draw[tl] (x) to (p); 
\draw[tl] (x) to (q); 
\draw[ntl] (x) to (y); 
\draw[ntl] (y) to (p);
\draw[ntl] (y) to (q); 
\node at (2.25,0.75) {$\cong$};
\draw [decorate,decoration={brace,amplitude=10pt,raise=4pt},yshift=0pt]
(5,0) -- (-0.5,0) node [below=12,black,midway] {$\neq$};
\end{scope}

\begin{scope}[shift={(3,-3.5)}]
\node[poento] (x) at (0,1.5) {$x$};
\node[poento] (y) at (1.5,1.5) {$y$};
\node[poento] (p) at (0,0) {$p$};
\node[poento] (q) at (1.5,0) {$q$};

\draw[ntl] (x) to (p); 
\draw[ntl] (x) to (q); 
\draw[ntl] (x) to (y); 
\draw[tl] (y) to (p);
\draw[tl] (y) to (q); 
\end{scope}

\begin{scope}[shift={(6,-3.5)}]
\node[poento] (x) at (0,1.5) {$x$};
\node[poento] (y) at (1.5,1.5) {$y$};
\node[poento] (p) at (0,0) {$p$};
\node[poento] (q) at (1.5,0) {$q$};

\draw[tl] (x) to (p); 
\draw[ntl] (x) to (q); 
\draw[ntl] (x) to (y); 
\draw[tl] (y) to (p);
\draw[ntl] (y) to (q); 
\node at (2.25,0.75) {$\cong$};
\draw [decorate,decoration={brace,amplitude=10pt,raise=4pt},yshift=0pt]
(5,0) -- (-0.5,0) node [below=12,black,midway] {$\neq$};
\end{scope}

\begin{scope}[shift={(9,-3.5)}]
\node[poento] (x) at (0,1.5) {$x$};
\node[poento] (y) at (1.5,1.5) {$y$};
\node[poento] (p) at (0,0) {$p$};
\node[poento] (q) at (1.5,0) {$q$};

\draw[ntl] (x) to (p); 
\draw[tl] (x) to (q); 
\draw[ntl] (x) to (y); 
\draw[ntl] (y) to (p);
\draw[tl] (y) to (q); 
\end{scope}

\begin{scope}[shift={(12,-3.5)}]
\node[poento] (x) at (0,1.5) {$x$};
\node[poento] (y) at (1.5,1.5) {$y$};
\node[poento] (p) at (0,0) {$p$};
\node[poento] (q) at (1.5,0) {$q$};

\draw[tl] (x) to (p); 
\draw[ntl] (x) to (q); 
\draw[ntl] (x) to (y); 
\draw[ntl] (y) to (p);
\draw[tl] (y) to (q); 
\node at (2.25,0.75) {$\cong$};
\draw [decorate,decoration={brace,amplitude=10pt,raise=4pt},yshift=0pt]
(5,0) -- (-0.5,0) node [below=12,black,midway] {$\neq$};
\end{scope}

\begin{scope}[shift={(15,-3.5)}]
\node[poento] (x) at (0,1.5) {$x$};
\node[poento] (y) at (1.5,1.5) {$y$};
\node[poento] (p) at (0,0) {$p$};
\node[poento] (q) at (1.5,0) {$q$};

\draw[ntl] (x) to (p); 
\draw[tl] (x) to (q); 
\draw[ntl] (x) to (y); 
\draw[tl] (y) to (p);
\draw[ntl] (y) to (q); 
\end{scope}


\begin{scope}[shift={(3,-7)}]
\node[poento] (x) at (0,1.5) {$x$};
\node[poento] (y) at (1.5,1.5) {$y$};
\node[poento] (p) at (0,0) {$p$};
\node[poento] (q) at (1.5,0) {$q$};

\draw[tl] (x) to (p); 
\draw[ntl] (x) to (q); 
\draw[tl] (x) to (y); 
\draw[ntl] (y) to (p);
\draw[ntl] (y) to (q); 
\node at (2.25,0.75) {$\cong$};
\draw [decorate,decoration={brace,amplitude=10pt,raise=4pt},yshift=0pt]
(11,0) -- (-0.5,0) node [below=12,black,midway] {$\neq$};
\end{scope}

\begin{scope}[shift={(6,-7)}]
\node[poento] (x) at (0,1.5) {$x$};
\node[poento] (y) at (1.5,1.5) {$y$};
\node[poento] (p) at (0,0) {$p$};
\node[poento] (q) at (1.5,0) {$q$};

\draw[ntl] (x) to (p); 
\draw[ntl] (x) to (q); 
\draw[tl] (x) to (y); 
\draw[tl] (y) to (p);
\draw[ntl] (y) to (q); 
\node at (2.25,0.75) {$\cong$};
\end{scope}

\begin{scope}[shift={(9,-7)}]
\node[poento] (x) at (0,1.5) {$x$};
\node[poento] (y) at (1.5,1.5) {$y$};
\node[poento] (p) at (0,0) {$p$};
\node[poento] (q) at (1.5,0) {$q$};

\draw[ntl] (x) to (p); 
\draw[tl] (x) to (q); 
\draw[tl] (x) to (y); 
\draw[ntl] (y) to (p);
\draw[ntl] (y) to (q); 
\node at (2.25,0.75) {$\cong$};
\end{scope}

\begin{scope}[shift={(12,-7)}]
\node[poento] (x) at (0,1.5) {$x$};
\node[poento] (y) at (1.5,1.5) {$y$};
\node[poento] (p) at (0,0) {$p$};
\node[poento] (q) at (1.5,0) {$q$};

\draw[ntl] (x) to (p); 
\draw[ntl] (x) to (q); 
\draw[tl] (x) to (y); 
\draw[ntl] (y) to (p);
\draw[tl] (y) to (q); 
\end{scope}


\draw[tl] (0.4,-7) to (0.4,-9) node[below,black]{$\tleq$};
\draw[ntl] (1,-7) to (1,-9) node[below right, xshift=-6, black]{$\tlneq_{\tinyneq}$};


\begin{scope}[shift={(3,-10.5)}]
\node[poento] (x) at (0,1.5) {$x$};
\node[poento] (y) at (1.5,1.5) {$y$};
\node[poento] (p) at (0,0) {$p$};
\node[poento] (q) at (1.5,0) {$q$};

\draw[ntl] (x) to (p); 
\draw[tl] (x) to (q); 
\draw[ntl] (x) to (y); 
\draw[tl] (y) to (p);
\draw[tl] (y) to (q); 
\node at (2.25,0.75) {$\cong$};
\draw [decorate,decoration={brace,amplitude=10pt,raise=4pt},yshift=0pt]
(11,0) -- (-0.5,0) node [below=12,black,midway] {$\neq$};
\end{scope}

\begin{scope}[shift={(6,-10.5)}]
\node[poento] (x) at (0,1.5) {$x$};
\node[poento] (y) at (1.5,1.5) {$y$};
\node[poento] (p) at (0,0) {$p$};
\node[poento] (q) at (1.5,0) {$q$};

\draw[tl] (x) to (p); 
\draw[tl] (x) to (q); 
\draw[ntl] (x) to (y); 
\draw[ntl] (y) to (p);
\draw[tl] (y) to (q); 
\node at (2.25,0.75) {$\cong$};
\end{scope}

\begin{scope}[shift={(9,-10.5)}]
\node[poento] (x) at (0,1.5) {$x$};
\node[poento] (y) at (1.5,1.5) {$y$};
\node[poento] (p) at (0,0) {$p$};
\node[poento] (q) at (1.5,0) {$q$};

\draw[tl] (x) to (p); 
\draw[ntl] (x) to (q); 
\draw[ntl] (x) to (y); 
\draw[tl] (y) to (p);
\draw[tl] (y) to (q); 
\node at (2.25,0.75) {$\cong$};
\end{scope}

\begin{scope}[shift={(12,-10.5)}]
\node[poento] (x) at (0,1.5) {$x$};
\node[poento] (y) at (1.5,1.5) {$y$};
\node[poento] (p) at (0,0) {$p$};
\node[poento] (q) at (1.5,0) {$q$};

\draw[tl] (x) to (p); 
\draw[tl] (x) to (q); 
\draw[ntl] (x) to (y); 
\draw[tl] (y) to (p);
\draw[ntl] (y) to (q); 
\end{scope}


\begin{scope}[shift={(0,-14)}]
\node[poento] (x) at (0,1.5) {$x$};
\node[poento] (y) at (1.5,1.5) {$y$};
\node[poento] (p) at (0,0) {$p$};
\node[poento] (q) at (1.5,0) {$q$};

\draw[ntl] (x) to (p); 
\draw[ntl] (x) to (q); 
\draw[tl] (x) to (y); 
\draw[tl] (y) to (p);
\draw[tl] (y) to (q); 
\node at (2.25,0.75) {$\cong$};
\draw [decorate,decoration={brace,amplitude=10pt,raise=4pt},yshift=0pt]
(5,0) -- (-0.5,0) node [below=12,black,midway] {$\neq$};
\end{scope}

\begin{scope}[shift={(3,-14)}]
\node[poento] (x) at (0,1.5) {$x$};
\node[poento] (y) at (1.5,1.5) {$y$};
\node[poento] (p) at (0,0) {$p$};
\node[poento] (q) at (1.5,0) {$q$};

\draw[tl] (x) to (p); 
\draw[tl] (x) to (q); 
\draw[tl] (x) to (y); 
\draw[ntl] (y) to (p);
\draw[ntl] (y) to (q); 
\end{scope}

\begin{scope}[shift={(6,-14)}]
\node[poento] (x) at (0,1.5) {$x$};
\node[poento] (y) at (1.5,1.5) {$y$};
\node[poento] (p) at (0,0) {$p$};
\node[poento] (q) at (1.5,0) {$q$};

\draw[ntl] (x) to (p); 
\draw[tl] (x) to (q); 
\draw[tl] (x) to (y); 
\draw[ntl] (y) to (p);
\draw[tl] (y) to (q); 
\node at (2.25,0.75) {$\cong$};
\draw [decorate,decoration={brace,amplitude=10pt,raise=4pt},yshift=0pt]
(5,0) -- (-0.5,0) node [below=12,black,midway] {$\neq$};
\end{scope}

\begin{scope}[shift={(9,-14)}]
\node[poento] (x) at (0,1.5) {$x$};
\node[poento] (y) at (1.5,1.5) {$y$};
\node[poento] (p) at (0,0) {$p$};
\node[poento] (q) at (1.5,0) {$q$};

\draw[tl] (x) to (p); 
\draw[ntl] (x) to (q); 
\draw[tl] (x) to (y); 
\draw[tl] (y) to (p);
\draw[ntl] (y) to (q); 
\end{scope}

\begin{scope}[shift={(12,-14)}]
\node[poento] (x) at (0,1.5) {$x$};
\node[poento] (y) at (1.5,1.5) {$y$};
\node[poento] (p) at (0,0) {$p$};
\node[poento] (q) at (1.5,0) {$q$};

\draw[ntl] (x) to (p); 
\draw[tl] (x) to (q); 
\draw[tl] (x) to (y); 
\draw[tl] (y) to (p);
\draw[ntl] (y) to (q); 
\node at (2.25,0.75) {$\cong$};
\draw [decorate,decoration={brace,amplitude=10pt,raise=4pt},yshift=0pt]
(5,0) -- (-0.5,0) node [below=12,black,midway] {$\neq$};
\end{scope}

\begin{scope}[shift={(15,-14)}]
\node[poento] (x) at (0,1.5) {$x$};
\node[poento] (y) at (1.5,1.5) {$y$};
\node[poento] (p) at (0,0) {$p$};
\node[poento] (q) at (1.5,0) {$q$};

\draw[tl] (x) to (p); 
\draw[ntl] (x) to (q); 
\draw[tl] (x) to (y); 
\draw[ntl] (y) to (p);
\draw[tl] (y) to (q); 
\end{scope}


\begin{scope}[shift={(0,-17.5)}]
\node[poento] (x) at (0,1.5) {$x$};
\node[poento] (y) at (1.5,1.5) {$y$};
\node[poento] (p) at (0,0) {$p$};
\node[poento] (q) at (1.5,0) {$q$};

\draw[tl] (x) to (p); 
\draw[tl] (x) to (q); 
\draw[tl] (x) to (y); 
\draw[tl] (y) to (p);
\draw[tl] (y) to (q); 

\draw [decorate,decoration={brace,amplitude=10pt,raise=4pt},yshift=0pt]
(2,0) -- (-0.5,0) node [below=12,black,midway] {$\Q^n\times \Q^n$};
\end{scope}

\begin{scope}[shift={(3,-17.5)}]
\node[poento] (x) at (0,1.5) {$x$};
\node[poento] (y) at (1.5,1.5) {$y$};
\node[poento] (p) at (0,0) {$p$};
\node[poento] (q) at (1.5,0) {$q$};

\draw[ntl] (x) to (p); 
\draw[tl] (x) to (q); 
\draw[tl] (x) to (y); 
\draw[tl] (y) to (p);
\draw[tl] (y) to (q); 
\node at (2.25,0.75) {$\cong$};
\draw [decorate,decoration={brace,amplitude=10pt,raise=4pt},yshift=0pt]
(11,0) -- (-0.5,0) node [below=12,black,midway] {$\neq$};
\end{scope}

\begin{scope}[shift={(6,-17.5)}]
\node[poento] (x) at (0,1.5) {$x$};
\node[poento] (y) at (1.5,1.5) {$y$};
\node[poento] (p) at (0,0) {$p$};
\node[poento] (q) at (1.5,0) {$q$};

\draw[tl] (x) to (p); 
\draw[tl] (x) to (q); 
\draw[tl] (x) to (y); 
\draw[ntl] (y) to (p);
\draw[tl] (y) to (q); 

\node at (2.25,0.75) {$\cong$};
\end{scope}

\begin{scope}[shift={(9,-17.5)}]
\node[poento] (x) at (0,1.5) {$x$};
\node[poento] (y) at (1.5,1.5) {$y$};
\node[poento] (p) at (0,0) {$p$};
\node[poento] (q) at (1.5,0) {$q$};

\draw[tl] (x) to (p); 
\draw[tl] (x) to (q); 
\draw[tl] (x) to (y); 
\draw[tl] (y) to (p);
\draw[ntl] (y) to (q); 

\node at (2.25,0.75) {$\cong$};
\end{scope}

\begin{scope}[shift={(12,-17.5)}]
\node[poento] (x) at (0,1.5) {$x$};
\node[poento] (y) at (1.5,1.5) {$y$};
\node[poento] (p) at (0,0) {$p$};
\node[poento] (q) at (1.5,0) {$q$};

\draw[tl] (x) to (p); 
\draw[ntl] (x) to (q); 
\draw[tl] (x) to (y); 
\draw[tl] (y) to (p);
\draw[tl] (y) to (q); 
\end{scope}

\begin{scope}[shift={(15,-17.5)}]
\node[poento] (x) at (0,1.5) {$x$};
\node[poento] (y) at (1.5,1.5) {$y$};
\node[poento] (p) at (0,0) {$p$};
\node[poento] (q) at (1.5,0) {$q$};

\draw[tl] (x) to (p); 
\draw[tl] (x) to (q); 
\draw[ntl] (x) to (y); 
\draw[tl] (y) to (p);
\draw[tl] (y) to (q); 
\draw [decorate,decoration={brace,amplitude=10pt,raise=4pt},yshift=0pt]
(2,0) -- (-0.5,0) node [below=12,black,midway] {$\Q^n\times \Q^n$};
\end{scope}

\end{tikzpicture}

%% file: Embeddings.tikz
\begin{tikzpicture}[scale=0.9]
\tikzstyle{poento}=[circle, draw, fill=white, inner sep=0,minimum size=10];
\tikzstyle{tl}=[blue,very thick]
\tikzstyle{ll}=[very thick, red, dashed]  
\tikzstyle{ntl}=[very thick, brown, dashdotted] 
\tikzstyle{sl}=[very thick, green!82!black, dotted] 

\begin{scope}[shift={(0,2)}]
\draw[tl] (-1.8,-1.5) to (-0.7,-1.5) node[right] {$\tleq$};
\draw[ntl] (0.2,-1.5) to (1.3,-1.5) node[right] {$\tlneq_{\tinyneq}$};
\draw[ll] (0.2,-2) to (1.3,-2) node[right] {$\lleq$};
\draw[sl] (0.2,-2.5) to (1.3,-2.5) node[right] {$\sleq$};
\end{scope}


\begin{scope}[shift={(6,0)}]
\begin{scope}[shift={(-2.5,0.5)},scale=0.5]
\node[poento] (x) at (0,1.5) {$x$};
\node[poento] (y) at (1.5,1.5) {$y$};
\node[poento] (p) at (0,0) {$p$};
\node[poento] (q) at (1.5,0) {$q$};

\draw[ntl] (x) to (p); 
\draw[ntl] (x) to (q); 
\draw[ntl] (x) to (y); 
\draw[ntl] (y) to (p);
\draw[ntl] (y) to (q); 
\end{scope}

\draw[dotted]  (-3,-1.75) rectangle (3,1.75);

\coordinate (vp) at (0.5,-1) ;
\coordinate (vq) at ([shift={(1.5,1.5)}]vp);
\coordinate (vx) at (-2,-1);
\coordinate (vy) at (-1,-1);

\draw[ll,thin] ([shift={(2,2)}]vy) to (vy);
\draw[ll] (vp) to (vq);
\draw[sl] (vx) to (vq);
\draw[sl] (vx) to (vy);
\draw[sl] (vp) to (vy) to (vq);
\draw[sl, ->] (vq) to ([shift={(0,-0.4)}]vq);
\draw[tl, ->] (vq) to ([shift={(0,0.4)}]vq);

\draw[fill] (vp) circle [radius=.07] node[below right] {$\vp$};
\draw[fill] (vq) circle [radius=.07] node[right] {$\vq$};
\draw[fill] (vx) circle [radius=.07] node[below left] {$\vx$};
\draw[fill] (vy) circle [radius=.07] node[below] {$\vy$};
\end{scope}


\begin{scope}[shift={(0,-3.5)}]
\begin{scope}[shift={(-2.5,0.5)},scale=0.5]
\node[poento] (x) at (0,1.5) {$x$};
\node[poento] (y) at (1.5,1.5) {$y$};
\node[poento] (p) at (0,0) {$p$};
\node[poento] (q) at (1.5,0) {$q$};

\draw[tl] (x) to (p); 
\draw[ntl] (x) to (q); 
\draw[ntl] (x) to (y); 
\draw[ntl] (y) to (p);
\draw[ntl] (y) to (q); 
\end{scope}

\draw[dotted]  (-3,-1.75) rectangle (3,1.75);

\coordinate (vp) at (0,-1) ;
\coordinate (vq) at ([shift={(1.5,1.5)}]vp);
\coordinate (vx) at (0,0);
\coordinate (vy) at (-1.5,-0.5);

\draw[ll,thin] ([shift={(1,1)}]vx) to (vx) to  ([shift={(1,-1)}]vx)   (vx) to  ([shift={(-1,-1)}]vx) ;
\draw[ll] (vp) to (vq);
\draw[tl] (vx) to (vp);
\draw[sl] (vx) to (vq);
\draw[sl] (vp) to (vy) to (vx);
\draw[sl, ->] (vq) to ([shift={(0,-0.4)}]vq);
\draw[tl, ->] (vq) to ([shift={(0,0.4)}]vq);

\draw[fill] (vp) circle [radius=.07] node[below right] {$\vp$};
\draw[fill] (vq) circle [radius=.07] node[right] {$\vq$};
\draw[fill] (vx) circle [radius=.07] node[above left] {$\vx$};
\draw[fill] (vy) circle [radius=.07] node[left] {$\vy$};
\end{scope}


\begin{scope}[shift={(6,-3.5)}]
\begin{scope}[shift={(-2.5,0.5)},scale=0.5]
\node[poento] (x) at (0,1.5) {$x$};
\node[poento] (y) at (1.5,1.5) {$y$};
\node[poento] (p) at (0,0) {$p$};
\node[poento] (q) at (1.5,0) {$q$};

\draw[ntl] (x) to (p); 
\draw[ntl] (x) to (q); 
\draw[tl] (x) to (y); 
\draw[ntl] (y) to (p);
\draw[ntl] (y) to (q); 
\end{scope}

\draw[dotted]  (-3,-1.75) rectangle (3,1.75);

\coordinate (vp) at (1,-1) ;
\coordinate (vq) at ([shift={(1,1)}]vp);
\coordinate (vx) at (-1,0);
\coordinate (vy) at (-1,-1);

\draw[ll,thin] ([shift={(2,2)}]vy) to (vy);
\draw[ll,thin] ([shift={(-2,2)}]vp) to (vp);
\draw[ll] (vp) to (vq);
\draw[sl] (vp) to (vx) to (vq);
\draw[tl] (vx) to (vy);
\draw[sl] (vp) to (vy) to (vq);
\draw[sl, ->] (vq) to ([shift={(0,-0.4)}]vq);
\draw[tl, ->] (vq) to ([shift={(0,0.4)}]vq);

\draw[fill] (vp) circle [radius=.07] node[below right] {$\vp$};
\draw[fill] (vq) circle [radius=.07] node[right] {$\vq$};
\draw[fill] (vx) circle [radius=.07] node[left] {$\vx$};
\draw[fill] (vy) circle [radius=.07] node[below left] {$\vy$};
\end{scope}


\begin{scope}[shift={(0,-7)}]
\begin{scope}[shift={(-2.5,0.5)},scale=0.5]
\node[poento] (x) at (0,1.5) {$x$};
\node[poento] (y) at (1.5,1.5) {$y$};
\node[poento] (p) at (0,0) {$p$};
\node[poento] (q) at (1.5,0) {$q$};

\draw[tl] (x) to (p); 
\draw[tl] (x) to (q); 
\draw[ntl] (x) to (y); 
\draw[ntl] (y) to (p);
\draw[ntl] (y) to (q); 
\end{scope}

\draw[dotted]  (-3,-1.75) rectangle (3,1.75);

\coordinate (vp) at (1,0) ;
\coordinate (vq) at ([shift={(1,1)}]vp);
\coordinate (vx) at ([shift={(0,-1)}]vp);
\coordinate (vy) at (-1.5,-1);

\draw[ll,thin] ([shift={(2,2)}]vy) to (vy);
\draw[ll,thin] ([shift={(-1,1)}]vp) to (vp);
\draw[ll,thin] ([shift={(-1.5,1.5)}]vx) to (vx) to  ([shift={(1,1)}]vx);
\draw[ll] (vp) to (vq);
\draw[tl] (vp) to (vx) to (vq);
\draw[sl] (vx) to (vy);
\draw[sl] (vp) to (vy) to (vq);
\draw[sl, ->] (vq) to ([shift={(0,-0.4)}]vq);
\draw[tl, ->] (vq) to ([shift={(0,0.4)}]vq);

\draw[fill] (vp) circle [radius=.07] node[right,yshift=-2] {$\vp$};
\draw[fill] (vq) circle [radius=.07] node[right] {$\vq$};
\draw[fill] (vx) circle [radius=.07] node[below right] {$\vx$};
\draw[fill] (vy) circle [radius=.07] node[below left] {$\vy$};
\end{scope}


\begin{scope}[shift={(6,-7)}]
\begin{scope}[shift={(-2.5,0.5)},scale=0.5]
\node[poento] (x) at (0,1.5) {$x$};
\node[poento] (y) at (1.5,1.5) {$y$};
\node[poento] (p) at (0,0) {$p$};
\node[poento] (q) at (1.5,0) {$q$};

\draw[tl] (x) to (p); 
\draw[ntl] (x) to (q); 
\draw[ntl] (x) to (y); 
\draw[tl] (y) to (p);
\draw[ntl] (y) to (q); 
\end{scope}

\draw[dotted]  (-3,-1.75) rectangle (3,1.75);

\coordinate (vp) at (0,-1) ;
\coordinate (vq) at ([shift={(2,2)}]vp);
\coordinate (vx) at (0.5,1);
\coordinate (vy) at (-0.5,1);

\draw[ll,thin] ([shift={(-1,1)}]vp) to (vp);
\draw[ll] (vp) to (vq);
\draw[tl] (vx) to (vp) to (vy);
\draw[sl] (vy) to (vq);
\draw[sl, ->] (vq) to ([shift={(0,-0.4)}]vq);
\draw[tl, ->] (vq) to ([shift={(0,0.4)}]vq);

\draw[fill] (vp) circle [radius=.07] node[below right] {$\vp$};
\draw[fill] (vq) circle [radius=.07] node[right] {$\vq$};
\draw[fill] (vx) circle [radius=.07] node[above] {$\vx$};
\draw[fill] (vy) circle [radius=.07] node[above left] {$\vy$};
\end{scope}


\begin{scope}[shift={(0,-10.5)}]
\begin{scope}[shift={(-2.5,0.5)},scale=0.5]
\node[poento] (x) at (0,1.5) {$x$};
\node[poento] (y) at (1.5,1.5) {$y$};
\node[poento] (p) at (0,0) {$p$};
\node[poento] (q) at (1.5,0) {$q$};

\draw[tl] (x) to (p); 
\draw[ntl] (x) to (q); 
\draw[ntl] (x) to (y); 
\draw[ntl] (y) to (p);
\draw[tl] (y) to (q); 
\end{scope}

\draw[dotted]  (-3,-1.75) rectangle (3,1.75);

\coordinate (vp) at (0,-1) ;
\coordinate (vq) at ([shift={(2,2)}]vp);
\coordinate (vx) at ([shift={(0,1)}]vp);
\coordinate (vy) at ([shift={(-0,-1)}]vq);

\draw[ll] (vp) to (vq);
\draw[tl] (vx) to (vp) ([shift={(0,-0.4)}]vq) to (vy);
\draw[sl] (vp) to (vy) to (vx) to (vq);
\draw[sl, ->] (vq) to ([shift={(0,-0.4)}]vq);
\draw[tl, ->] (vq) to ([shift={(0,0.4)}]vq);

\draw[fill] (vp) circle [radius=.07] node[below left] {$\vp$};
\draw[fill] (vq) circle [radius=.07] node[right] {$\vq$};
\draw[fill] (vx) circle [radius=.07] node[above left] {$\vx$};
\draw[fill] (vy) circle [radius=.07] node[below right] {$\vy$};
\end{scope}


\begin{scope}[shift={(6,-10.5)}]
\begin{scope}[shift={(-2.5,0.5)},scale=0.5]
\node[poento] (x) at (0,1.5) {$x$};
\node[poento] (y) at (1.5,1.5) {$y$};
\node[poento] (p) at (0,0) {$p$};
\node[poento] (q) at (1.5,0) {$q$};

\draw[tl] (x) to (p); 
\draw[ntl] (x) to (q); 
\draw[tl] (x) to (y); 
\draw[ntl] (y) to (p);
\draw[ntl] (y) to (q); 
\end{scope}

\draw[dotted]  (-3,-1.75) rectangle (3,1.75);

\coordinate (vp) at (0,-1) ;
\coordinate (vq) at ([shift={(2,2)}]vp);
\coordinate (vx) at ([shift={(-.5,2)}]vp);
\coordinate (vy) at ([shift={(-0.5,-2)}]vx);

\draw[ll,thin] ([shift={(-1,1)}]vp) to (vp);
\draw[ll] (vp) to (vq);
\draw[tl] (vy) to (vx) to (vp);
\draw[sl] (vp) to (vy) to (vq) to (vx);
\draw[sl, ->] (vq) to ([shift={(0,-0.4)}]vq);
\draw[tl, ->] (vq) to ([shift={(0,0.4)}]vq);

\draw[fill] (vp) circle [radius=.07] node[below right] {$\vp$};
\draw[fill] (vq) circle [radius=.07] node[right] {$\vq$};
\draw[fill] (vx) circle [radius=.07] node[above left] {$\vx$};
\draw[fill] (vy) circle [radius=.07] node[below left] {$\vy$};
\end{scope}

\end{tikzpicture}

%% file: neweTtoS.tikz
\begin{tikzpicture}[scale=0.85]
\tikzstyle{poento}=[circle, draw, inner sep=2, fill=white];
\tikzstyle{linio}=[very thick];

\tikzstyle{tl}=[blue,very thick]
\tikzstyle{ll}=[very thick, red, dashed]  
\tikzstyle{ntl}=[very thick, brown, dashdotted] 
\tikzstyle{sl}=[very thick, green!82!black, dotted]

\begin{scope}[shift={(3,1.2)},scale=0.45]
\node at (-6.5,1.5) {\Huge$\newettos\;=$};
\node (E) at (-3,3) {\Huge$\exists$};
\node[poento] (y) at (0,1) {$y$};
\node[poento] (x)  at (-1,3) {$x$};
\node[poento] (z) at (1,3) {$z$};
\node[poento] (p) at (-2,0) {$p$};
\node[poento] (q) at (2,0) {$q$};

\draw[ntl] (p) to  (z) to (x)  to (q);
\draw[ntl] (p) to (y) to (q) ;
\draw[tl] (p) to (x) to (y) to (z) to (q);
\end{scope}

\begin{scope}[shift={(0,-3.7)},scale=0.45]
\node (E) at (-3,1.8) {\Huge$\nexists$};
\node[poento] (y) at (0,1) {$y$};
\node[poento] (x)  at (-1,3) {$x$};
\node[poento] (z) at (1,3) {$z$};
\node[poento] (p) at (-2,0) {$p$};
\node[poento] (q) at (2,0) {$q$};

\draw[ntl] (p) to  (z) to (x)  to (q);
\draw[ntl] (p) to (y) to (q) ;
\draw[tl] (p) to (x) to (y) to (z) to (q);
\draw[ll] (p) to (q);
\end{scope}

\begin{scope}[shift={(3,-3.7)},scale=0.45]
\node (E) at (-3,1.8) {\Huge$\nexists$};
\node[poento] (y) at (0,1) {$y$};
\node[poento] (x)  at (-1,3) {$x$};
\node[poento] (z) at (1,3) {$z$};
\node[poento] (p) at (-2,0) {$p$};
\node[poento] (q) at (2,0) {$q$};

\draw[ntl] (p) to  (z) to (x)  to (q);
\draw[ntl] (p) to (y) to (q) ;
\draw[tl] (p) to (x) to (y) to (z) to (q);
\draw[tl] (p) to (q);
\end{scope}

\begin{scope}[shift={(1.5,1.5)}]
\draw[tl] (-1.8,-1.5) to (-0.7,-1.5) node[right] {$\tleq$};
\draw[ntl] (0.2,-1.5) to (1.3,-1.5) node[right] {$\tlneq_{\tinyneq}$};
\draw[ll] (0.2,-2) to (1.3,-2) node[right] {$\lleq$};
\draw[sl] (0.2,-2.5) to (1.3,-2.5) node[right] {$\sleq$};
\end{scope}

\begin{scope}[shift={(8.5,-3)},scale=1.2]
\draw[dotted] (-3,-1) rectangle (3,5);
\coordinate (vp) at (-2,0);
\coordinate (vq) at (2,-0);
\coordinate (vy) at (0,0); 
\coordinate (vz) at (2,3);
\coordinate (vx) at (-2,3);

\node[below left] at (vp) {$\vp$};
\node[below right] at (vq) {$\vq$};
\node[above left] at (vx) {$\vx$};
\node[below=2] at (vy) {$\vy$};
\node[above right] at (vz) {$\vz$};

\draw[tl] (vp) to (vx) to (vy) to (vz) to (vq);
\draw[sl] (vp) to (vy) to (vq) to (vx) to (vz) to (vp); 

\draw[ll,thick] ([shift={(-0.5,0.5)}]vp) to (vp) to ([shift={(4.5,4.5)}]vp); 
\draw[ll,thick] ([shift={(0.5,0.5)}]vq) to (vq) to ([shift={(-4.5,4.5)}]vq); 
\draw[ll,thick] ([shift={(-2.5,2.5)}]vy) to (vy) to ([shift={(2.5,2.5)}]vy);

\draw[fill] (vp) circle [radius=0.07];
\draw[fill] (vq) circle [radius=0.07];
\draw[fill] (vx) circle [radius=0.07];
\draw[fill] (vy) circle [radius=0.07];
\draw[fill] (vz) circle [radius=0.07];

\node[below] at (0,-1) {\txplane};

\end{scope}

\end{tikzpicture}

%% file: neweStoT.tikz
\begin{tikzpicture}[scale=.5]

\tikzstyle{tl}=[blue,very thick]
\tikzstyle{ll}=[very thick, red, dashed]  
\tikzstyle{ntl}=[very thick, brown, dashdotted] 
\tikzstyle{sl}=[very thick, green!82!black, dotted] 

\tikzstyle{poento}=[circle, draw, inner sep=2, fill=white];

\begin{scope}[shift={(7.5,0)}]
\draw[dotted]  (-7,-4.5) rectangle (4,4.5);
\draw[very thick, blue,]  (-2,0) to (0,3);

\draw[red] (-2,0) circle(4);
\draw[red] (0,0) circle (2);
\draw[ll] (0,0) to (2,0);
\draw[ll] (-2,0) to (-2,4);

\draw[fill] (-2,0) circle [radius=.04] node[above left] {$\vp$};
\draw[fill] (-2,0) circle [radius=.04] node[below left] {\tiny$(-2,-2,0)$};

\draw[fill] (0,0) circle [radius=.07] node[above] {$\vy$};
\draw[fill] (0,0) circle [radius=.07] node[below] {\tiny$(0,0,0)$};

\draw[fill] (2,0) circle [radius=.12] node[above right] {$\vq$};
\draw[fill] (2,0) circle [radius=.12] node[below right] {\tiny$(2,2,0)$};

\draw[fill] (0,3) circle [radius=.12] node[above left] {$\vz$};
\draw[fill] (0,3) circle [radius=.12] node[above right] {\tiny$(2,0,3)$};

\node[below] at (-1.5,-5) {top view};
\end{scope}

\begin{scope}[shift={(-6,-11)}]
\draw[dotted]  (-6,-4.5) rectangle (5,4.5);

\draw[ll] (-2,-2) to (2,2);
\draw[sl]  (0,-2) to (-2,-2) ;
\draw[sl]  (0,2) to (2,2) ;

\draw[fill] (-2,-2) circle [radius=.12] node[above left] {$\vp$};
\draw[fill] (-2,-2) circle [radius=.12] node[below left] {\tiny$(-2,-2,0)$};

\draw[fill] (0,0) circle [radius=.12] node[above] {$\vy$};
\draw[fill] (0,0) circle [radius=.12] node[below right] {\tiny$(0,0,0)$};

\draw[fill] (2,2) circle [radius=.12] node[above right] {$\vq$};
\draw[fill] (2,2) circle [radius=.12] node[below right] {\tiny$(2,2,0)$};

\draw[fill] (0,2) circle [radius=.07] node[above left] {$\vz$};
\draw[fill] (0,2) circle [radius=.07] node[below left] {\tiny$(2,0,3)$};

\draw[fill] (0,-2) circle [radius=.07] node[above right] {$\vx$};
\draw[fill] (0,-2) circle [radius=.07] node[below right, black] {\tiny$(-2,0,3)$};
\node[below] at (-1.5,-5) {front view};

\end{scope}

\begin{scope}[shift={(7.5,-11)}]
\draw[dotted]  (-7,-4.5) rectangle (4,4.5);

\draw[ll]  (0,-2) to  (-2,0) to (0,2);
\draw[tl]  (1,-2) to (1,2);
\draw[sl]  (1,-2) to (-2,0) to (1,2);

\draw[fill] (-2,-2) circle [radius=.04] node[above left] {$\vp$};
\draw[fill] (-2,-2) circle [radius=.04] node[below left] {\tiny$(-2,-2,0)$};

\draw[fill] (-2,0) circle [radius=.07] node[above left] {$\vy$};
\draw[fill] (-2,0) circle [radius=.07] node[below left] {\tiny$(0,0,0)$};

\draw[fill] (-2,2) circle [radius=.12] node[above left] {$\vq$};
\draw[fill] (-2,2) circle [radius=.12] node[below left] {\tiny$(2,2,0)$};

\draw[fill] (1,2) circle [radius=.07] node[above right] {$\vz$};
\draw[fill] (1,2) circle [radius=.07] node[below right] {\tiny$(2,0,3)$};

\draw[fill] (1,-2) circle [radius=.07] node[above right] {$\vx$};
\draw[fill] (1,-2) circle [radius=.07] node[below right, black] {\tiny$(-2,0,3)$};
\node[below] at (-1.5,-5) {side view};

\end{scope}

\begin{scope}[shift={(-6.5,0)}]
\node (E) at (-3,3) {\Huge$\exists$};
\node[poento] (y) at (0,1) {$y$};
\node[poento] (x)  at (-1,3) {$x$};
\node[poento] (z) at (1,3) {$z$};
\node[poento] (p) at (-2,0) {$p$};
\node[poento] (q) at (2,0) {$q$};

\draw[tl] (p) to  (z) to (x)  to (q);
\draw[ll] (p) to (y) to (q) to (p);
\draw[sl] (p) to (x) to (y) to (z) to (q);

\draw[tl] (-1,-2) to (1,-2) node[right] {$\tleq$};
\draw[ll] (-1,-3) to (1,-3) node[right] {$\lleq$};
\draw[sl] (-1,-4) to (1,-4) node[right] {$\sleq$};
\end{scope}

\end{tikzpicture}

%% file: HyperbInv.tikz
\begin{tikzpicture}[scale=0.75]

\def\r{0.07}

\pgfmathsetmacro{\e}{sqrt(2)}   
\pgfmathsetmacro{\a}{sqrt(2)}
\pgfmathsetmacro{\b}{(\a*sqrt((\e)^2-1)} 

\pgfmathsetmacro{\c}{1/sqrt(3)}

\begin{scope}[]
\draw[gray] (-3.36,0) to (3.36,0);
\draw[gray] (0,-3.36) to (0,3.36);
\draw[gray,thick, dashed] (-3.2,-3.2) to (3.2,3.2);
\draw[gray,thick, dashed] (-3.2,3.2) to (3.2,-3.2);

\fill[blue!42]   (0.1,\a) to  (0,\a) to(0,3.36) to (3.08,3.36) ;
\fill[blue!42]   plot[domain=0:1.51] ({\a*sinh(\x)},{\b*cosh(\x)});

\fill[red!42]  (0,\a) to (0,0) to (3.2,3.2) to (3.05,3.35) to (2,2.45) to (1,1.72) to (0.5,1.5);

\draw[thick] plot[domain=-1.51:1.51] ({\a*sinh(\x)},{\b*cosh(\x)});
\draw[thick] plot[domain=-1.51:1.51] ({\a*sinh(\x)},{-\b*cosh(\x)});
\draw[thick] plot[domain=-1.51:1.51] ({-\b*cosh(\x)},{\a*sinh(\x)});    
\draw[thick] plot[domain=-1.51:1.51] ({\b*cosh(\x)},{\a*sinh(\x)});    

\draw[dotted] (0,0) to (-1.68,-3.36);
\draw[] (0,0) to (1.68,3.36);
\draw[gray,fill=white] (0,0) circle [radius=\r];

\coordinate (a) at (\c*\a,2*\c*\a) ;

\draw[blue, very thick] (\a,0) to node[right]{$\tleq$} (\a,2*\a);
\draw[red, very thick] (-\a,0) to node[above]{$\sleq$} (\a/3,2*\a/3);

\draw[fill] (\a,2*\a) node[below right]{$\vq$} circle [radius=\r];
\draw[fill] (\a/3,2*\a/3) node[below right]{$h(\vq)$} circle [radius=\r];
\draw[fill] (\a,0) node[below right] {$\vp$} circle [radius=\r];
\draw[fill] (-\a,0) node[below left] {$h(\vp)$} circle [radius=\r];

\end{scope}

\begin{scope}[shift={(9,0)}]
\draw[gray] (-3.36,0) to (3.36,0);
\draw[gray] (0,-3.36) to (0,3.36);
\draw[gray,dashed,thick] (-3.2,-3.2) to (3.2,3.2);
\draw[gray, dashed,thick] (-3.2,3.2) to (3.2,-3.2);

\fill[red!42]  (\a,0.1) to (\a,0) to (3.36,0) to (3.36,3.08) --cycle ;
\fill[red!42]   plot[domain=0:1.51] ({\b*cosh(\x)},{\a*sinh(\x)});

\fill[blue!42]  (-\a,0) to (0,0) to (-3.2,-3.2) to (-3.35,-3.05) to (-2.45,-2) to (-1.72,-1) to (-1.5,-0.5);

\draw[thick] plot[domain=-1.51:1.51] ({\a*sinh(\x)},{\b*cosh(\x)});
\draw[thick] plot[domain=-1.51:1.51] ({\a*sinh(\x)},{-\b*cosh(\x)});
\draw[thick] plot[domain=-1.51:1.51] ({-\b*cosh(\x)},{\a*sinh(\x)});    
\draw[thick] plot[domain=-1.51:1.51] ({\b*cosh(\x)},{\a*sinh(\x)});    

\coordinate (a) at (2*\c*\a,\c*\a) ;
\coordinate (-a) at (-2*\c*\a,-\c*\a) ;

\draw[] (a) to (3.36,1.68);
\draw[] (-a) to (0,0);
\draw[dotted] (-3.36,-1.68) to (3.36,1.68);
\draw[gray,fill=white] (0,0) circle [radius=\r];

\draw[blue, very thick] (0,\a) to node[left]{$\tleq$}  (-2*\a/3,-\a/3);
\draw[red, very thick] (0,\a) to node[above]{$\sleq$}  (2*\a,\a);

\draw[fill] (2*\a,\a) node[below right]{$\vq$} circle [radius=\r];
\draw[fill] (-2*\a/3,-\a/3) node[below right]{$h(\vq)$} circle [radius=\r];
\draw[fill] (0,\a) node[below right] {$\vp=h(\vp)$} circle [radius=\r];

\end{scope}

\end{tikzpicture}

%% file: Winnie2.tikz
\begin{tikzpicture}[scale=0.7]

\pgfmathsetmacro{\conesize}{2}

\newcommand{\lcone}[6]{
\draw[red]   (#1-#5,#2+#5) to (#1,#2) to (#1+#5,#2+#5);
\fill[#4, opacity=.2] (#1,#2) to (#1-#5,#2+#5) to (#1+#5,#2+#5) to cycle ;

\draw[red, fill=#4, fill opacity=0.2] (#1+#5,#2+#5) arc [start angle=0,end angle=180,x radius=#5, y radius=.1*#5];
\draw[red] (#1-#5,#2+#5) arc [start angle=180,end angle=360,x radius=#5, y radius=.1*#5];

\draw[red]   (#1-#6,#2-#6) to (#1,#2) to (#1+#6,#2-#6);

\fill[#4, opacity=.2] (#1,#2) to (#1-#6,#2-#6) to (#1+#6,#2-#6) to cycle ;

\draw[dashed, red] (#1+#6,#2-#6) arc [start angle=0,end angle=180,x radius=#6, y radius=.1*#6];
\draw[red, fill=#4, fill opacity=0.2] (#1-#6,#2-#6) arc [start angle=180,end angle=360,x radius=#6, y radius=.1*#6];

\draw[fill] (#1,#2)  circle [radius=0.05] node[left]{#3};
}

\newcommand{\clcone}[7]{
\draw[#7]   (#1-#5,#2+#5) to (#1,#2) to (#1+#5,#2+#5);
\fill[#4, opacity=.2] (#1,#2) to (#1-#5,#2+#5) to (#1+#5,#2+#5) to cycle ;

\draw[#7, fill=#4, fill opacity=0.2] (#1+#5,#2+#5) arc [start angle=0,end angle=180,x radius=#5, y radius=.1*#5];
\draw[#7] (#1-#5,#2+#5) arc [start angle=180,end angle=360,x radius=#5, y radius=.1*#5];

\draw[#7]   (#1-#6,#2-#6) to (#1,#2) to (#1+#6,#2-#6);

\fill[#4, opacity=.2] (#1,#2) to (#1-#6,#2-#6) to (#1+#6,#2-#6) to cycle ;

\draw[dashed, #7] (#1+#6,#2-#6) arc [start angle=0,end angle=180,x radius=#6, y radius=.1*#6];
\draw[#7, fill=#4, fill opacity=0.2] (#1-#6,#2-#6) arc [start angle=180,end angle=360,x radius=#6, y radius=.1*#6];

\draw[fill] (#1,#2)  circle [radius=0.05] node[left]{#3};
}

\begin{scope}[shift={(-10,-.5)}]


\lcone{0}{0}{$\vx$}{white}{2}{1.45}
\draw[gray, fill opacity=0.4, fill=brown] (-2,2.5) to (5.5,2.5)   to (4.5,1.5)  to (-3,1.5) to cycle;

\draw[gray, fill opacity=0.4, fill=brown] (-2,0.5) to (5.5,0.5)   to (4.5,-0.5)  to (-3,-0.5) to cycle;

\draw[gray, fill=white] (\conesize,\conesize) arc [start angle=0,end angle=360,x radius=\conesize, y radius=.1*\conesize];
\lcone{2.05}{2.05}{$\vy$}{white}{0.55}{2}
\draw[gray, fill=white] (2.05+\conesize,2.05-\conesize) arc [start angle=0,end angle=360,x radius=\conesize, y radius=.1*\conesize];

\draw[black, fill=red, opacity=0.4] (4,2.05) arc [start angle=0,end angle=360,x radius=1.5, y radius=.15];

\draw[black, fill=red, opacity=0.4] (3,.05) arc [start angle=0,end angle=360,x radius=.5, y radius=.05];

\lcone{2.5}{.55}{$\vu$}{red}{2}{2}

\draw[fill] (3,2.05)  circle [radius=0.05] node[right]{$\vz_u$};

\end{scope}

\draw[dotted, gray] (-3.7,-2.5) to (-3.7,2.5);

\lcone{-1}{0}{$\vu$}{white}{2}{2}
\node[right] at (1,0) {$\vv$};
\clcone{0}{-1}{$\vx$}{blue}{3}{0}{black}
\clcone{0}{1}{$\vy$}{blue}{0}{3}{black}

\draw[white, fill=white] (0,.98) to (0.97,0) to (0,-.97) to (-0.98,0);
\draw[black, fill=blue, opacity=0.2] (0,1) to (1,0) to (0,-1) to (-1,0) to (0,1);

\draw[black,dashed] (1,0) arc [start angle=0,end angle=180,x radius=1, y radius=.1];
\draw[black] (1,0) arc [start angle=0,end angle=-180,x radius=1, y radius=.1];

\draw[fill] (0,1)  circle [radius=0.05];
\draw[fill] (0,-1)  circle [radius=0.05];
\draw[fill] (1,0)  circle [radius=0.05];
\draw[fill] (-1,0)  circle [radius=0.05];

\end{tikzpicture}